\documentclass[a4paper,reqno]{amsart}
\usepackage[T1]{fontenc}
\usepackage[utf8x]{inputenc}
\usepackage[english]{babel}
\usepackage{amsmath,amsfonts,amssymb,upgreek,comment}
\usepackage{graphicx}
\usepackage{enumerate}
\usepackage{xcolor}
\usepackage{esint}

\usepackage{amsthm}
\usepackage{hyperref}
\hypersetup{bookmarks=true}

\theoremstyle{plain}
\newtheorem*{theorem*}{Theorem}
\newtheorem{theorem}{Theorem}[section]
\newtheorem{thmA}{Theorem}

\newtheorem{chara}{Characterization}
\newtheorem{D}[theorem]{Definition}
\newtheorem{lemma}[theorem]{Lemma}
\newtheorem{cor}[theorem]{Corollary}
\newtheorem{prop}[theorem]{Proposition}

\theoremstyle{definition}
\newtheorem{rem}[theorem]{Remark}
\newtheorem{rems}[theorem]{Remarks}

\newcommand{\R}{\ensuremath{\mathbb R}}
\newcommand{\N}{\ensuremath{\mathbb N}}

\newcommand{\eps}{\ensuremath{\varepsilon}}

\newcommand{\Ricm}{\ensuremath{\mbox{Ric}_{\mbox{\tiny{-}}}}}
\newcommand{\kato}{\ensuremath{\mbox{k}_t(M^n,g)}}

\DeclareMathOperator{\PI}{PI}
\DeclareMathOperator{\Kato}{k}

\newcommand{\setN}{\mathbb{N}}

\newcommand{\setR}{\mathbb{R}}

\newcommand{\cD}{\mathcal{D}}

\newcommand{\cE}{\mathcal{E}}
\newcommand{\cV}{\mathcal{V}}

\newcommand{\mE}{\mathrm{E}}

\newcommand{\weakto}{\rightharpoonup}

\newcommand{\di}{\mathop{}\!\mathrm{d}}

\newcommand{\restr}{\raisebox{-.1618ex}{$\bigr\rvert$}}
\DeclareMathOperator{\supp}{supp}

\newcommand{\Ch}{{\sf Ch}}
\newcommand{\scal}[2]{\ensuremath{\langle #1 , #2 \rangle}} 

\DeclareMathOperator{\Lip}{Lip}

\newcommand{\measrestr}{%
  \,\raisebox{-.127ex}{\reflectbox{\rotatebox[origin=br]{-90}{$\lnot$}}}\,%
}

\newcommand{\dist}{\mathsf{d}}

\newcommand{\meas}{\mu}

\newcommand{\diam}{\mathrm{diam}}
\DeclareMathOperator{\CD}{CD}
\DeclareMathOperator{\RCD}{RCD}

\newcommand{\vol}{v}

\DeclareMathOperator{\Ric}{Ric}

\newfont{\tmpf}{cmsy10 scaled 2500}

\def\R{\mathbb R}\def\N{\mathbb N}

\def\cC{\mathcal C}
\def\cD{\mathcal D}

\def\cH{\mathcal H}

\def\cM{\mathcal M}

\def\cS{\mathcal S}\def\cE{\mathcal E}
\def\cT{\mathcal T}
\def\cV{\mathcal V}

\def\ra{\rangle}
\def\la{\langle}

\newcommand{\Rad}{\mathrm{Rad}}

\title[Structure of Kato limits]{Limits of manifolds with a Kato bound on the Ricci curvature}

\author{Gilles Carron}
\address{G. Carron, Laboratoire de Math\'ematiques Jean Leray, UMR CNRS 6629, Universit\'e de Nantes,
2, rue de la Houssini\`ere, B.P.~92208, 44322 Nantes Cedex~3, France.} 
\email{Gilles.Carron@univ-nantes.fr}

\author{Ilaria Mondello}
\address{I. Mondello, Université Paris Est Cr\'eteil, Laboratoire d'Analyse et Math\'ematiques appliqu\'es, UMR CNRS 8050, F-94010 Creteil, France.}
\email{ilaria.mondello@u-pec.fr}

\author{David Tewodrose}
\address{D. Tewodrose, Université Libre de Bruxelles, Service d’Analyse, CP 218, Boulevard du Triomphe, B-1050 Bruxelles, Belgique.}
\email{David.Tewodrose@ulb.ac.be}
\date{}

\begin{document}

\maketitle

\begin{abstract}
We study the structure of Gromov-Hausdorff limits of sequences  of Riemannian manifolds $\{(M_\alpha^n,g_\alpha)\}_{\alpha \in A}$ whose Ricci curvature satisfies a uniform Kato bound. We first obtain Mosco convergence of the Dirichlet energies to the Cheeger energy and show that tangent cones of such limits satisfy the $\RCD(0,n)$ condition. When assuming a non-collapsing assumption, we introduce a new family of monotone quantities, which allows us to prove that tangent cones are also metric cones. We then show the existence of a well-defined stratification in terms of splittings of tangent cones. We finally prove volume convergence to the Hausdorff $n$-measure. 
\end{abstract}

\section*{Introduction}

We study the structure of Gromov-Hausdorff or measured Gromov-Hausdorff limits of manifolds whose Ricci curvature satisfy a Kato type bound. Our results extend previous results proven by J. Cheeger and T. Colding for limits of manifolds carrying a uniform lower bound on the Ricci curvature.\\

\textbf{Gromov-Hausdorff convergence of manifolds.} In the 1980s, Mr~Gromov showed a compactness result for Riemannian manifolds satisfying a uniform lower bound on the Ricci curvature. It states that, if $\{(M_\alpha^n,g_\alpha,o_\alpha)\}_\alpha$ is a sequence of pointed complete Riemannian manifolds satisfying uniformly
\begin{equation}\label{eq:riclob}\Ric\ge K g\end{equation} for some $K\in \R$  and $o_\alpha\in M_\alpha$, then up to extracting a sub-sequence, the sequence $\left\{(M_\alpha^n,\dist_{g_\alpha}, o_\alpha)\right\}_\alpha$ converges in the pointed Gromov-Hausdorff topology to a complete proper metric space $(X,\dist,o)$. A natural question was then to describe the structure of such metric spaces arising as limits of smooth manifolds. In the 1990s, a series of results by J. Cheeger and T. Colding \cite{Col97, ChCo96, CheegerColdingI, ChCo00, CheegerColdingIII} made it possible to better understand this problem and their work launched a vast research program. Many recent results have led to a significant understanding of the so-called ``Ricci limit spaces'' \cite{CheegerPisa, CheegerNaber15, JiangNaber, CJN}.  

In the study of limit spaces, there are two different scenarios, depending on whether the sequence of manifolds is \textit{collapsing} or \textit{non-collapsing}. In the first case, the volume of unit balls $B_1(o_\alpha)$ goes to $0$ as $\alpha$ tends to infinity, while in the non-collapsed case there is a uniform lower bound on this volume. Since the work of K.~Fukaya \cite{Fukaya} it is known that, in the collapsed case, the Gromov-Hausdorff topology is not sufficient to recover good geometric information on the limit space, such as, for instance, information about the spectrum of its Laplacian. K.~Fukaya then introduced the \textit{measured} Gromov-Hausdorff topology. For that, one re-scales the Riemannian measure $\di\nu_{g_\alpha}$ and considers sequences of manifolds as sequences of metric measure spaces $(M_\alpha^n,\dist_{g_\alpha}, \di\mu_\alpha,o_\alpha)$, where $\di\mu_\alpha=\nu^{-1}_{g_\alpha}\left(B_1(o_\alpha)\right)\, \di\nu_{g_\alpha}$. Then again, up to extracting a sub-sequence, there is convergence to a metric measure space $(X,\dist, \mu,o)$ in the pointed measured Gromov-Hausdorff topology. This allows for finer results on the structure of the limit space. 

In the 2000s the works of J.~Lott, K.-T.~Sturm and C.~Villani showed that it is possible to define a generalization of a Ricci lower bound in the setting of metric measure spaces. This lead to the notions of $\CD(K,\infty)$ and $\CD(K,n)$ metric measure spaces, that are known to include Ricci limit spaces \cite{LottVillani,sturm2006I,Sturm2006II,Villani}. Later on, L.~Ambrosio, N.~Gigli and G.~Savaré introduced a refinement of the infinite dimensional $\CD(K,\infty)$ condition, the so-called Riemannian curvature dimension condition $\RCD(K,\infty)$, which is also satisfied by manifolds carrying the lower bound \eqref{eq:riclob} and is preserved under measured Gromov-Hausdorff convergence \cite{AmbrosioGigliSavareDuke}.  The finite dimensional $\RCD(K,n)$ condition was subsequently introduced and studied by N.~Gigli in \cite{GigliMAMS}. Under this more restrictive condition, it is possible, for instance, to define a Laplacian operator and thus reformulate classical inequalities of Riemannian geometry in the setting of metric measure spaces. Nowadays, the structure of $\RCD(K,n)$ spaces is fairly well understood \cite{MondinoNaber,BrueSemola,DPG}, and such spaces provide a good conceptual framework comprising Ricci limit spaces. \\

\textbf{Ricci limit spaces and beyond.} The work of J. Cheeger and T. Colding relies on several crucial tools and results. The first one is the Bishop-Gromov volume comparison theorem, which provides a monotone quantity given by the \emph{volume ratio}
$$r\mapsto \frac{\nu_g\left(B_r(x)\right)}{\mathbb{V}_{n,K}(r)}, \footnote{Where $\mathbb{V}_{n,K}(r)$ is the volume of a geodesic balls with radius $r$ in a simply connected space of constant curvature equals to $K/(n-1)$.}$$
Monotone quantities play a crucial role when investigating blow-up phenomena in geometric analysis. In the case of Ricci limit spaces, the monotonicity of the volume ratio is the keystone for understanding the local geometry of limit spaces. Two other very important results are the almost splitting theorem \cite{CheegerColdingI} and the theorem now known as ``almost volume cones implies almost metric cone'' (see \cite[Theorem 3.6]{ChCo96}). Together with the monotonicity of the volume ratio, they imply in particular that tangent cones of non-collapsed Ricci limit spaces are genuine metric cones. Many additional technical results are involved in the study of Ricci limit spaces. For example, the existence of good cut-off functions (with bounded gradient and Laplacian) plays an important role in many proofs for exploiting the Bochner formula. The construction of cut-off functions relies on strong analytic properties of manifolds with Ricci curvature bounded from below, such as Laplacian comparisons, Gaussian heat kernel bounds, the Cheng-Yau gradient estimate.

However, there are some very interesting contexts in which a good understanding of the convergence of smooth manifolds is needed, but a uniform lower bound on the Ricci curvature is not satisfied and the previous tools are not available: for example, in the study of the Ricci flow \cite{Bamler1, Bamler2, BamlerZhang,MSimon1,MSimon2} and of critical metrics \cite{TV1,TV2,TV3}. It is then important to investigate situations in which weaker assumptions on the curvature are made. The case in which one assumes some $L^p$ bound on the Ricci curvature, for $p > n/2$, has been well studied since the end of the 1990s and one gets a number of results about the structure of the limit space under an additional smallness assumption \cite{PW1,PW2, TianZhang, DWZ18}. Very recently, \cite{Chen20} obtained more results about the structure of the limit space. C.~Ketterer \cite{Ketterer} also opened the way to a new interesting perspective in the study of limits of manifolds with the appropriate $L^p$ bound on the Ricci curvature, by showing, among others, that tangent cones of the limit space are $\RCD(0,n)$-spaces. One question he raises in this work is whether the same result holds when assuming a Kato condition on the Ricci curvature: we give a positive answer to this question in one of our results (see Theorem \ref{thm:collapsed}). \\

\textbf{Kato type bound.} It has been recently remarked \cite{C16, Rose} that a Kato type bound on the Ricci curvature makes it possible to use ideas of Q.S. Zhang and M. Zhu \cite{ZZ} and, as a consequence, to obtain a Li-Yau type bound for solutions of the heat kernel. Many geometric estimates, that are known when the Ricci curvature is bounded from below, follow under such less restrictive Kato type bound. 

Let $(M^n,g)$ be a closed Riemannian manifold and introduce the function  $\Ricm\colon M\rightarrow \R_+$ to be the best function such that 
$$\Ric\ge -\Ricm g.$$
Define for all $\tau >0$
$$\mbox{k}_\tau(M^n,g)=\sup_{x\in M} \int_{[0,\tau]\times M}H(t,x,y)\Ricm(y)\di\nu_g(y)\di t,$$
where $H(t,\cdot,\cdot)$ is the Schwartz kernel (the heat kernel) of the heat operator $e^{-t \Delta}$. If for some $T>0$ the following Kato type bound
\begin{equation}\label{eq:Dynkin}
\mbox{k}_T(M^n,g)\le \frac{1}{16n},\end{equation}
holds, then one gets geometric and analytic results similar to the ones implied by the condition \eqref{eq:riclob} with $K=-1/T.$ Moreover, as noticed in \cite{C16},  the set of closed manifolds satisfying \eqref{eq:Dynkin} is pre-compact in the Gromov-Hausdorff topology. As a consequence, it is natural to ask under which extent results on the structure of Ricci limit spaces can be obtained under this weaker assumption. 

The Kato condition was introduced with the aim of studying Schr\"odinger operators in the Euclidean space
$$L=\Delta-V,$$ where $V\ge 0$ and our convention is that $\Delta=\textbf{-}\sum_i \partial^2/\partial x_i^2$ on $\R^n$. A non negative potential $V\colon \R^n\rightarrow \R_+$ is said to be in the Kato class, or to satisfy the Kato condition, if
$$\lim_{T\to 0+} \sup_{x\in \R^n}\iint_{[0,T]\times \R^n} \frac{e^{-\frac{\| x-y\|^2}{4t}}}{(4\pi t)^{\frac n2}} V(y)\di y \di t=0.$$
At the regularity level, this condition only requires that $V$ is the Laplacian of a continuous function. Moreover, if $V$ is in the Kato class,  when $t$ tends to 0 one can compare the semi-groups $e^{-t\Delta}$ and $e^{-t(\Delta-V)}$ and thus recover good properties of $e^{-t(\Delta-V)}$ for $t$ small enough.  We refer to the beautiful survey of B. Simon \cite{BSimon} for a extensive overview on the Kato condition in the Euclidean setting, and to the book of B. Guneysu \cite{Guneysu} for an account of the Kato condition in the Riemannian setting. In our context, the potential $V$ is chosen to be $\Ricm$. 

Now, an assumption in the spirit of the Kato condition in $\R^n$ would require not only that $\mbox{k}_T$ is uniformly bounded along the sequence of manifolds for a fixed $T>0$, but also some uniform control on the way that $\mbox{k}_\tau$ goes to $0$ when $\tau$ goes to $0$. This kind of control is actually required in our analysis in order to be able to compare infinitesimal geometry of limit spaces with Euclidean geometry. In particular, this plays an important role to get the appropriate monotone quantities that we rely on for studying the geometry of tangent cones.\\

\textbf{Main results.}  We begin by illustrating our main results in the non-collapsed case. 

\begin{thmA}
\label{thm:nc} 
Let $(X,\dist, o)$ be the pointed Gromov-Hausdorff limit of a sequence of closed Riemannian manifolds $\left\{(M_\alpha^n,\dist_{g_\alpha}, o_\alpha)\right\}_\alpha$ satisfying the uniform Kato bound \begin{equation}\label{eq:MajKato}
\forall \tau\in (0,1]\colon \quad \Kato_\tau(M_\alpha^n,g_\alpha)\le f(\tau), 
\end{equation}
where $f\colon [0,1]\rightarrow \R_+$ is a non-decreasing function such that
\begin{equation}\label{eq:StrongKato}
\int_0^1 \sqrt{f(\tau)}\,\frac{d\tau}{\tau}<\infty,
\end{equation}
and the non-collapsing condition
\begin{equation}
\label{eq:NCintro}
\nu_{g_\alpha}\left(B_1(o_\alpha)\right)\geq v >0.
\end{equation}
Then the following holds.  
 \begin{enumerate}
 \item[\emph{(i)}] \textbf{Volume convergence:} For any $r>0$ and $x_\alpha\in M_\alpha$ such that $x_\alpha\to x\in X$ we have
 $$\lim_{\alpha\to \infty} \nu_{g_\alpha}\left( B_r(x_\alpha)\right)=\cH^n\left(B_r(x)\right).$$
 \item[\emph{(ii)}]  \textbf{Structure of tangent cones: } For any $x\in X$, tangent cones of $X$ at $x$ are $\RCD(0,n)$ metric cones.
 \item[\emph{(iii)}]  \textbf{Almost everywhere regularity :}  For $\cH^n\text{-a.e. } x\in X$, $\left(\R^n, \dist_{\mathrm{eucl}}\right)$ is the unique tangent cone of $X$ at $x$.
 \item[\emph{(iv)}]  \textbf{Stratification:} Let $\cS^k$ be the set consisting of the  points $x\in X$ such that $X$ does not carry any tangent cone at $x$ that splits isometrically a factor $\R^{k+1}$. Then $\cS^k$ satisfies
$$\dim_{\cH}\cS^k\le k.$$ 
 \end{enumerate}
 \end{thmA}

The first point is a generalization of the volume continuity showed in \cite{Col97}. The fact that tangent cones are metric cones in the analog of \cite[Theorem 5.2]{CheegerColdingI} and the two last points correspond to \cite[Theorem 4.7]{CheegerColdingI} (see also \cite[Theorem 10.20]{CheegerPisa}). We also conjecture that, under the same assumptions, we have $\cS_{n-1}=\cS_{n-2}$, that is the singular set has codimension at least two. As for the case of Ricci limit spaces, we expect the existence of an open subset that is $n$-rectifiable and bi-Hölder homeomorphic to a manifold. We plan to address these questions in our subsequent work.

Observe that the uniform Kato bound \eqref{eq:MajKato} with a function satisfying \eqref{eq:StrongKato} is guaranteed as soon as one has an appropriate estimate on the $L^p$ norm of the Ricci curvature. This is due to C.~Rose and P.~Stollmann \cite{RoseStollmann}: thanks to their work, it is possible to show that if $p>n/2$ and $\eps(p,n,\kappa)$ is small enough, then the following estimate \footnote{with $x_+=\max\{x,0\}$.}
$$\diam^2(M_\alpha,g_\alpha)\left(\fint_M \left|\Ricm-\kappa^2\right|_+^p\di\nu_{g_\alpha}\right)^{\frac 1p}\le \eps(p,n,\kappa)$$ implies \eqref{eq:MajKato} and \eqref{eq:StrongKato}. 
Similarly, our non-collapsing assumption and the uniform Kato bound is ensured under the assumptions considered by G.~Tian and Z.~Zhang in the study of Kähler-Ricci flow $g(t)$ \cite{TianZhang}, that are an a-priori bound on the $L^p$ norm of Ricci curvature for $p > n/2$ 
$$\forall t \geq 0 \int_M \left|\Ric_{g(t)}\right|^p\di\nu_{g(t)}\le \Lambda$$ 
and a non-collapsing condition 
$$\forall t \geq 0, \ \forall x \in M, \forall r \in (0,1), \nu_{g(t)}(B_r(x)) \geq v r^n.$$

In the collapsed case, our results give less information about the structure of the limit space, but apply with a weaker hypothesis. 

\begin{thmA} 
\label{thm:collapsed}
Let $(X,\dist, \mu,o)$ be the pointed measured Gromov-Hausdorff limit of a sequence  $\left\{(M_\alpha^n,\dist_{g_\alpha}, \mu_\alpha, o_\alpha)\right\}_\alpha$, satisfying the uniform Kato bound \eqref{eq:MajKato} for some non-decreasing positive function  $f\colon [0,1]\rightarrow \R_+$ such that  
 \begin{equation}\label{eq:Kato}
\lim_{\tau\to 0} f(\tau)=0, 
\end{equation}
with the re-scaled measure
\begin{equation}\label{rescaledmu}
 \di\mu_\alpha=\frac{\di \nu_{g_\alpha}}{\nu_{g_\alpha}\!\!\left(B_1(o_\alpha)\right)}.\end{equation}
Then we have:
\begin{enumerate}
\item[\emph{(i)}]  the Cheeger energy is quadratic and  $(X,\dist,\mu)$ is an infinitesimally Hilbertian space in the sense of \cite{GigliMAMS};
\item[\emph{(ii)}]  for any $x\in X$, metric measure tangent cones of $X$ at $x$ are $\RCD(0,n)$;
\item[\emph{(iii)}]  if $X$ is compact, then the spectrum of the Laplace operators of  $(M_\alpha^n,g_\alpha)$ converges to the spectrum of the Laplacian associated  to $(X,\dist,\mu)$.
\end{enumerate}
\end{thmA}

The last point extends \cite[Theorem 0.4]{Fukaya}. The second point generalizes C.~Ketterer's result \cite[Corollary 1.7]{Ketterer}: under the same assumptions of \cite{TianZhang} that we recalled above, he proved that tangent cones are $\RCD(0,n)$ spaces. Part of his proof relies on an almost splitting theorem of \cite{TianZhang}. In our case, we do not use an almost splitting theorem. Nonetheless, we point out that our proof shows that whenever the sequence of manifolds $(M_\alpha, g_\alpha)$ is such that, for some $\tau>0$, $\mbox{k}_\tau(M_\alpha, g_\alpha)$ tends to zero as $\alpha$ goes to infinity, then the limit $(X,\dist,\mu)$ is an $\RCD(0,n)$ space. As a consequence, Gigli's splitting theorem for $\RCD(0,n)$ spaces applies \cite{GigliSplitting, GigliOverview}. Moreover, a contradiction argument based on pre-compactness leads to an almost splitting theorem for manifolds with $\mbox{k}_\tau$ small enough. Then we do have an almost splitting theorem in our setting, but in contrast to what happens in the study of Ricci limit and $\RCD$ spaces, where such theorem represents a key tool, we obtain it as a consequence of our results rather than relying on it on our proofs. \\

\textbf{Outline of proofs.} We now describe some of the ideas playing a role in our proofs  and their organization, starting from Theorem \ref{thm:collapsed}. The Kato type bound \eqref{eq:Dynkin} provides very good heat kernel estimates (see for example Proposition \ref{prop:UnifPI}) which imply in particular that a sequence of manifolds satisfying \eqref{eq:Dynkin}, when considered as a sequence of \emph{Dirichlet} spaces, is uniformly doubling and carries a uniform Poincaré inequality. This, together with the results of A.~Kasue \cite{KasueSurvey} and K.~Kuwae and T.~ Shioya \cite{KuwaeShioya}, ensures that the measured Gromov-Hausdorff convergence can be strengthened, in the sense that one additionally obtains Mosco convergence of the Dirichlet energies. More precisely, assume that $(X,\dist,\mu,o)$ is a pointed measured Gromov-Hausdorff limit of a sequence of closed manifolds $\left\{(M_\alpha^n,\dist_{g_\alpha}, \mu_\alpha, o_\alpha)\right\}_\alpha$, where 
 $\di\mu_\alpha$ is either the Riemannian volume in the non-collapsing case, or its re-scaled version \eqref{rescaledmu} in the collapsing one. Up to extraction of a sub-sequence, it is possible to define a closed, densely defined quadratic form $\cE$ on $L^2(X,\mu)$ which is obtained as the Mosco limit of the Dirichlet energies: 
$$u\mapsto \int_{M_\alpha} |du|^2_{g_\alpha}\di\mu_\alpha.$$
A priori, different sub-sequences could lead to different quadratic forms. Moreover, the space $(X,\dist, \mu, o)$ carries both the Dirichlet energy $\cE$ and the Cheeger energy canonically associated to $\dist$ and $\mu$. In general, these two energies do not need to coincide, see for instance \cite[Theorem 7.1]{AldanaCarronTapie}: it gives an example of a limit space such that the distance is a Finsler metric and thus the Cheeger energy, not being quadratic, cannot coincide with any Dirichlet form. 

However, under the Kato bound \eqref{eq:MajKato} together with \eqref{eq:Kato}, we can use the Li-Yau type inequality in order to get estimates for the solutions of the heat equation on the manifolds $(M_\alpha, g_\alpha)$ . We show that such estimates pass to the limit and hold on the Dirichlet limit space $(X,\dist,\mu,\cE)$. As a consequence, we can apply a result due to L.~Ambrosio, N.~Gigli, G.~Savaré \cite{AGS15} and to P.~Koskela, N.~Shanmugalingam, Y.~Zhou \cite{KoskelaShanZhou} and we obtain that the limit Dirichlet energy $\cE$ coincides in fact with the Cheeger energy of the metric measure space $(X,\dist,\mu)$. Hence, under  conditions \eqref{eq:MajKato} and \eqref{eq:Kato}, measured Gromov-Hausdorff convergence implies Mosco convergence of the Dirichlet energies to the Cheeger energy. 

Our proof also applies when for some $\tau>0$
\begin{equation}
\label{eq:KatoTgCone}
\lim_{\alpha\to \infty}  \mbox{k}_\tau(M_\alpha^n,g_\alpha)=0.
\end{equation}
Under this condition, we additionally show that the Bakry-Ledoux gradient estimate holds on the limit space and thus $(X,\dist,\mu)$ is an $\RCD(0,n)$ space. Thanks to the re-scaling properties of the heat kernel, if $(X,\dist,\mu,o)$ is a limit of manifolds satisfying \eqref{eq:MajKato} and \eqref{eq:Kato}, then any tangent cone of $X$ is a limit of re-scaled manifolds for which \eqref{eq:KatoTgCone} holds for all $\tau>0$. Therefore, this implies Theorem \ref{thm:collapsed}(ii). 

As for the non-collapsed case, we prove that the limit measure $\mu$ coincides with the $n$-dimensional Hausdorff measure, so that Gromov-Hausdorff convergence under conditions  \eqref{eq:StrongKato} and \eqref{eq:NCintro}  not only implies Mosco converge of the energies, but also measured Gromov-Hausdorff convergence. To prove this, we introduce a new family of monotone quantities that, when the Ricci curvature is non-negative, interpolates between the Li-Yau's inequality and Bishop-Gromov volume comparison theorem.  Our quantities are modeled on Huisken's entropy for the mean curvature flow \cite{HuiskenMCF}. In order to define them, for a closed manifold $(M^n,g)$ with heat kernel $H$, we define the function $U$ by setting
$$H(t,x,y)=\frac{\exp\left(-\frac{U(t,x,y)}{4t}\right)}{(4\pi t)^{\frac n2}}.$$
We then introduce for any $s,t>0$ the Gaussian's type entropy
$$\theta_x(s,t)=\int_M \frac{\exp\left(-\frac{U(t,x,y)}{4s}\right)}{(4\pi s)^{\frac n2}}\di\nu_g(y).$$
When the Ricci curvature is non-negative, we show that for all $x \in M$ the function 
$$\lambda \mapsto \theta_x(\lambda s, \lambda t),$$
is monotone non-increasing for $s\geq t$, non-decreasing for $s\leq t$. This interpolates between the Bishop-Gromov and Li-Yau's inequalities in the following sense. When $t=0$, we can write
$$\theta_x(s,0)=\int_M \frac{\exp\left(-\frac{\dist^2(x,y))}{4s}\right)}{(4\pi s)^{\frac n2}}\di\nu_g(y)=\frac 12 \int_0^\infty e^{-\frac{\rho^2}{4}}\rho\, \frac{\nu_g\left(B_{\rho\sqrt{s}}(x)\right) }{\mathbb{V}_{n,0}(\rho\sqrt{s})} d\rho.$$
Then, Bishop-Gromov volume comparison implies that for any $s \geq 0$ the function $\lambda \mapsto \theta_x(\lambda s, 0)$ is monotone non-increasing. Moreover, one of the consequences of Li-Yau's inequality is that for all $x \in M$ the map
$$t \mapsto (4\pi t)^{\frac{n}{2}}H(t,x,x)$$
is monotone non-decreasing. When noticing that the semi-group law allows one to write
$$H(2t,x,x)=\int_M H(t,x,y)^2 \di \nu_g(y),$$
a simple computation shows that for any $t>0$ and $s= t/2$ the function $\lambda \mapsto \theta_x(\lambda t/2, \lambda t)$ is monotone non-decreasing. 

Observe that by Varadhan's formula \eqref{eq:varadhan} we have
$$\dist(x,y)^2=\lim_{t\to 0} U(t,x,y),$$
so that, when $t$ tends to zero, our quantities $\theta_x$ tend to 
$$\Theta_x(s)= (4\pi s)^{\frac{n}{2}}\int_M e^{-\frac{d(x,y)^2}{4s}}\di \nu_g(y).$$
This corresponds to Huisken's entropy and to the $\cH_s$ volume considered by W.~Jiang and A.~Naber in \cite{JiangNaber}, where it is shown to be monotone non-increasing if the Ricci curvature is non-negative. Moreover, in the case of a Ricci limit space $(X,\dist,\mu)$, the limit of $\Theta_x$ as $s$ tends to 0 coincides with the \emph{volume density} at $x$, that is
$$\vartheta_X(x)=\lim_{r \to 0}\frac{\mu(B(x,r))}{\omega_n r^n},$$
where $\omega_n$ is the volume of the unit ball in $\R^n$. Bishop-Gromov inequality guarantees that such limit does exist.

In our setting, with the uniform Kato bounds \eqref{eq:Kato} and \eqref{eq:StrongKato}, the Li-Yau type inequality allows us to show that our quantities $\theta_x$ are \emph{almost monotone}, in the sense that there exists a function $F$ of $\lambda$, tending to 1 as $\lambda$ tends to 0, such that the map
$$\lambda \mapsto \theta_x(\lambda s, \lambda t) F(\lambda)$$
has the same monotonicity as $\theta_x$ when the Ricci curvature is non-negative. There is a limitation on the range of parameter where our monotonicity holds, when $t\le s$ we also need $ s\preceq t/\sqrt{f(t)}$. As a consequence, the quantity $\Theta_x$ is not monotone and we do not get a monotone quantity based on the volume ratio. 

Observe that the only bound \eqref{eq:MajKato}, with a function tending to 0 as $t$ goes to zero, is not enough to obtain the above family of monotone quantities: due to the dependence of the Li-Yau type inequality on $\mbox{k}_\tau$, some kind of integral bound on $\mbox{k}_\tau$ is needed. Moreover, for a sequence of smooth manifolds $(M^n_\alpha, g_\alpha)$, the uniform bound \eqref{eq:StrongKato} implies that function $F$ is the same for all $(M_\alpha, g_\alpha)$, so that we get a corresponding family of monotone quantities on the limit space $(X,\dist, \mu)$. 

Thanks to this almost monotonicity, we are able to show that on a tangent cone at $x \in X$ the quantity $\Theta_x$ is constant. Then for all $r >0$ the measure of balls centered at $x$ is equal to $\Theta_x \omega_n r^n$. This, together with the fact that tangent cones are $\RCD(0,n)$ spaces and with the main result of \cite{DePhGi}, allows us to obtain that tangent cones are metric cones.  

We also prove that the almost monotonicity of $\theta_x$ implies that the volume density $\vartheta_X(x)$ is well defined on the limit space, despite the lack of monotonicity of the volume ratio. We then show that the volume density is lower semi-continuous under measured Gromov-Hausdorff convergence. As a consequence, we obtain the stratification result from arguments inspired by B.~White \cite{White} and G.~De Philippis, N.~Gigli \cite{DPG}. In the same proof, we get that $\mu$-almost everywhere tangent cones are Euclidean, with a measure given by $\vartheta_X(x)\cH^n$. In order to prove volume convergence, we then show that 
$$\mu\!-\!\text{a.e. } x\in X\colon \vartheta_X(x)=1.$$

For this purpose, we prove the existence at almost every point $x\in X$ of harmonic $\eps$-splitting maps $H\colon B_r(x)\rightarrow \R^n$. Splitting maps are ``almost coordinates'', in the sense that they are $(1+\eps)$-Lipschitz, $\nabla H$ is close to the identity and the Hessian is close to zero in $L^2$. They have been extensively used in the study of Ricci limit spaces and were recently proven to exist on $\RCD$ spaces too \cite{BPS}. In our case, we obtain a very good control of $\nabla H$ thanks to the Mosco convergence of Dirichlet energy. This is still not enough to prove, as in \cite{CheegerPisa} or \cite{Gallot}, that $\vartheta_X(x)=1$. But we are also able to obtain the Hessian bound thanks to one of our main technical tools, that is the existence of good cut-off functions when just the Kato type bound \eqref{eq:Dynkin} is satisfied. \\

\textbf{Outline of the paper.} 
Section 2 includes the main preliminary tools that we rely on throughout the paper. After introducing the convergence notions that we need, we focus on Dirichlet spaces. We state a compactness result for PI Dirichlet spaces, originally observed in \cite{KasueSurvey}, for which we give a proof in the Appendix, and we collect the assumptions under which a Dirichlet space satisfies the $\RCD$ condition. 

In Section 3, we introduce the different Kato type conditions that we consider in the rest of the paper, we state pre-compactness results and show that under assumptions \eqref{eq:MajKato} and \eqref{eq:StrongKato} the intrinsic distance associated to the Dirichlet energy coincides with the limit distance. In the non-collapsing case, we recall a useful Ahlfors regularity result due to the first author that also holds in the limit. 

Section 4 is devoted to proving some technical tools obtained under assumption \eqref{eq:Dynkin}, in particular the existence of good cut-off functions and the resulting Hessian bound. 

In Section 5 we prove Theorem \ref{thm:collapsed}, first by showing that under assumptions \eqref{eq:MajKato} and \eqref{eq:Kato} the Dirichlet energies converge to the Cheeger energy. This immediately implies convergence of the spectrum when $X$ is compact. We then prove that if $\mbox{k}_\tau(M_\alpha,g_\alpha)$ tends to zero for some fixed $\tau>0$, the limit space is an $\RCD(0,n)$ space.  

In Section 6, we introduce and study the quantity $\theta_x(t,s)$. We show its almost monotonicty and then obtain that, in the non-collapsing case, under assumptions \eqref{eq:MajKato} and \eqref{eq:StrongKato}, tangent cones are metric cones and the volume density is well-defined. Section 7 is devoted to proving Theorem \ref{thm:nc} (iv).  In particular, we obtain that $\mu$-a.e. tangent cones are unique and coincides  with $(\R^n, \dist_e, \vartheta_X(x)\cH^n,0)$. In the last section we show that $\vartheta_X(x)$ is equal to one almost everywhere: we prove existence of harmonic splitting maps and as a consequence we get volume convergence. 

In the Appendix we show the convergence results that are needed in Section 5, for passing to the limit the appropriate estimates on manifolds, and in Section 4, to get the existence of $\eps$-splitting harmonic maps with a good bound on the gradient. We also give an explicit proof of the compactness theorem for PI Dirichlet spaces. \\

\textbf{Acknowledgement.} The first author thanks the Centre Henri Lebesgue ANR-11-LABX-0020-01 for creating an attractive mathematical environment; he was also partially supported by the ANR grants: {\bf ANR-17-CE40-0034}: {\em CCEM} and {\bf ANR-18-CE40-0012}: {\em RAGE}. The second author was partially funded by the ANR grant {\bf ANR-17-CE40-0034}: {\em CCEM}.

\tableofcontents

\section{Preliminaries}

Throughout this paper, $n$ is a positive integer, and $A$ is a countable, infinite, directed set like $\setN$, for instance. We choose to denote sequences with countable infinite sets: this means that if $\{u_\alpha\}_{\alpha \in A}$ is a sequence in a topological space $(X,\cT)$, then $\{u_\alpha\}$ converges to $u$ if and anly if for any neighborhood $U$ of $u$ there exists a finite subset $C \subset A$ such that $\alpha \notin C$ implies $u_\alpha\in U$. Similarly, a sequence $\{u_\alpha\}_{\alpha \in A}$ admits a convergent sub-sequence if and only if there exists an infinite subset $B \subset A$ such that the sequence $\{u_\beta\}_{\beta \in B}$ converges.

All the manifolds we deal with in this paper are smooth and connected, and the Riemannian metrics we consider on these manifolds are all smooth. We often use the notation $M^n$ to specify that a manifold $M$ is $n$-dimensional. We call closed any Riemannian manifold which is compact without boundary. Whenever $(M,g)$ is a Riemannian manifold, we write  $\dist_g$ for its Riemannian distance, $\nu_g$ for its Riemannian volume measure, $\Delta_g$ for its Laplacian operator which we choose to define as a non-negative operator, i.e.
\[
\int_M g(\nabla u,\nabla v) \di\nu_g = \int_M (\Delta_gu) v \di\nu_g
\]
for any compactly supported smooth functions $u,v:M\to \setR$.

We recall that a metric space $(X,\dist)$ is called proper if all closed balls are compact and that it is called geodesic if for any $x,y \in X$ there exists a rectifiable curve $\gamma$ joining $x$ to $y$ whose length is equal to $\dist(x,y)$, in which case $\gamma$ is called a geodesic from $x$ to $y$. We also recall that the diameter of a metric space $(X,\dist)$ is set as $\diam(X):=\sup \{\dist(x,y) : x,y \in X \}$.  If $f:X \to \setR$ is a locally $\dist$-Lipschitz function, we define its local Lipschitz constant $\Lip_\dist f$ by setting
$$
\Lip_\dist f(x) :=
\begin{cases}
\limsup\limits_{y \to x} \frac{|f(x)-f(y)|}{\dist(x,y)} &  \text{if $x \in X$ is not isolated},\\
\qquad \quad 0 & \text{otherwise}.
\end{cases}
$$

We call metric measure space any triple $(X,\dist,\mu)$ where $(X,\dist)$ is a metric space and $\mu$ is a Radon measure finite and non-zero on balls with positive radius, and we write $B_r(x)$ for the open metric ball centered at $x \in X$ with radius $r>0$, and $\bar{B}_r(x)$ for the closed metric ball. We may often implicitly consider a Riemannian manifold $(M,g)$ as the metric measure space $(M,\dist_g,\nu_g)$ in which case metric balls are geodesic balls. 

We use standard notations to denote several classical function spaces: $L^p(X,\mu)$, $L^p_{loc}(X,\mu)$, $\cC(X)$, $\Lip(X)$ or $\Lip(X,\dist)$, $\cC^\infty(M)$ and so on. We use the subscript $c$ to denote the subspace of compactly supported functions of a given function space, like $\cC_c(X)$ for compactly supported functions in $\cC(X)$, for instance.  We write $\cC_0(X)$ for the space of continuous functions converging to $0$ at infinity, which is the $L^\infty(X,\mu)$-closure of $\cC_c(X)$.

We write $\mathbf{1}_A$ for the characteristic function of some set $A \subset X$.  By $\supp f$ (resp.~$\supp \mu$) we denote the support of a function $f$ (resp.~a measure $\mu$).  If $f$ is a measurable map from a measured space $(X,\mu)$ to a measurable one $Y$, we write $f_\#\mu$ for the push-forward measure of $\mu$ by $f$.

For any $s>0$ we write $\omega_s$ for the constant $\pi^{s/2}/\Gamma(s/2+1)$, where $\Gamma$ is the usual gamma function; as well-known, in case $s$ is an integer $k$, then $\omega_k$ coincides with the Hausdorff measure of the unit Euclidean ball in $\setR^k$.

\subsection{The doubling condition}

Let us begin with recalling the definition of a doubling metric measure space.

\begin{D} Given $R\in (0,+\infty]$ and $\upkappa\ge 1$ a  metric measure space  $(X,\dist,\mu)$ is called $\upkappa$-doubling at scale $R$ if for any ball $B_r(x)\subset X$ with $r\le R$ we have
$$\mu\left(B_{2r}(x)\right)\le \upkappa\mu\left(B_{r}(x)\right).$$
When $R=+\infty$, we simply say that $(X,\dist,\mu)$ is doubling.
\end{D}
Doubling metric measure spaces have the  following useful properties.
\begin{prop}\label{prop:doubling} Assume that $(X,\dist,\mu)$ is  $\upkappa$-doubling at scale $R$ for some $\upkappa\ge 1$ and $R \in(0,+\infty]$ and that $(X,\dist)$ is geodesic. Then there exists $c,\lambda,\updelta>0$ depending only on $\upkappa$ such that:
\begin{enumerate}[i)]
\item  $\mu\left(B_{r}(x)\right)\le c e^{\lambda \frac{\dist(x,y)}{r}} \mu\left(B_{r}(y)\right)$ for any $x,y\in X$ and $0<r\le R$,
\item $\mu\left(B_{r}(x)\right)\le c \left(\frac{r}{s}\right)^\lambda \mu\left(B_{s}(x)\right)$ for any $x\in X$ and $0<s\le r\le R$,
\item $ \mu\left(B_{S}(x)\right)\le e^{\lambda \frac{S}{s}} \mu\left(B_{s}(x)\right)$ for any $x\in X$ and $0<s\le R\le S$,
\item $ \mu\left(B_{s}(x)\right) e^{-\lambda} \left(\frac{r}{s}\right)^\updelta \le\mu\left(B_{r}(x)\right)$ for any $x\in X$ and $0<s<r<\min\{R, D/2\}$, where $D=\diam (X,\dist)$,
\item $\mu\left(B_{r}(x)\setminus B_{r-\tau}(x)\right)\le c\left( \frac{\tau}{r}\right)^\updelta 
\mu\left(B_{r}(x)\right)$ for any $x\in X$, $r>0$ and $0<\tau<\min\{r,R\}$.
\end{enumerate}
\end{prop}
We refer to \cite[Subsection 2.3]{HS} for the first four properties and to \cite[Lemma 3.3]{colding1998liouville} or \cite{tessera2007volume} for the last one.

\subsection{Dirichlet spaces} Let us recall now some classical notions from the theory of Dirichlet forms; we refer to \cite{FOT} for details.

Let $H$ be a Hilbert space of norm $|\cdot|_H$. We recall that a quadratic form $Q:H\to[0,+\infty]$ is called closed if its domain $\mathcal{D}(Q)$ equipped with the norm $|\cdot|_Q:=(|\cdot|_{H}^2 + Q(\cdot))^{1/2}$ is a Hilbert space. 

Let $(X,\cT)$ be a locally compact separable topological space equipped with a $\sigma$-finite Radon measure $\mu$ fully supported in $X$. A Dirichlet form on $L^2(X,\mu)$ with dense domain $\cD(\cE) \subset L^2(X,\meas)$ is a non-negative definite bilinear map $\cE : \cD(\cE) \times \cD(\cE) \to \setR$ such that $\cE(f):=\cE(f,f)$ is a closed quadratic form satisfying the Markov property, that is for any $f \in \cD(\cE)$, the function $f_{0}^{1} = \min ( \max(f,0),1)$ belongs to $\cD(\cE)$ and $\cE(f_{0}^{1}) \le \cE(f)$; we denote by $\langle \cdot, \cdot \rangle_\cE$ the scalar product associated with $|\cdot|_\cE$.
We call such a quadruple $(X,\cT,\mu,\cE)$ a Dirichlet space. When $\cT$ is induced by a given distance $\dist$ on $X$,  we write $(X,\dist,\mu,\cE)$ instead of $(X,\cT,\mu,\cE)$ and call $(X,\dist,\mu,\cE)$ a metric Dirichlet space.

Any Dirichlet form $\cE$ is naturally associated with a non-negative definite self-adjoint operator $L$ with dense domain $\cD(L) \subset L^2(X,\mu)$ defined by
$$
\cD(L):=\left\{f \in \cD(\cE) \, : \, \exists h=:Lf \in L^2(X,\mu)\, \, \text{s.t.}\, \, \cE(f,g)=  \int_X h g \di \mu \, \, \, \forall g \in \cD(\cE)\right\}.
$$
The spectral theorem implies that $L$ generates an analytic sub-Markovian semi-group $(P_t:=e^{-tL})_{t>0}$ acting on $L^2(X,\mu)$ where for any $f \in L^2(X,\mu)$, the map $t \mapsto P_tf$ is characterized as the unique $C^1$ map $(0,+\infty)\to L^2(X,\mu)$, with values in $\cD(L)$, such that
$$
\begin{cases}
\frac{\di}{\di t} P_tf = -L(P_t f) \qquad \forall t>0,\\
\lim\limits_{t \to 0} \|P_t f - f\|_{L^2(X,\mu)}=0.
\end{cases}
$$
Moreover, we get the property that when $0\le f\le 1$ then $0\le P_tf\le 1$.
Standard functional analytic theory shows that $(P_t)_{t>0}$ extends uniquely for any $p \in [1,+\infty)$ to a strongly continuous semi-group of linear contractions in $L^p(X,\meas)$. Moreover, the spectral theorem yields a functional calculus which justifies the following estimate: for any $t>0$ and $f \in \cD(\cE)$,
\begin{equation}\label{eq:functcal}
\|f - P_tf\|_{L^2} \le \sqrt{t} \cE(f).
\end{equation}

\emph{Heat kernel.} We call heat kernel of $\cE$ any function $H:(0,+\infty) \times X \times X \to \setR$ such that for any $t>0$ the function $H(t,\cdot,\cdot)$ is $(\mu \otimes \mu)$-measurable and
\begin{equation}\label{eq:defH}
P_t f (x) = \int_X H(t,x,y) f(y) \di \mu(y)  \qquad \text{for $\mu$-a.e.~$x \in X$},
\end{equation}
for all $f\in L^2(X,\mu)$. If $\cE$ admits a heat kernel $H$, then it is non negative and  symmetric with respect to its second and third variable, and for any $t>0$ the function $H(t,\cdot,\cdot)$ is uniquely determined up to a $(\mu\otimes \mu)$-negligible subset of $X\times X$. Moreover, the semi-group property of $(P_t)_{t>0}$ results into the so-called Chapman-Kolmogorov property for $H$:
\begin{equation}\label{eq:ChapmanKolmogorov}
\int_X H(t,x,z)H(s,z,y) \di \mu(z) = H(t+s,x,y), \qquad \forall x,y \in X, \, \, \forall s,t>0.
\end{equation}
The space $(X,\cT,\mu,\cE)$ -- or the heat kernel $H$ -- is called stochastically complete whenever for any $x \in X$ and $t>0$ it holds
\begin{equation}\label{eq:stocom}
\int_X H(t,x,y) \di \mu(y) = 1.
\end{equation}

\emph{Strongly local, regular Dirichlet spaces.} Let us recall now an important definition.

\begin{D}
A Dirichlet form $\cE$ on $L^2(X,\mu)$ is called \textit{strongly local} if $\cE(f,g)=0$ for any $f, g \in \cD(\cE)$ such that $f$ is constant on a neighborhood of $\supp g$, and \textit{regular} if $\cC_c(X) \cap \cD(\cE)$ contains a core, that is a subset which is both dense in $\cC_c(X)$ for $\|\cdot\|_{\infty}$ and in $\cD(\cE)$ for $|\cdot|_\cE$. If $(X,\cT,\mu,\cE)$ is a Dirichlet space where $\cE$ is strongly local and regular, we say that $(X,\cT,\mu,\cE)$ is a strongly local, regular Dirichlet space.
\end{D}

By a celebrated theorem from A.~Beurling and J.~Deny \cite{BeurlingDeny}, any strongly local, regular Dirichlet form $\cE$ on $L^2(X,\mu)$ admits a \textit{carré du champ}, that is a non-negative definite symmetric bilinear map $\Gamma : \cD(\cE) \times \cD(\cE) \to \mathrm{Rad}$, where $\mathrm{Rad}$ is the set of signed Radon measures on $(X,\cT)$, such that
$$
\cE(f,g) = \int_X \di \Gamma(f,g) \qquad \forall f,g \in \cD(\cE),
$$
where $\int_X \di \Gamma(f,g)$ denotes the total mass of the measure $\Gamma(f,g)$.  Moreover, $\Gamma$ is local, meaning that
$$
\int_A \di \Gamma(u,w) = \int_A \di \Gamma(v,w)
$$
holds for any open set $A\subset X$ and any $u,v,w \in \cD(\cE)$ such that $u=v$ on $A$. Thanks to this latter property, $\Gamma$ extends to any $\mu$-measurable function $f$ such that for any compact set $K\subset X$ there exists $g \in \cD(\cE)$ such that $f=g$ $\mu$-a.e.~on $K$; we denote by $\cD_{loc}(\cE)$ the set of such functions. Then $\Gamma$ satisfies the Leibniz rule and the chain rule. If we set $\Gamma(f):=\Gamma(f,f)$, this implies
\begin{equation}\label{eq:Leibniz}
\Gamma(fg) \le 2 (\Gamma(f)+\Gamma(g))
\end{equation}
for any $f,g \in \cD_{loc}(\cE) \cap L^{\infty}_{loc}(X,\mu)$ and
\begin{equation}\label{eq:chain}
\Gamma(\eta \circ h) = (\eta' \circ h)^2 \Gamma(h)
\end{equation}
for any $h \in \cD_{loc}(\cE)$ and $\eta \in C^1(\setR)$ bounded with bounded derivative.

A final consequence of strong locality and regularity is that the operator $L$ canonically associated to $\cE$ satisfies the classical chain rule:
\begin{equation}\label{eq:chainrule}
L(\phi \circ f) =  (\phi' \circ f) Lf - (\phi''\circ f) \Gamma(f) \qquad \forall f \in \mathbb{G}, \,\, \, \forall \phi \in C^\infty([0,+\infty),\setR),
\end{equation}
where $\mathbb{G}$ is the set of functions $f \in \cD(L)$ such that $\Gamma(f)$ is absolutely continuous with respect to $\mu$ with density also denoted by $\Gamma(f)$. In particular:
\begin{equation}\label{eq:chainrulesquare}
Lf^2 = 2fLf - 2 \Gamma(f) \qquad \forall f \in \mathbb{G}.
\end{equation}

\emph{Intrinsic distance.} The carré du champ operator of a strongly local, regular Dirichlet form $\cE$ provides an extended pseudo-metric structure on $X$ given by the next definition.

\begin{D}The \textit{intrinsic} extended pseudo-distance $\dist_\cE$ associated with $\cE$ is defined by
\begin{equation}\label{eq:defdist}
\dist_\cE(x,y):=\sup \{|f(x)-f(y)| \, : \, f \in \cC(X) \cap \cD_{loc}(\cE) \, \, \, \text{s.t.} \, \, \Gamma(f) \le \mu\}
\end{equation}
for any $x, y \in X$, where $\Gamma(f) \le \mu$ means that $\Gamma(f)$ is absolutely continuous with respect to $\mu$ with density lower than $1$ $\mu$-a.e.~on $X$.
\end{D}
A priori, $\dist_\cE(x,y)$ may be infinite, hence we use the word ``extended''. Of course the case where $\dist_\cE$ does provide a metric structure on $X$ is of special interest.  In this regard, if $(X,\cT,\mu,\cE)$ is a strongly local, regular Dirichlet space where $\dist_\cE$ is a distance inducing $\cT$, we denote it by $(X,\dist_\cE,\mu,\cE)$. 

\medskip

\subsection{The Poincaré inequality and PI Dirichlet spaces} Given $R \in (0,+\infty]$, we say that a strongly local, regular Dirichlet space $(X,\dist_\cE,\mu,\cE)$ satisfies a $R$-scale- invariant Poincaré inequality  if there exists $\upgamma>0$ such that
\begin{equation}\label{eq:Poincaré}
\|u-u_B\|^2_{L^2(B)} \le \upgamma r^2 \int_B \di\Gamma(u)
\end{equation}
for any $u \in \cD(\cE)$ and any ball $B$ with radius $r \in (0,R]$. 
When $R=+\infty$, we simply say that $(X,\dist_\cE,\mu,\cE)$ satisfies a Poincaré inequality. The next definition is central in our work.

\begin{D}\label{def:PI} Given $R \in (0,+\infty]$,  $\upkappa \ge 1$ and $\upgamma>0$,  we say that a strongly local, regular Dirichlet space $(X,\dist_\cE,\mu,\cE)$ is $\mathrm{PI}_{\upkappa,\upgamma}(R)$ if it satisfies the following conditions:
\begin{itemize}
\item $(X,\dist_\cE,\mu)$ is $\upkappa$-doubling at scale $R$,
\item $(X,\dist_\cE,\mu,\cE)$ satisfies a $R$-scale- invariant Poincaré inequality \eqref{eq:Poincaré} with constant $\upgamma$.
\end{itemize}
\end{D}

We may use the terminology $\mathrm{PI}(R)$ if no reference to the doubling or Poincaré constant is required, or even $\mathrm{PI}$ if we do not need to mention the scale $R$.

\medskip
\emph{Geometry and analysis of $\mathrm{PI}$ Dirichlet spaces.} Assume that $(X,\dist_\cE,\mu,\cE)$ is a $\mathrm{PI}_{\upkappa,\upgamma}(R)$ Dirichlet space for some given $R \in (0,+\infty]$,  $\upkappa \ge 1$ and $\upgamma>0$. According to \cite{sturm1996analysis},  the strong locality and regularity assumptions on $\cE$ imply that the metric space $(X,\dist_\cE)$ is geodesic and that it satisfies the Hopf-Rinow theorem: it is proper if and only if it is complete.  Moreover, there is a relationship between the local Lipschitz constant and the carré du champ of $\dist_\cE$-Lipschitz functions, see \cite[Theorem 2.2]{KZ} and \cite[Lemma 2.4]{KoskelaShanZhou}:  when $u\in \Lip(X,\dist_\cE)$, then $u\in \cD_{loc}(\cE)$ and the Radon measure $\Gamma(u)$ is absolutely continuous with respect to $\mu$; moreover, there exists a constant $\eta \in (0,1]$ depending only on $\upkappa,\upgamma$ such that 
\begin{equation}\label{EnergyLip} \eta (\Lip_{\dist_\cE} u)^2 \le \frac{\di\Gamma(u)}{\di\mu}\le (\Lip_{\dist_\cE} u)^2 \qquad \text{$\mu$-a.e.~on $X$}.\end{equation}
In addition, it follows from \cite[Theorem 2.2]{KZ} that $\Lip_c(X,\dist_\cE)$ is dense in $\cD(\cE)$ and that for any $u\in \cD(\cE)$ the Radon measure $\Gamma(u)$ is absolutely continuous with respect to $\mu$ with density $\rho_u\in L^2_{loc}(X,\mu)$ comparable to the approximate Lipschitz constant of $u$. 

For a strongly local, regular Dirichlet space $(X,\dist_\cE,\mu,\cE)$,  to be $\mathrm{PI}_{\upkappa,\upgamma}(R)$ implies to have a Hölder continuous heat kernel $H$ satisfying Gaussian upper and lower bounds: there exists $C_1,C_2>0$ depending only on $\upkappa,\upgamma$ such that
\begin{equation}\label{eq:HKbound}
\frac{C_2^{-1}}{\meas(B_{\sqrt{t}}(x)} e^{- C_2 \frac{\dist_{\cE}^2(x,y)}{t}} \le H(t,x,y) \le \frac{C_1}{\meas(B_{\sqrt{t}}(x))}e^{- \frac{\dist_{\cE}^2(x,y)}{5t}}
\end{equation}
for all $t \in (0,R^2)$ and $x,y \in X$. This implication is actually an equivalence: see Theorem \ref{thm:eq} in the Appendix where we provide references. In fact such a Dirichlet space satisfies the Feller property: the heat semi-group extends to a continuous semi-group on $\cC_0(X)$. 

Moreover, a $\mathrm{PI}(R)$ Dirichlet space is necessary stochastically complete: this was proved on Riemannian manifolds by A. Grigor'yan \cite[Theorem 9.1]{GriBAMS} and extended to Dirichlet spaces by K-T. Sturm \cite[Theorem 4]{sturm1994analysis}. 

The above Gaussian upper bound can be improved to get the optimal Gaussian rate decay.
\begin{prop}\label{est:Gaussian14}
Let $(X,\dist_\cE,\mu,\cE)$ be a $\mathrm{PI}_{\upkappa,\upgamma}(R)$ Dirichlet space.  Then there exist $C, \nu>0$ depending only on $\upkappa,\upgamma$ such that for any $x,y \in X $ and $t \in (0,R^2)$,
\begin{equation}\label{eq:optimalgaussian}
H(t,x,y)\le \frac{C}{\mu (B_R(x))} \frac{R^\nu}{t^{\frac \nu 2}} \left(1+\frac{\dist_\cE^2(x,y)}{t}\right)^{\nu+1} \ e^{-\frac{\dist_\cE^2(x,y)}{4t}}
\end{equation}
Moreover,  Varadhan's formula holds:  for any $x,y\in X$,
\begin{equation}\label{eq:varadhan}
\dist_\cE^2(x,y)=-4\lim_{t\to 0+} t\log H(t,x,y).\end{equation}
\end{prop}
The Gaussian upper bound can be found in \cite[Theorem 5.2 ]{GriRev} (see also \cite{S,C} for optimal versions) and Varadhan's formula is due to ter Elst,  D.~Robinson, and A.~Sikora \cite{varadhan} (see also \cite{MR1809739} for an earlier result).

\subsection{Notions of convergence} We provide now our working definitions of convergence of spaces and of points, functions, bounded operators and Dirichlet forms defined on varying spaces.

\subsubsection{Convergence of spaces} Let us start with some classical definitions.

\emph{Pointed Gromov-Hausdorff convergence.} For any $\eps>0$, an $\eps$-isometry between two metric spaces $(X,\dist)$ and $(X',\dist')$ is a map $\Phi:X\to X'$ such that $|\dist(x_0,x_1) - \dist'(\Phi(x_0),\Phi(x_1))|<\eps$ for any $x_0, x_1 \in X$ and $X'=\bigcup_{x \in X} B_\eps(\Phi(x))$. A sequence of  pointed metric spaces $\{(X_\alpha,\dist_\alpha,o_\alpha)\}_\alpha$ converges in the pointed Gromov-Hausdorff topology (pGH for short) to another pointed metric space $(X,\dist,o)$ if there exist two sequences $\{R_\alpha\}_\alpha, \{\eps_\alpha\}_\alpha \subset (0,+\infty)$ such that $R_\alpha \uparrow +\infty$, $\eps_\alpha \downarrow 0$, and, for any $\alpha$, an $\eps_\alpha$-isometry $\Phi_{\alpha}:B_{R_\alpha}(o_\alpha)\to B_{R_\alpha}(o)$ such that $\Phi_{\alpha}(o_\alpha)=o$. We denote this by $$(X_\alpha,\dist_\alpha,o_\alpha) \stackrel{pGH}{\longrightarrow} (X,\dist,o).$$

\emph{Pointed measured Gromov-Hausdorff convergence.} Let us assume that the spaces $\{(X_\alpha,\dist_\alpha,o_\alpha)\}_\alpha, (X,\dist,o)$ are equipped with Radon measures $\{\mu_\alpha\}_\alpha, \mu$ respectively. Then the sequence of pointed metric measure spaces $\{(X_\alpha,\dist_\alpha,\mu_\alpha,o_\alpha)\}_\alpha$ converges to $(X,\dist,\mu,o)$ in the pointed measured Gromov-Hausdorff topology (pmGH for short) if there exist two sequences $\{R_\alpha\}_\alpha, \{\eps_\alpha\}_\alpha \subset (0,+\infty)$ such that $R_\alpha \uparrow +\infty$, $\eps_\alpha \downarrow 0$, and, for any $\alpha$, an $\eps_\alpha$-isometry $\Phi_{\alpha}:B_{R_\alpha}(o_\alpha)\to B_{R_\alpha}(o)$ such that $\Phi_{\alpha}(o_\alpha)=o$ and $$(\Phi_{\alpha})_{\#}\mu_\alpha \weakto \mu,$$ where we recall that 
  $(\Phi_{\alpha})_{\#}\mu_\alpha \weakto \mu$ means that for any $\varphi\in \cC_c(X)$,
  $$\lim_{\alpha} \int_{X_\alpha} \varphi\circ\Phi_\alpha\di\mu_\alpha= \int_{X} \varphi\di\mu.$$
   We denote this by $(X_\alpha,\dist_\alpha,\mu_\alpha,o_\alpha) \stackrel{pmGH}{\longrightarrow} (X,\dist,\mu,o)$.

  \emph{Precompactness results.} Gromov's well-known precompactness theorem yields the following.
  \begin{prop}\label{precompmGH} For any $R>0$ and $\upkappa,\upeta \ge 1$, the space of pointed proper geodesic metric measure spaces  $(X,\dist,\mu,o)$ satisfying 
  \begin{equation}\label{condcompact1}
(X,\dist,\mu)\text{ is $\upkappa$-doubling at scale }R,
\end{equation}
\begin{equation}\label{condcompact1'}
 \upeta^{-1}\le \mu\left(B_R(o)\right)\le \upeta
\end{equation}
is compact in the pointed measured Gromov-Hausdorff topology, i.e.~for every sequence of pointed proper geodesic metric measure spaces $\left\{(X_\alpha, \dist_\alpha, \meas_\alpha, o_\alpha) \right\}_\alpha$ satisfying \eqref{condcompact1} and \eqref{condcompact1'}, there is a subsequence $B\subset A$ and  a pointed proper geodesic metric measure space  $(X,\dist,\mu,o)$ satisfying \eqref{condcompact1} and \eqref{condcompact1'} such that 
$$(X_\beta, \dist_\beta, \meas_\beta, o_\beta) \stackrel{pmGH}{\longrightarrow} (X,\dist,\mu,o).$$ 
\end{prop}

We point out that Gromov's precompactness theorem is usually stated for complete, locally compact, length metric spaces \cite[Proposition 5.2]{Gromov}, but the Hopf-Rinow theorem ensures that these assumptions are equivalent to being proper and geodesic.

Remark that the condition \eqref{condcompact1'} is stated for the radius $R$ but the doubling condition implies that when $\underline{R}>0$ then there is some $\upeta \ge 1$ such that  \eqref{condcompact1'} holds if and only if there is some $\underline{\upeta}\ge 1$, depending only on $\upeta, R,\underline{R}$ and $\upkappa$ such that :
$$\underline{\upeta}^{-1}\le \mu\left(B_{\underline{R}}(o)\right)\le\underline{\upeta}.$$
  
Note that the doubling condition is stable with respect to multiplication of the measure by a constant factor. Therefore,  if $\left\{(X_\alpha, \dist_\alpha, \meas_\alpha, o_\alpha) \right\}_\alpha$  is a sequence of pointed proper geodesic metric measure spaces satisfying \eqref{condcompact1} but not \eqref{condcompact1'}, we may rescale each measure $\mu_\alpha$ into $m_\alpha \mu_\alpha$ for some $m_\alpha>0$ in such a way that the sequence $\left\{ m_\alpha^{-1}\meas_\alpha(B_R(o_\alpha))+ m_\alpha\meas^{-1}_\alpha(B_R(o_\alpha))\right\}_\alpha$  is bounded; then $\left\{(X_\alpha, \dist_\alpha, m_\alpha^{-1}\meas_\alpha, o_\alpha) \right\}_\alpha$ admits a pmGH convergent subsequence. We can choose $m_\alpha=\meas_\alpha(B_R(o_\alpha))$, for instance.

\emph{Tangent cones of doubling spaces.} We recall the classical definition of a tangent cone.

\begin{D} 
Let $(X, \dist,\mu)$ be a metric measure space and $x \in X$. The pointed metric space $(Y,\dist_Y, x)$ is a \textbf{tangent cone} of $X$ at $x$ if there exists a sequence $\{\eps_\alpha\}_{\alpha \in A} \subset(0,+\infty)$ such that $\eps_\alpha \downarrow 0$ and 
$$(X, \eps_\alpha^{-1}\dist, x) \stackrel{pGH}{\longrightarrow} (Y, \dist_Y, x);$$
it can always be equipped with a limit meaure $\mu_Y$ such that, up to a subsequence,
\begin{equation}\label{eq:convtc}
(X, \eps_\alpha^{-1}\dist,\mu(B^\dist_{\eps_\alpha}(y))^{-1}\mu, x) \stackrel{pmGH}{\longrightarrow} (Y, \dist_Y,\mu_Y, x).
\end{equation}
The pointed metric measure space $(Y, \dist_Y,\mu_Y, x)$ is then called a \textbf{measured tangent cone} of $X$ at $x$. If $(Y,\dist_Y, \mu_Y,y)$ is a measured tangent cone of $X$ at $x$ and $y\in Y$, we refer to a measured tangent cone $(Z,\dist_Z, \mu_Z,y)$ of $Y$ as $y$ as an \textbf{iterated measured tangent cone} of $X$. 
\end{D}

\begin{rem}
We often use $(X_x,\dist_x,\mu_x,x)$ to denote a measured tangent cone of $(X,\dist,\mu)$ at $x$.
\end{rem}

As well-known,  on a geodesic proper metric measure space $(X,\dist,\mu)$ that is $\upkappa$-doubling at scale $R$ for some $\upkappa \ge 1$ and $R \in (0,+\infty)$,  the existence of measured tangent cones at any point $x$ is guaranteed and any of these measured tangent cones is $\upkappa$-doubling. Indeed, for any $\eps >0$, the rescaled space $(X,\eps^{-1}\dist,\mu)$ is $\upkappa$-doubling at scale $R/\eps$. Hence when $\eps\le R$, the   the rescaled space $(X,\eps^{-1}\dist,\mu)$ is $\upkappa$-doubling at scale $1$. Hence Proposition \ref{precompmGH} applies to the rescaled spaces $\{(X,\eps^{-1}\dist,\mu(B_\eps(x))^{-1}\mu,x)\}_{\eps >0}$ and yields the existence of measured tangent cones which are $\upkappa$-doubling at any scale $S\ge 1$.

When for some $m>0$ the space $(X,\dist,\mu)$ additionally satisfies a (local) $m$-Ahlfors regularity condition,  i.e.~for each $\rho>0$ there exists $c_\rho>0$ such that for any $x \in X$, any $r \in (0,1)$ and $y\in B_\rho(x)$,
$$ c_\rho r^m\le \mu(B_r(y))\le r^m/c_\rho,$$
then it is convenient to rescale the measure by $\eps^{-m}$ to study measured tangent cones.  In this case,  the tangent measures  are only changed by a multiplicative constant positive factor. 

\subsubsection{Convergence of points and functions}A natural way to formalize the notions of convergence of points and functions defined on varying spaces is the following.  We let $\{(X_\alpha, \dist_\alpha, \meas_\alpha, o_\alpha)\}_\alpha, (X,\dist,\mu,o)$ be proper pointed metric measure spaces such that 
\begin{equation}\label{eq:conv}(X_\alpha, \dist_\alpha, \meas_\alpha, o_\alpha) \stackrel{pmGH}{\longrightarrow} (X,\dist,\mu,o).\end{equation} As the $\eps_\alpha$-isometries between $X_\alpha$ and $X$ are not unique (they can be composed for instance with isometries of $X_\alpha$ or $X$), we make a specific choice by using the following characterization:

\begin{chara}\label{chara}
The pmGH convergence \eqref{eq:conv} holds if and only if there exist $\{R_\alpha\}_\alpha, \{\eps_\alpha\}_\alpha \subset (0,+\infty)$ with $R_\alpha \uparrow +\infty$, $\eps_\alpha \downarrow 0$ and $\eps_\alpha$-isometries $\Phi_\alpha:B_{R_\alpha}(o_\alpha) \to B_{R_\alpha}(o)$  such that:
\begin{enumerate}
\item $\Phi_\alpha(o_\alpha)=o$,
\item $(\Phi_{\alpha})_{\#}\mu_\alpha \weakto \mu$.
\end{enumerate}
\end{chara}
From now on and until the end of this section, we work with the notations provided by this characterization.\\
 
\emph{Convergence of points.}
Let $x_\alpha \in X_\alpha$ for any $\alpha$ and $x \in X$ be given. We say that the sequence of points $\{x_\alpha\}_\alpha$ converges to $x$ if $\dist(\Phi_\alpha(x_\alpha),x) \to 0$.  We denote this by $x_\alpha \to x$.

\emph{Uniform convergence.} Let $u_\alpha \in \cC(X_\alpha)$ for any $\alpha$ and $u \in \cC(X,\mu)$ be given. We say that the sequence of functions $\{u_\alpha\}_\alpha$ converges uniformly on compact sets to $u$ if $\|u_\alpha - u \circ \Phi_\alpha\|_{L^\infty(B(o_\alpha,R))} \to 0$ for any $R>0$. It is easy to show the following useful criterion for uniform convergence on compact sets.
\begin{prop}\label{prop:criUC} Let $u_\alpha  \in \cC(X_\alpha)$ for any $\alpha$ and $u\in \cC(X)$ be given.  Then $\left\{u_\alpha\right\}_\alpha$ converges uniformly on compact sets to $u$ if and only if $u_\alpha(x_\alpha) \to u(x)$ whenever $x_\alpha \to x$.
\end{prop}  
In case $\varphi_\alpha\in \cC_c(X_\alpha)$ for any $\alpha$ and $\varphi\in \cC_c(X)$, we write 
$$\varphi_\alpha\stackrel{\cC_c}{\longrightarrow} \varphi$$ if there is $R>0$ such that
$\supp \varphi_\alpha\subset B_R(o_\alpha)$ for any $\alpha$ large enough and if $\{\varphi_\alpha\}_\alpha$ converges uniformly to $\varphi$.  When the spaces $\{(X_\alpha, \dist_\alpha, \meas_\alpha)\}_\alpha, (X,\dist,\mu)$ are all $\upkappa$-doubling at scale $R$, then for every $\varphi\in \cC_c(X)$ we can build functions  $\varphi_\alpha\in \cC_c(X_\alpha)$ such that $\varphi_\alpha\stackrel{\cC_c}{\longrightarrow} \varphi$: see Proposition \ref{prop:approx} in the Appendix. 

\emph{Weak $L^p$ convergence.} Let $p\in (1,+\infty)$.  Let $f_\alpha\in L^p(X_\alpha,\mu_\alpha)$ for any $\alpha$ and $f \in L^p(X,\mu)$ be given.  We say that the sequence of functions $\{f_\alpha\}_{\alpha}$ converges weakly in $L^p$ to $f$, and we note $$f_\alpha \stackrel{L^p}{\weakto}f,$$ if $\sup_\alpha \|f_\alpha\|_{L^p}<+\infty$ and
$$\varphi_\alpha\stackrel{\cC_c}{\longrightarrow} \varphi\quad  \Longrightarrow \quad \int_{X_\alpha}\varphi_\alpha f_\alpha\di\mu_\alpha=\int_{X}\varphi f\di\mu.$$
We have the following compactness result:
\begin{prop}If $\sup_\alpha \|f_\alpha\|_{L^p}<+\infty$, then there exists a subsequence $B \subset A$ and $f \in L^p(X,\mu)$ such $f_\beta \stackrel{L^p}{\weakto}f$.\end{prop}

\emph{Strong $L^p$ convergence and duality.} Let $p\in (1,+\infty)$.  Let $f_\alpha\in L^p(X_\alpha,\mu_\alpha)$ for any $\alpha$ and $f \in L^p(X,\mu)$ be given.  We say that the sequence of functions $\{f_\alpha\}_{\alpha}$ converges strongly in $L^p$ to $f$, and we note $$f_\alpha \stackrel{L^p}{\to}f,$$ if $f_\alpha \stackrel{L^p}{\weakto}f$ and $\|f_\alpha\|_{L^p} \to \|f\|_{L^p}$. For every $f\in L^p(X,\mu)$, we can build functions $f_\alpha\in L^p(X_\alpha,\mu_\alpha)$ converging to $f$ strongly in $L^p$: this follows from approximating $f$ with functions $\{f_i\} \subset \cC_c(X)$,  approximating each $f_i$ with functions $f_{i,\alpha} \subset \cC_c(X_\alpha)$ as mentioned before, and using a diagonal argument. 

Moreover there is a duality between weak convergence in $L^p$ and strong convergence in $L^q$ when $p$ and $q$ are conjugate exponent, as detailed in the next proposition.
\begin{prop}Let $p,q \in (1,+\infty)$ be satisfying $1/p+1/q=1$. Consider  $f_\alpha\in L^p(X_\alpha,\mu_\alpha)$ for any $\alpha$ and $f\in L^p(X,\mu)$. Then
\begin{itemize}
\item $f_\alpha \stackrel{L^p}{\to}f$ if and only if $\varphi_\alpha \stackrel{L^q}{\weakto}\varphi\Longrightarrow  \int_{X_\alpha}\varphi_\alpha f_\alpha\di\mu_\alpha=\int_{X}\varphi f\di\mu,$
\item $f_\alpha \stackrel{L^p}{\weakto}f$ if and only if $\varphi_\alpha \stackrel{L^q}{\to}\varphi\Longrightarrow  \int_{X_\alpha}\varphi_\alpha f_\alpha\di\mu_\alpha=\int_{X}\varphi f\di\mu.$
\end{itemize} 
\end{prop}

\emph{Convergence of bounded operators.}
When $B_\alpha\colon L^2(X_\alpha,\mu_\alpha)\rightarrow L^2(X_\alpha,\mu_\alpha)$ for any $\alpha$ and $B\colon L^2(X,\mu)\rightarrow L^2(X,\mu)$ are bounded linear operators, we say that $\{B_\alpha\}_\alpha$ converges weakly to $B$ if 
$$f_\alpha \stackrel{L^2}{\weakto}f \Longrightarrow B_\alpha f_\alpha \stackrel{L^2}{\weakto}Bf$$
and that $\{B_\alpha\}_\alpha$ converges strongly to $B$ if $$f_\alpha \stackrel{L^2}{\to}f \Longrightarrow B_\alpha f_\alpha \stackrel{L^2}{\to}Bf.$$
By duality, $\{B_\alpha\}_\alpha$ converges weakly to $B$ if and only if the sequence of the adjoint operators $\{B^*_\alpha\}_\alpha$ converges strongly to the adjoint operator $B^*$. In particular, if the operators $B_\alpha$ and $B$ are all self-adjoint, weak and strong convergences are equivalent.

\emph{Convergence in energy.} When each metric measure space is endowed with a Dirichlet form so that $\{(X_\alpha,\dist_\alpha,\mu_\alpha,\cE_\alpha)\}_\alpha, (X,\dist,\mu,\cE)$ are Dirichlet spaces, we can similarly  define convergence in energy of functions.  Let $f_\alpha\in \cD(\cE_\alpha)$ for any $\alpha$ and $f \in \cD(\cE)$ be given.  We say that the sequence $\{f_\alpha\}_\alpha$ converges weakly in energy to $f$, and we note 
$$
f_\alpha \stackrel{\mE}{\weakto}f,
$$
if  $f_\alpha \stackrel{L^2}{\weakto}f$ and $\sup_\alpha \cE_\alpha(f_\alpha)<+\infty$.  We say that $\{f_\alpha\}_\alpha$ converges strongly in energy to $f$, and we note
$$
f_\alpha \stackrel{\mE}{\to}f
$$
if  it converges weakly in energy to $f$ and additionally $$f_\alpha \stackrel{L^2}{\to}f\text{ and }\lim_\alpha \cE_\alpha(f_\alpha)=\cE(f).$$ Using the non negative selfadjoint operator $L_\alpha$ (resp. $L$) associated to $\cE_\alpha$ (resp.  to $\cE$), we have
$$f_\alpha \stackrel{\mE}{\weakto}f \quad \Longleftrightarrow\quad (1+L_\alpha)^{\frac 12} f_\alpha \stackrel{L^2}{\weakto}(1+L)^{\frac 12}f$$
and $$f_\alpha \stackrel{\mE}{\to}f  \quad \Longleftrightarrow \quad (1+L_\alpha)^{\frac 12} f_\alpha \stackrel{L^2}{\to}(1+L)^{\frac 12}f.$$

\begin{rem}\label{convLploc}
All the above definitions have also a localized version where each function $f_\alpha$ is defined only on a ball centered at $o_\alpha$ with a fixed radius. For instance for a given $\rho>0$, if $f_\alpha\in L^p(B_\rho(o_\alpha))$ for any $\alpha$ and $f\in  L^p(B_\rho(o))$, we say that the sequence $\{f_\alpha\}$ converges weakly to $f$ in $L^p(B_\rho)$, and we note
$$
f_\alpha \stackrel{L^p(B_\rho)}{\weakto} f,
$$
provided
\begin{itemize}
\item $\sup_\alpha \int_{B_\rho(o_\alpha)} |f_\alpha|^p\di\mu_\alpha<\infty,$
\item for any sequence $\{\varphi_\alpha\}_\alpha$ where $\varphi_\alpha\in \cC_c\left(B_\rho(o_\alpha)\right)$ for any $\alpha$ and any $\varphi\in \cC_c\left(B_\rho(o)\right)$, 
$$
\varphi_\alpha \stackrel{\cC_c}{\to} \varphi \quad \Rightarrow \quad \lim_\alpha  \int_{X_\alpha}\varphi_\alpha f_\alpha\di\mu_\alpha=\int_{X}\varphi f\di\mu.$$
\end{itemize}
Similarly, we define $L^p_{loc}$ convergence of functions through pmGH convergence of spaces in the following way: if $f_\alpha \in L_{loc}^p(X_\alpha,\mu_\alpha)$ for any $\alpha$ and  $f \in L_{loc}^p(X,\mu)$, we say that the sequence $\{f_\alpha\}_\alpha$ converges weakly to $f$ in $L_{loc}^p$, and we note $$f_\alpha \stackrel{L_{loc}^p}{\weakto}f,$$ if for any $\rho>0$, $$\left.f_\alpha\right|_{B_\rho(o_\alpha)} \stackrel{L^p(B_\rho)}{\weakto}\left.f\right|_{B_\rho(o)}.$$
 \end{rem}

\subsubsection{Mosco convergence}

We recall the following notion of convergence that was introduced by U.~Mosco in \cite{Mosco} for quadratic forms. We formulate it if for Dirichlet forms as this is sufficient for our purposes.

\begin{D}
Let $\{(X_\alpha,\cT_\alpha,\mu_\alpha,\cE_\alpha)\}_\alpha, (X,\cT,\mu,\cE)$ be Dirichlet spaces. We say that the sequence of Dirichlet forms $\{\cE_\alpha\}_\alpha$ converges to $\cE$ in the Mosco sense if the two following properties hold:\begin{enumerate}
\item for any sequence $\{u_\alpha\}_\alpha$ where $u_\alpha \in \mathcal{D}(\cE_\alpha)$ for any $\alpha$ and any $u \in \mathcal{D}(\cE)$,
$$
u_\alpha \stackrel{L^2}{\weakto} u \quad \Longrightarrow \quad \cE(u) \le \liminf_\alpha \cE(u_\alpha),
$$
\item for any $u \in \mathcal{D}(\cE)$ there exists $u_\alpha \in \mathcal{D}(\cE_\alpha)$ for any $\alpha$ such that $u_\alpha  \stackrel{\mE}{\to} u.$\end{enumerate}
\end{D}
Mosco convergence of Dirichlet forms is equivalent to the convergence of many related objects: this follows from \cite[Theorem 2.4]{KuwaeShioya}. Recall that for a sequence of self-adjoint operators $\{B_\alpha\}_\alpha$ weak and strong convergence are equivalent.

\begin{prop}\label{prop:equivmosco}
Let $\{(X_\alpha,\dist_\alpha,\mu_\alpha,\cE_\alpha)\}_\alpha, (X,\dist,\mu,\cE)$ be metric Dirichlet spaces.  For any $\alpha$ let $L_\alpha$ (resp.~$L$) be the non-negative self-adjoint operator associated with $\cE_\alpha$ (resp.~$\cE$) and let $(P_t^\alpha)_{t>0}$ (resp. ~$(P_t)_{t>0}$) be the generated semi-group.Then the following statements are equivalent:
\begin{enumerate}
\item $\cE_\alpha \to \cE$ in the Mosco sense,

\item there exists $t>0$ such that the sequence of bounded operators $\{P_t^\alpha\}_\alpha$ strongly/
weakly converges to $P_t$,

\item for all $t>0$ the sequence of bounded operators $\{P_t^\alpha\}_\alpha$ strongly/
weakly converges to $P_t$,

\item the sequence of operators $\{\xi(L_\alpha)\}$ strongly converges to $\xi(L)$ for any smooth function $\xi : [0,+\infty) \to \setR$ with $\supp \xi \subset [0,R]$ for some $R>0$,

\item the sequence of operators $\{\xi_\alpha(L_\alpha)\}$ strongly converges to $\xi(L)$ for any sequence $\{\xi_\alpha : [0, +\infty) \rightarrow \R$\} of continuous functions vanishing at infinity which converges to a continuous function $\xi:[0,+\infty)\rightarrow \R$ vanishing at infinity. 
\end{enumerate}
\end{prop}
\begin{D}\label{def:MGH} Let  $\{(X_\alpha,\dist_\alpha,\mu_\alpha,\cE_\alpha,o_\alpha)\}_\alpha, (X,\dist,\mu,\cE,o)$ be pointed metric Dirichlet spaces. We say that the sequence $\{(X_\alpha,\dist_\alpha,\mu_\alpha,\cE_\alpha,o_\alpha)\}_\alpha$ converges to $(X,\dist,\mu,\cE,o)$ in the  \textbf{pointed Mosco-Gromov-Hausdorff} sense if
$$(X_\alpha,\dist_\alpha,\mu_\alpha,o_\alpha)\stackrel{pmGH}{\rightarrow}(X,\dist,\mu,o)\quad \text{and} \quad \text{$\cE_\alpha \to \cE$ in the Mosco sense}.$$ We note
$$(X_\alpha,\dist_\alpha,\mu_\alpha,\cE_\alpha,o_\alpha)\stackrel{pMGH}{\to}(X,\dist,\mu,\cE,o).$$
\end{D}
\subsection{A compactness result for Dirichlet spaces}
The next theorem is a key tool in our analysis. It was already observed by A. Kasue \cite[Theorem 3.4]{KasueSurvey} and extended \cite[Theorem 5.2]{KuwaeShioya} in showing that the limit space is regular and strongly local.  We refer to Section D in the Appendix for the proof.

\begin{theorem}\label{th:improvedKS} Let $\upkappa,\eta \ge 1$, $\upgamma>0$ and $R\in (0,+\infty]$ be given.  Assume that $\{(X_\alpha,\dist_{\cE_\alpha},\mu_\alpha,o_\alpha,\cE_\alpha)\}_{\alpha \in A}$ is a sequence of complete $\mathrm{PI}_{\upkappa, \upgamma}(R)$ pointed Dirichlet spaces such that for all $\alpha$,
$$ \eta^{-1}\le \mu_\alpha\left(B_R(o_\alpha)\right)\le \eta.$$ Then there exist a complete pointed metric Dirichlet space $(X,\dist,\mu,o,\cE)$ and a subsequence $B\subset A$ such that $\{(X_\beta,\dist_{\cE_\beta},\mu_\beta,o_\beta, \cE_\beta)\}_{\beta\in B}$ Mosco-Gromov-Hausdorff converges to $(X,\dist,\mu,o,\cE)$; moreover, $\cE$ is regular, strongly local, the intrinsic pseudo-distance $\dist_\cE$ is a distance, and there is a constant $c\in (0,1]$ depending only on $\upkappa$ and $ \upgamma$ such that:
$$c \dist_\cE\le \dist\le \dist_\cE.$$ Moreover, the space $(X,\dist_\cE,\mu,\cE)$ is $\mathrm{PI}_{\upkappa,\upgamma'}(R)$ for some constant $\upgamma'\ge 1$.  Furthermore, if $H_\beta$ (resp.~$H$) is the heat kernel of $(X_\beta,\dist_\beta,\mu_\beta,\cE_\beta)$ (resp.~$(X,\dist,\mu,o,\cE)$) for any $\beta$, then for any $t>0$ \begin{equation}\label{IKS2}
H_\beta(t,\cdot,\cdot) \to H(t,\cdot,\cdot) \qquad \text{uniformly on compact sets,}
\end{equation}
where we make an implicit use of the obvious convergence $$(X_\beta \times X_\beta,\dist_\beta\otimes \dist_\beta,(o_\beta,o_\beta)) \stackrel{pGH}{\longrightarrow} (X\times X,\dist\times \dist,(o,o)).$$
\end{theorem}

If $(X,\dist_\cE,\mu,\cE)$ is a strongly local and regular Dirichlet space satisfying a $R$-scale-invariant Poincaré inequality for some $R>0$, it is not difficult to check that for any $x \in X$ and any $\rho>0$, the rescaled quadratic form $\cE_\rho:=\rho^2\mu(B_\rho(x))^{-1} \cE$ is a strongly local and regular Dirichlet form on $(X,\dist_\cE,\mu_\rho:=\mu(B_\rho(x))^{-1}\mu)$ such that $\dist_{\cE_\rho}=\rho^{-1}\dist_\cE$, and the space $(X,\dist_{\cE_\rho},\mu_\rho,\cE_\rho)$ satisfies a $(R/\rho)$-scale-invariant Poincaré inequality. This observation coupled with the previous theorem leads to the next result.

\begin{cor}\label{cor:Diritangent}
Let $(X,\dist_\cE,\mu,\cE)$ be a complete $\mathrm{PI}_{\upkappa, \upgamma}(R)$ Dirichlet space and $x\in X$. If  $(X_x,\dist_x,\mu_x,x)$ is a measured tangent cone of $(X,\dist_\cE,\mu)$ at $x$.  Then it can endowed with a strongly local and regular Dirichlet form $\cE_x$ such that $\dist_{\cE_x}$ is a distance bi-Lispchitz equivalent to $\dist_x$ and the space $(X_x,\dist_{\cE_x},\mu_x,\cE_x)$ is $\mathrm{PI}_{\upkappa, \upgamma'}(\infty)$ for some $\upgamma'$. Moreover, there exists a sequence $\{\rho_\alpha\} \subset (0,+\infty)$ such that $\rho_\alpha \to 0$ and
\[
(X,\dist_{\cE_{\rho_\alpha}},\mu_{\rho_\alpha},x,\cE_{\rho_\alpha}) \stackrel{pMGH}{\longrightarrow} (X_x,\dist_{\cE_x},\mu_x,x,\cE_x)
\]
where $\mu_{\rho_\alpha}:=\mu(B_{\rho_\alpha}(x))^{-1}\mu$ and $\cE_{\rho_\alpha}:=\rho_\alpha^2\mu(B_{\rho_\alpha}(x))^{-1} \cE$ for any $\alpha$.
\end{cor}
Note that different sequences could lead to different Dirichlet form on the same  measured tangent cone.

\subsection{Dirichlet spaces satisfying an $\RCD$ condition} Let us conclude these preliminaries with some facts concerning the Riemannian Curvature Dimension condition $\RCD^*(K,n)$, where $K \in \setR$, in the setting of Dirichlet spaces.

The Cheeger energy \cite{Cheeger} of a metric measure space $(X,\dist,\mu)$ is the convex and $L^2(X,\mu)$-lower semicontinuous functional $\Ch : L^2(X,\mu) \to [0,+\infty]$ defined by
\begin{equation}\label{eq:relax2}
\Ch(f)=\inf_{f_n \to f} \left\{ \liminf\limits_{n \to +\infty} \int_X |\nabla f_n|^2 \di \meas \right\}
\end{equation}
for any $f \in L^2(X,\meas)$, where the infimum is taken over the set of sequences $\{f_n\}_n \subset L^2(X,\meas) \cap \Lip(X)$ such that $\|f_n - f\|_{L^2(X,\meas)} \to 0$. We set
\[
H^{1,2}(X,\dist,\meas):=\cD(\Ch)=\{\Ch<+\infty\}
\]
and call $H^{1,2}(X,\dist,\meas)$ the Sobolev space of $(X,\dist,\meas)$.  A suitable diagonal argument shows that for any $f \in H^{1,2}(X,\dist,\meas)$ there exists a unique $L^2$-function $|df|$ called \textit{minimal relaxed slope} of $f$ such that
$$
\Ch(f)=\int_X |d f|^2 \di \meas
$$
and $|df| = |dg|$ $\meas$-a.e.~on $\{f = g\}$ for any $g \in H^{1,2}(X,\dist,\meas)$.  Moreover, this function $|d f|$ coincides $\mu$-a.e.~with the local Lipschitz function of $f$ in case $f$ is locally Lipschitz.

There is no reason a priori for the Cheeger energy to be a Dirichlet form or even a quadratic form.  In this respect,  we provide the next definition and the subsequent proposition which are taken from  \cite{GigliMAMS} and \cite{AmbrosioGigliSavareDuke}.

\begin{D}
A Polish metric measure space $(X,\dist,\mu)$ is called infinitesimally Hilbertian if $\Ch$ is a quadratic form.
\end{D}

\begin{prop}
Let $(X,\dist,\mu)$ be an infinitesimally Hilbertian space.  Then $H^{1,2}(X,\dist,\meas)$ endowed with the norm $\|\cdot\|_{H^{1,2}} = \sqrt{ \|\cdot\|_{L^2} + \Ch(\cdot)}$ is a Hilbert space.  Moreover, the Cheeger energy $\Ch$ is a strongly local and regular Dirichlet form; its carré du champ operator takes values in the set of absolutely continuous Radon measures, and for any $f_1, f_2 \in H^{1,2}(X,\dist,\meas)$,
\[
\langle d f_1, d f_2 \rangle := \frac{\di \Gamma(f_1,f_2)}{\di \meas}=\lim_{\epsilon\to 0}\frac{|d (f_1+\epsilon f_2)|^2-|d f_1|^2}{2\epsilon}
\]
in $L^1(X,\meas)$. In particular,  $\di \Gamma(f)=|d f|^2 \di \mu$ for any $f \in H^{1,2}(X,\dist,\meas)$.
\end{prop}

When $(X,\dist,\meas)$ is infinitesimally Hilbertian,  we call Laplacian of $(X,\dist,\meas)$ the non-negative, self-adjoint operator associated to $\Ch$, and we denote it by $\Delta$.  We also write $(e^{-t\Delta})_{t\ge 0}$ for the semi-group generated by $\Delta$.

For the scope of our work, we must know under which conditions does the Dirichlet form $\cE$ of a Dirichlet space $(X,\dist,\mu,\cE)$ coincide with the Cheeger energy of $(X,\dist,\mu)$.  The next result brings us such a condition in the context of PI Dirichlet spaces; it follows from \cite[Th.~4.1]{KoskelaShanZhou}.

\begin{prop}\label{KSZ}
Let $(X,\dist_\cE,\mu,\cE)$ be a $\PI$ Dirichlet space.  Assume that for some $T>0$ there exists a locally bounded function $\kappa : [0,T] \to [0,+\infty)$ such that $\liminf_{t \to 0} \kappa(t)=1$ and
\begin{equation}
\int_X \varphi \di \Gamma(P_tu) \le \kappa(t) \int_X P_t\varphi \di \Gamma(u)
\end{equation}
for all $u \in \cD(\cE)$,  nonnegative $\varphi \in \cD(\cE) \cap \cC_c(X)$ and $t \in [0,T]$.  Then $\Ch=\cE$.
\end{prop}

To perform our analysis in Sections 5 and 6,  we need some results from the theory of spaces satisfying a Riemannian Curvature Dimension condition $\RCD^*(K,n)$, where $K \in \setR$ is fixed from now on.  In our setting, the original formulation of the $\RCD(K,n)$ and $\RCD^{*}(K,n)$ conditions based on optimal transport \cite{Sturm2006II,LottVillani,BacherSturm,AmbrosioGigliSavareDuke, GigliMAMS} is less relevant than the one provided by a suitable combination of \cite{AGS15} and \cite[Section 5]{EKS}.  In a general framework which covers our needs, these two articles discuss on how to recover a distance $\dist$  from a measure space $(X,\mu)$ equipped with a suitable Dirichlet energy $\cE$ in such a way that $(X,\dist,\mu)$ is an $\RCD^*(K,n)$ space.  When particularized to our context, these works provide us with the following definition.

\begin{D}
A $\PI$ Dirichlet space $(X,\dist_\cE,\mu,\cE)$ is called an $\RCD^*(K,n)$ space if and only if there exists $T>0$ such that for any $t \in (0,T)$ and $f \in \cD(\cE)$,
\begin{equation}\label{BL(K,n)}
\tag{BL(K,n)}
\frac 12 (P_t f^2 - (P_tf)^2) \geq  I_K(t) |\nabla P_tf|^2 +J_K(t)\frac{(\Delta P_t f)^2}{n}
\end{equation}
holds in a weak sense,  namely against any non-negative test function $\varphi \in \cC^0_c(X)\cap \cD(\cE)$,
where $I_K(t):=(e^{2K}t-1)/2K$ and $J_K(t):= (e^{2Kt}-2Kt-1)/4K^2$ for any $K\neq 0$ and $I_0(t)=t$, $J_0(t)=t^2/2$.
\end{D}

\begin{rem}
Inequality \ref{BL(K,n)} is one of many equivalent forms of an estimate due to Bakry and Ledoux \cite{BakryLedoux} which is equivalent to the well-known Bakry-\'Emery condition \cite[Cor.~2.3]{AGS15}.
\end{rem}

To conclude this section, let us consider an $\RCD(0,n)$ space $(X,\dist,\mu)$ -- we point out that in case $K=0$, the $\RCD(K,n)$ and $\RCD^*(K,n)$ conditions are equivalent.  The Bishop-Gromov theorem for $\RCD(0,n)$ spaces (known even for the broader class of $\CD(0,n)$ spaces, see \cite[Th.~30.11]{Villani}) ensures that for any $x \in X$ the volume ratio $r \mapsto \mu(B(x,r))/r^n$ is non-increasing, hence we can define the \emph{volume density} at a point $x$ as follows. 

\begin{D}\label{volumedensity}
Let $(X, \dist, \mu)$ be an $\RCD(0,n)$ space. Then the volume density at $x \in X$ is defined as
\[
\vartheta_X(x):=\lim\limits_{r \to 0} \frac{\mu(B(x,r))}{\omega_n r^n} \in (0,+\infty]\, \cdot
\]
\end{D}

Note that without any particular assumption $\vartheta_X(x)$ may be infinite. On this matter, N. ~Gigli and G. ~De Philippis introduced in \cite{DPG} an important definition.

\begin{D}
\label{def:wnc}
We say that an $\RCD(0,n)$ space $(X,\dist,\mu)$ is weakly non-collapsed if the volume density $\vartheta_X(x)$ is finite for all $x\in X$. 
\end{D}

Weakly non-collapsed $\RCD(0,n)$ spaces are important to us because they enjoy the so-called volume-cone-implies-metric-cone property.  To state this latter,  we must recall a couple of definitions.

If $(Z,\dist_Z)$ is a metric space, then the metric cone over $Z$ is the metric space $(C(Z),\dist_{C(Z)})$ defined in the usual way, see e.g.~\cite[Section 3.6.2.]{BuragoBuragoIvanov}.

\begin{D}
We say that a metric measure space $(X,\dist,\mu)$ is a \textbf{$\alpha$-metric measure cone} with vertex $x \in X$ for some $\alpha \ge 1$ if there exists a metric measure space $(Z,\dist_Z,\mu_Z)$ and an isometry $\varphi:X \to C(Z)$ sending $x$ to the vertex of the cone $C(Z)$ and such that
\[
\di(\varphi_\# \mu)(r,z) = \di r \otimes r^{\alpha-1}\di \mu_Z.
\]
\end{D}

Here is the property of weakly non-collapsed $\RCD(0,n)$ spaces \cite[Theorem 1.1]{DPG} we use in a crucial way in our analysis.

\begin{prop}
\label{prop:DPG}
Let $(X,\dist,\meas)$ be a weakly non-collapsed $\RCD(0,n)$ space such that the function $(0,+\infty) \ni r \mapsto \meas(B_r(x))/r^n$ is constant for some $x \in X$. Then $(X,\dist,\meas)$ is a $n$-metric measure cone with vertex $x$.
\end{prop}

\section{Kato limits}

Let $(M^n, g)$ be a closed Riemannian manifold. The $\|\cdot\|_{1,2}$-closure $H^{1,2}(M)$ of the space of smooth functions and the distributional Sobolev space $W^{1,2}(M)$ coincide and do not depend on the metric $g$, see e.g.~\cite[Chapter 2]{Hebey}, hence we indifferently use both notations in the rest of the article. Moreover, since the space of smooth functions is $\|\cdot\|_{1,2}$-dense in the one of Lipschitz functions, we also have $H^{1,2}(M,\dist_g,\nu_g)=H^{1,2}(M)$, and the Cheeger energy $\Ch_g$ of $(M,\dist_g,\nu_g)$ coincide with the usual Dirichlet energy defined by
$$\cE (u,v)= \int_M g(\nabla u, \nabla v) \di \nu_g, \quad \cE(u)= \int_M |\nabla u|^2 \di \nu_g,$$ for any $u,v \in W^{1,2}(M)$. 
As well-known, $\Ch_g$ is a strongly local and regular Dirichlet form with core $C_0^{\infty}(M)$ and associated operator the Laplacian $\Delta_g$. Moreover, $\dist_{\Ch_g}$ is a distance that coincides with $\dist_g$. We denote by $H:  \R_+ \times M \times M \rightarrow \R$ the heat kernel of $\Ch_g$ which we call heat kernel of $(M,g)$, and by $(P_t)_{t>0}$ the associated semi-group. For any $x \in M$, we define
$$ \rho(x)= \inf_{\substack{v \in T_xM,\\ g_x(v,v)=1}} \Ric_x(v,v) \quad \text{and} \quad \Ricm(x)=  \max\{ -\rho(x), 0\}.$$
For all $t>0$ we introduce the following quantity:
\begin{equation}
\label{def:KatoL1}
\kato= \sup_{x \in M}\int_0^t\int_M H(s,x,y)\Ricm(y) \di \nu_g(y) \di s.
\end{equation}
This quantity is defined more generally as $\mbox{k}_t(V)$ for a Borel function $V$, where $\Ricm$ is replaced by $V$. Then $V$ is said to be in the \emph{contractive Dynkin class} when $\Kato_t(V) < 1$ and in the \emph{Kato class} if $\mbox{k}_t(V)$ tends to 0 as $t$ goes to zero (see for example \cite[Chapter VI]{Guneysu}). In our case, since the manifold is compact, $\Ricm$ always belongs to the Kato class. 

We point out that $\mbox{k}_t(M^n,g)$ has a useful scaling property given by
\begin{equation}
\label{scalingKatoL1}
\forall \eps, t>0, \quad \mbox{k}_t(M^n, \eps^{-2}g)= \mbox{k}_{\eps^2 t}(M^n,g),
\end{equation}
It is an easy consequence of the scaling property of the heat kernel: if $H_\eps$ is the heat kernel of $(M^n, \eps^{-2}g)$, then $H_\eps(s,x,y)=\eps^nH(\eps^2s,x,y)$ for all $s>0$ and $x,y \in M$.

In the following, we consider the next uniform bounds for sequences of closed smooth manifolds. 

\begin{D}
\label{def:bounds}
Let $\{(M_\alpha^n, g_\alpha)\}_{\alpha \in A}$ be a sequence of closed manifolds.  We say that $\{(M_\alpha^n, g_\alpha)\}_{\alpha \in A}$ satisfies
\begin{itemize}
\item[•] a \textbf{uniform Dynkin bound} if there exists $T >0$ such that 
\begin{equation}
\label{UD}
\tag{UD}
\sup_{\alpha} \Kato_T(M_\alpha, g_\alpha) \leq \frac{1}{16n};
\end{equation}
\item[•] a \textbf{uniform Kato bound} if there exists a non-decreasing function $f:(0,T] \to \R_+$ such that $f(t) \to 0$ when $t \to 0$ and for all $t \in (0,T]$
\begin{equation}
\label{UniformKato}
\tag{UK}
\sup_\alpha \Kato_t(M_\alpha, g_\alpha) \leq f(t);
\end{equation}
\item[•] a \textbf{strong uniform Kato bound} if there exist a non-decreasing function $f:(0,T] \to \R_+$, $T, \Lambda>0$ such that for all $t \in (0,T]$
\begin{equation}
\label{StrongKato}
\tag{SUK}
\sup_{\alpha} \Kato_t(M_\alpha, g_\alpha) \leq f(t), \quad \int_{0}^T \frac{\sqrt{f(s)}}{s}\di s \leq \Lambda. 
\end{equation}
\end{itemize}
\end{D}

We observe that obviously a Kato bound implies a Dynkin bound, and a strong Kato bound implies a Kato bound. Indeed, if $f$ is as in \eqref{StrongKato}, then for any $t \in (0,T]$ we have
$$\sqrt{f(t)}\leq \Lambda \log\left(\frac{T}{t}\right)^{-1}.$$
Therefore $f$ tends to zero when $t$ goes to zero. Without loss of generality, we can always assume that the function $f$ is bounded by $(16n)^{-1}$. 

\begin{rem}
\label{rem:scaling}
The scaling property of $\mbox{k}_t(M,g)$ ensures that the previous bounds are preserved when rescaling the metrics by a factor $\eps^{-2}$ for $\eps \in (0,1)$.
Indeed, for any $t >0$ and $\eps \in (0,1)$
$$\mbox{k}_t(M^n, \eps^{-2}g) = \mbox{k}_{\eps^2 t}(M^n,g) < \kato.$$
\end{rem}

\subsection{Dynkin limits}
In this section, we prove a pre-compactness result for sequences of manifolds satisfying a uniform Dynkin bound. We start by proving that a uniform Dynkin bound leads to  a uniform volume estimate and a uniform Poincaré inequality.

\begin{prop}
\label{prop:UnifPI}
Let $(M^n,g)$  be a closed Riemannian manifold, and $T>0$. Assume
\begin{equation*}
\Kato_T(M^n,g) \leq \frac{1}{16n}
\end{equation*}
and set $\nu:=e^2n$. Then there exists $\uptheta\ge 1$ and $\upgamma>0$ depending only on $n$ such that 
\begin{itemize}
\item[1.] for any $x\in M$ and $0<s<r\le \sqrt{T}$, 
$$\frac{\nu_g(B_r(x))}{\nu_g(B_s(x))}\le \uptheta \left(\frac{r}{s}\right)^\nu,$$
\item[2.] for any ball $B\subset M$ with radius $r\le \sqrt{T}$  and any $\varphi\in \cC^1(B)$,
$$\int_B (\varphi-\varphi_B)^2 \di \nu_g\le \upgamma r^2 \int_B |d\varphi|^2 \di \nu_g.$$
\end{itemize}
In particular, $(M^n,\dist_g,\nu_g,\Ch_g)$ is a $\mbox{PI}_{\upkappa, \upgamma}(\sqrt{T})$ Dirichlet space for $\upkappa=2^{\nu}\uptheta$. 
\end{prop}

\begin{rem} The previous proposition is a minor variation of \cite[Propositions 3.8 and  3.11]{C16} where similar estimates were shown for balls with radii lower than $\diam(M)/2$ but with constants that additionally depended on $\diam(M)$.
\end{rem}

\begin{proof}
\textbf{Step 1.} Observe that $\nu>2$. We begin with proving the following Sobolev inequality: there exists $\uplambda>0$ depending only on $n$ such that for any ball $B\subset M$ with radius $r\le  \sqrt{T}$ and any $\varphi \in \cC_c^1(B)$,
\begin{equation}\label{eq:Sobeta}
\left(\int_B |\varphi|^{\frac{2\nu}{\nu-2}}\di \nu_g\right)^{1-\frac 2\nu}\le \uplambda\frac{ r^2}{(\nu_g(B))^{\frac2\nu}}\left[ \int_B |d\varphi|^2 \di \nu_g+\frac{1}{r^2} \int_B |\varphi|^2 d \nu_g\right].
\end{equation}
To this aim, take $r \in (0,\sqrt{T})$, $x \in M$ and $y \in B_r(x)$. From \cite[Theorem  3.5]{C16}, we know that there exists  $c_n>0$ depending only on $n$ such that for any $s \in (0,r^2/2]$,
$$e^{-s/r^2}H(s,y,y)\le H(s,y,y) \le \frac{c_n}{\nu_g(B_r(x))} \left(\frac{r^2}{s}\right)^{\nu/2}  \,\cdot$$
Moreover, since the function $s\mapsto H(s,y,y)$ is non-increasing, for any $s > r^2/2$,
$$e^{-s/r^2} H(s,y,y)\le e^{-s/r^2} H(r^2/2,y,y)\le e^{-s/r^2} \frac{c_n}{\nu_g(B_r(x))} \, 2^{\nu/2}.$$
As the function $\xi \mapsto e^{-\xi}\xi^{-\nu/2}$ is bounded from above on $[1/2,+\infty)$ by some constant $c_n'>0$ depending only on $n$, we get
$$e^{-s/r^2}H(s,y,y)\le c_n' \left(\frac{r^2}{s}\right)^{\nu/2} \frac{c_n}{\nu_g(B_r(x))} \, 2^{\nu/2} .$$
Setting $c_n'':=\max(c_n,c_n')$, we obtain for any $s>0$
$$e^{-s/r^2}H(s,y,y)\le c_n'' \frac{r^\eta}{\nu_g(B_r(x))} \frac{1}{s^{\eta/2}}\, \cdot$$
In particular the heat kernel of the operator $\Delta+1/r^2$ acting on $L^2(B,\mu)$ with Dirichlet boundary condition satisfies the same estimate and one deduces the Sobolev inequality \eqref{eq:Sobeta} from a famous result of Varopoulos \cite[Section 7]{V}.

\textbf{Step 2.} Let us prove 1. To this aim, we follow the argument of \cite{A,Csmf}: for given $0<s<r\le \sqrt{T}$, apply the Sobolev inequality \eqref{eq:Sobeta} in the case $B=B_r(x)$ and $\varphi=\mathrm{dist}(\cdot, M\setminus B_s(x))$ to obtain
\begin{equation}\label{eq:iter}
\left(\frac{s}{2}\right)^2\left(\nu_g(B_{\frac s2}(x))\right)^{1-\frac2\nu}\le 2 \uplambda\frac{ r^2}{(\nu_g(B_r(x))^{\frac2\nu}}\nu_g(B_s(x)).
\end{equation}
Set $\Theta(\tau):= \nu_g(B_{\tau}(x))/\tau^\eta$ for any $\tau>0$ and use elementary manipulations to turn \eqref{eq:iter} into
\begin{equation}\label{eq:iter2}
\Theta(s/2)^{1-2/\nu} \le \Lambda \Theta(s) \qquad \text{with $\Lambda=\frac{2^{\nu+1}\uplambda}{\Theta(r)^{2/\nu}}$} \, \cdot
\end{equation}
Iterating, we get for any positive integer $\ell$
\[
\Theta(s/2^\ell)^{(1-2/\nu)\ell} \le \Lambda^{\sum_{k=0}^{\ell-1} (1-2/\nu)^k}\Theta(s).
\]
As $\lim_{\ell\to \infty}\Theta(s/2^\ell)^{(1-\frac 2\nu)^\ell}=1$ and $\sum_{k=0}^{+\infty} (1-2/\nu)^k=\nu/2$ we obtain
\[
1 \le \Lambda^{\nu/2}\Theta(s)= \frac{[2^{\nu+1}\uplambda]^{\nu/2}}{\Theta(r)}\Theta(s)
\]
which is $1.$~with $\uptheta=[2^{\nu+1}\uplambda]^{\nu/2}$.

Let us now prove 2. We recall a classical result (see e.g.~\cite[Th.~2.7]{HS}): a metric measure space $(X,\dist,\mu)$ doubling at scale $\sqrt{T}$ equipped with a strongly local and regular Dirichlet form $\cE$ with heat kernel $H$ satisfies a Poincaré inequality at scale $\sqrt{T}$ if and only if there exists $c,C,\eps_1,\eps_2>0$ such that for all $x \in M$ and $t \in (0,\eps_1 \sqrt{T})$,
\begin{itemize}
\item[(i)] $H(t,x,x)\le C \mu(B_{\sqrt{t}}(x))^{-1}$,
\item[(ii)] $c\mu(B_{\sqrt{t}}(x))^{-1} \le \inf\{ H(t,x,y) : y \in B_{\eps_2 \sqrt{t}}(x)\}$.
\end{itemize}
In our context, (i) and (ii) hold with $\eps_1=\eps_2=1$ and $C,c$ depending only on $n$: indeed, (i) is a direct consequence of \cite[Theorem  3.5]{C16} while (ii) follows from the same argument as in the proof of \cite[Proposition 3.11]{C16} based on a method from \cite{CG}.
\end{proof}

The previous proposition ensures that we can apply Theorem \ref{th:improvedKS} to a sequence of manifolds satisfying a uniform Dynkin bound \eqref{UD}, and obtain the following pre-compactness result. 

\begin{cor}
\label{converge1}
Let $T>0$ and $\{(M_\alpha, g_\alpha)\}_{\alpha \in A}$ be a sequence of closed manifolds satisfying the uniform Dynkin bound \eqref{UD}. For all $\alpha \in A$ let $o_\alpha \in M_\alpha$ and set 
$$\mu_{g_\alpha}=\frac{\nu_{g_\alpha}}{\nu_g(B_{\sqrt{T}}(x))}, \ \cE_{g_\alpha} = \int_{M_\alpha}|du|^2 \di \mu_\alpha,$$
for all $u \in \cC^1(M_\alpha)$. Then there exist a $\mbox{PI}_{\upkappa,\upgamma}(\sqrt{T})$ Dirichlet space $(X,\dist,\mu,o,\cE)$ and a subsequence $B \subset A$ such that $\{(M_\beta, \dist_{g_\beta}, \mu_\beta, o_\beta, \cE_\beta)\}_{\beta \in B}$ Mosco-Gromov-Hausdorff converges to $(X,\dist,\mu,o,\cE)$.
\end{cor}

\begin{D}
A metric Dirichlet space $(X,\dist,\mu,o,\cE)$ is called a \textbf{Dynkin limit} space if it is obtained as the Mosco-Gromov-Hausdorff limit of a sequence of closed Riemannian manifolds satisfying a uniform Dynkin bound. 
\end{D}

\begin{rem}\label{rem:tgConeDynkin}
For any $T>0$, the class of Dynkin limit spaces obtained as limits of manifolds satisfying \eqref{UD} is closed under pointed Mosco-Gromov-Hausdorff convergence: this follows from a direct diagonal argument.

As a consequence, tangent cones of Dynkin limit spaces equipped with their intrinsic distance are Dynkin limit spaces too. Indeed, if $(X,\dist, \mu, o,\cE)$ is the pointed Mosco-Gromov-Hausdorff limit of pointed manifolds $\{(M_\alpha, g_\alpha,o_\alpha)\}_{\alpha \in A}$ satisfying \eqref{UD} and $(X_x, \dist_x, \mu_x, x, \cE_x)$ is a tangent cone at $x \in X$ provided by Corollary \ref{cor:Diritangent},  then this latter can be written as the limit of a sequence of rescaled manifolds $\{(M_\beta, \eps_\beta^{-1}\dist_{g_\beta}, \nu_{g_\beta}(B_{\eps_\beta}(x_\beta))^{-1}\nu_\beta, x_\beta, \eps_{\beta}^2\nu_{g_\beta}(B_{\eps_\beta}(x_\beta))^{-1} \cE_\beta)\}_{\beta \in B}$ where $B \subset A$, $x_\beta\in M_\beta$ and $\eps_\beta>0$ for any $\beta \in B$,  and $x_\beta \to x$, $\eps_\beta\to0$. 

Another consequence is that if $\{z_\alpha\}_{\alpha \in A}$ belongs to a compact set of $X$ and $\{\eps_\alpha\} \subset (0,+\infty)$ satisfies $\eps_\alpha \to 0$, then there exists a subsequence $B \subset A$ such that the sequence $\{(X, \eps_\beta^{-1}\dist,\mu(B_{\eps_{\beta}}(z_{\beta}))^{-1}\mu,z_\beta,\eps_\beta^2 \mu(B_{\eps_\beta}(z_\beta))^{-1}\cE)\}_\beta$ converges to a Dynkin limit space $(Z,\dist_Z,\mu_Z,z,\cE_Z)$ which may also be written as the limit of a sequence of rescaled manifolds $\{(M_{\beta}, \eps_{\beta}^{-1}\dist_{g_{\beta}}, \nu_{g_{\beta}}(B_{\eps_{\beta}}(x_{\beta}))^{-1}\nu_{g_\beta}, x_{\beta}, \nu_{g_\beta}(B_{\eps_\beta}(x_\beta))^{-1}\cE_{\beta})\}_{\beta}$, with $x_{\beta} \in M_{\beta}$ for any $\beta \in B$ and $\dist(z_{\beta},\Phi_{\beta}(x_{\beta})) \to 0$, where $\Phi_{\beta}$ is as in Characterization \ref{chara}.
\end{rem}

\subsection{Kato limits}

In this section we consider manifolds with a uniform Kato bound. In this case, some better properties can be proved for the distance in the limit and for tangent cones (see the next remark and Proposition \ref{converge3}). Thanks to the previous pre-compactness result we can give the following definition. 

\begin{D}
A Dirichlet space $(X,\dist,\mu,o,\cE)$ is called a \textbf{Kato limit} space if it is obtained as a Mosco-Gromov-Hausdorff limit of manifolds with a uniform Kato bound \eqref{UniformKato}.  
\end{D}

\begin{rem}
\label{rem:KatoPI}
A Kato limit space is obviously a $\mbox{PI}_{\upkappa,\upgamma}(\sqrt{T})$ Dirichlet space for any $T>0$ such that $f(T) \le 1/(16n)$.
\end{rem}

\begin{rem}
\label{rem:tgConeKato}
As in the case of Dynkin limits, tangent cones of Kato limits are Kato limits as well. Not only, if $(X,\dist, \mu, o, \cE)$ is a Kato limit and $(X_x,\dist_x,\mu_x,x,\cE_x)$ is a tangent cone at $x\in X$ provided by  Corollary \ref{cor:Diritangent}, then this latter is a limit of rescaled manifolds $\{(M_\alpha, \eps_\alpha^{-1}\dist_{g_\alpha},\nu_{g_\alpha}(B_{\eps_\alpha}(x_\alpha))^{-1}\nu_\alpha, x_\alpha, \eps_\alpha^2 \nu_{g_\alpha}(B_{\eps_\alpha}(x_\alpha))^{-1}\cE_\alpha))\}$ such that for all $t >0$
$$\mbox{k}_t(M_\alpha, \eps_\alpha^{-2}g_\alpha) \to 0 \mbox{ as } \alpha \to \infty.$$
Indeed, we have $\mbox{k}_t(M_\alpha, \eps_\alpha^{-2}g_\alpha) = \mbox{k}_{\eps_\alpha^2 t}(M_\alpha, g_\alpha)\leq f(\eps_\alpha^2 t) \to 0 \mbox{ as } \alpha \to \infty.$ This observation also applies to spaces $(Z,\dist_Z,\mu_Z,z,\cE_Z)$ obtained as limits of rescalings of $X$ centered at varying but convergent points, as considered in Remark \ref{rem:tgConeDynkin}.
\end{rem}

As a consequence of Theorem \ref{th:improvedKS}, any Dynkin limit space $(X,\dist,\mu,o,\cE)$ satisfies $\dist \le \dist_\cE$.  But for Kato limit spaces, this inequality turns out to be an equality. To prove this fact, we need the following Li-Yau inequality which was proved in  \cite[Proposition 3.3]{C16}.

\begin{prop}
Let $(M^n,g)$  be a closed Riemannian manifold, and $T>0$.  Assume
\begin{equation*}
\Kato_T(M^n,g) \leq \frac{1}{16n}\, \cdot
\end{equation*}
If $u$ is a positive solution of the heat equation on $[0, T]\times M$, then for any $(t,x)\in [0,T]\times M$,
\begin{equation}\label{eq:LYC}
e^{-8\sqrt{n\Kato_{t}(M,g)}}\,\frac{|du|^{2}}{u^{2}}-\frac{1}{u}\frac{\partial u}{\partial t} \le \frac{n}{2t}e^{8\sqrt{n\Kato_{t}(M,g)}}.
\end{equation} 
\end{prop}

We are now in a position to prove the following.

\begin{prop}\label{converge3}
Let $(X,\dist,\mu,o,\cE)$ be a Kato limit space. Then $\dist = \dist_\cE$.
\end{prop}

\begin{proof}
We only need to prove $\dist \ge \dist_\cE$. Let $\{(M_\beta^n,g_\beta,o_\beta)\}_{\beta \in B}$ be a sequence of closed Riemannian manifolds satisfying a uniform Kato bound and such that the sequence $\{(M_\beta,\dist_\beta,\mu_\beta,o_\beta,\cE_\beta)\}$ converges in the Mosco-Gromov-Hausdorff sense to $(X,\dist,\mu,o,\cE)$. In particular, there exists  $T>0$ such that $\{(M_\beta^n,g_\beta)\}_{\beta \in B}$ satisfies the uniform Dynkin bound \eqref{UD}. We claim that for any $x,y \in X$, $t\in(0,T)$ and $\theta \in (0,1)$,
\begin{equation}\label{eq:estHeatKernel}
\log\left( \frac{H(\theta t,x,x)}{H(t, x,y)}\right) \le \frac{n}{2}\log(1/\theta)e^{2} +  \frac{\dist^2(x,y)}{4(1-\theta)t}e^{8\sqrt{n f(t)}}\, \cdot
\end{equation}
Let us explain how to conclude from there. Multiply by $-4t$ and apply Varadhan's formula \eqref{eq:varadhan} as $t \to 0$ to get
$$\dist_{\cE}^2(x,y) \leq \frac{\dist^2(x,y)}{1-\theta}\,\cdot$$
The desired inequality follows from $\theta \downarrow 0$.

In the following we thus prove \eqref{eq:estHeatKernel}. Take $\beta \in B$. Let $u$ be a positive solution of the heat equation on $[0, T]\times M_\beta$. Take $x,y \in M_\beta$, $t \in (0,T]$ and $s \in (0,t)$. Let $\gamma:[0,t-s]\rightarrow M_\beta$ be a minimizing geodesic from $y$ to $x$.  For any $\tau \in [0,t-s]$ set
$$\phi(\tau):=\log u\left(t-\tau,\gamma(\tau)\right)$$
and note that $\phi(0)=\log u\left(t,y\right)$ and $\phi(t-s)=\log u\left(s,x\right)$.  Differentiate $\phi$ at $\tau$ and apply the Li-Yau inequality \eqref{eq:LYC} and the simple fact $\Kato_{t-\tau}\leq \Kato_t$ to get the first inequality in the following calculation,  where $u$ and its derivatives are implicitly evaluated at $(t-\tau,\gamma(\tau))$ and where we omit $(M_\beta,g_\beta)$ in the notation $\Kato_{t}(M_\beta,g_\beta)$ for the sake of simplicity:
\begin{equation*}
\begin{split} 
\dot\phi(\tau)&=-\frac{1}{u}\frac{\partial u}{\partial t}+\la \dot\gamma(\tau),d \log u \ra\\
&\le\frac{ne^{8\sqrt{n\Kato_{t-\tau}}}}{2(t-\tau)}-e^{-8\sqrt{n\Kato_{t}}}\frac{|du|^{2}}{u^{2}}+\la \dot\gamma(\tau),d\log u\ra\\
&=\frac{ne^{8\sqrt{n\Kato_{t-\tau}}}}{2(t-\tau)}- \left|e^{-4\sqrt{n\Kato_t}}d\log u - \frac{e^{4\sqrt{n \Kato_t}}}{2}\dot{\gamma}(\tau) \right|^2 +\frac{e^{8\sqrt{n\mbox{k}_{t}}}}{4}|\dot\gamma(\tau)|^{2}\\
&\le\frac{ne^{8\sqrt{n\Kato_{t-\tau}}}}{2(t-\tau)}+\frac{e^{8\sqrt{n\Kato_{t}}}}{4}|\dot\gamma(\tau)|^{2}\\
&\le \frac{ne^{8\sqrt{n\Kato_{t-\tau}}}}{2(t-\tau)}+\frac{e^{8\sqrt{n\Kato_{t}}}\dist_\beta^{2}(x,y)}{4(t-s)^{2}} \,\cdot
\end{split}
\end{equation*}
Hence, when integrating between $0$ and $t-s$ and changing variables in the first term, we obtain:
$$\log\left(\frac{u(s,x)}{u(t,y)}\right) \le \frac{n}{2}\int_s^t e^{8\sqrt{n\Kato_{\tau}}}\frac{d\tau}{\tau}+\frac{e^{8\sqrt{n\Kato_{t}}}\dist_\beta^{2}(x,y)}{4(t-s)} \, \cdot$$ 
Write $s=\theta t$ for some $\theta \in (0,1)$. The uniform Dynkin bound \eqref{UD} allows us to bound $e^{8\sqrt{n\Kato_{\tau}}}$ in the first term of the right-hand side by $e^2$, while the uniform Kato bound lets us bound $e^{8\sqrt{n\Kato_{t}}}$ by $e^{8\sqrt{nf(t)}}$ in the second term. Thus
\[
\log\left(\frac{u(s,x)}{u(t,y)}\right) \le \frac{n}{2} e^{2}\log(t/s)+\frac{\dist_\beta^{2}(x,y)}{4(1-\theta)t} e^{8\sqrt{nf(t)}} .
\]
Choose $s=\theta t$ and $u(\tau,z)=H_\beta(\tau,x,z)$ for any $(\tau,z) \in (0,T) \times M$ to get
\[
\log\left( \frac{H_\beta(\theta t,x,x)}{H_\beta(t, x,y)}\right) \le \frac{n}{2}\log(1/\theta)e^{2} +  \frac{\dist_\beta^2(x,y)}{4(1-\theta)t}e^{8\sqrt{n f(t)}}\, \cdot
\]
The Mosco-Gromov-Hausdorff convergence $(M_\beta,\dist_\beta,\mu_\beta,o_\beta,\cE_\beta) \to (X,\dist,\mu,o,\cE)$ eventually yields \eqref{eq:estHeatKernel}.
\end{proof}

\begin{rem}
\label{rem:plusgen1}
Observe that the previous proof also applies more generally in the case of a sequence of manifolds $\{(M_\alpha,g_\alpha)\}_{\alpha \in A}$ such that there exists a non-decreasing function $f: (0,T) \to \R_+$, tending to $0$ as $t$ goes to 0 and for which
$$\limsup_{\alpha \to \infty} \Kato_t(M_\alpha, g_\alpha) \leq f(t) \ \mbox{ for all } t \in (0,T].$$
In particular, we have $\dist = \dist_{\cE}$ whenever $(X,\dist, \mu,o,\cE)$ is the Mosco-Gromov-Hausdorff limit of a sequence $\{(M_\alpha,g_\alpha)\}_{\alpha \in A}$ such that for some $T>0$ 
$$\lim_{\alpha \to \infty}\Kato_T(M_\alpha, g_\alpha)=0;$$
in this case $f$ is constantly equal to $0$. As a consequence of Remark \ref{rem:tgConeKato}, Proposition \ref{converge3} applies to tangent cones (and rescalings centered at convergent points) of Kato limits.
\end{rem}

\subsection{Ahlfors regularity}
We now discuss volume  estimates for closed Riemannian manifolds $(M^n,g)$ satisfying the strong Kato condition integral bound 
\begin{equation}
\label{eq:hyp_intKato}
\Kato_T(M^n,g)\le \frac{1}{16 n}\text{ and }
\int_0^{ T}\frac{\sqrt{\Kato_s(M^n,g)}}{s} \di s \leq \Lambda, 
\end{equation}
for some $ T, \Lambda >0$. This condition was also considered in \cite{C16}. 

The proof of \cite[Proposition 3.13]{C16} gives the following volume estimate which improves the one given in Proposition \ref{prop:UnifPI}. 

\begin{prop}
\label{prop:IntKatoBG}
Let $(M^n,g)$ be a closed Riemannian manifold satisfying \eqref{eq:hyp_intKato} for some $ T, \Lambda>0$. Then there exists a constant $C_n>0$ depending only on $n$ such that for all $x \in X$ and  $0\le r\le s\le \sqrt{T}$ then 
\begin{equation}\label{AR1}
 \nu_g(B_r(x))\le C_n^{\Lambda+1} r^n \text{ and } \frac{\nu_g(B_s(x))}{\nu_g(B_r(x))}\le C_n^{\Lambda+1} \left(\frac{s}{r}\right)^n.
\end{equation}
\end{prop}
\begin{proof}
The upper bound is proven in \cite[Page 3144]{C16}. The other estimate is a consequence of the proof of Proposition \ref{prop:UnifPI} and of the estimate (see again \cite[Page 3144]{C16}): $0<s<t\le\sqrt{T}:$
$$ s^{\frac n 2} H(s,x,x) \le  C_n^{\Lambda+1} t^{\frac n 2} H(t,x,x),$$
 that holds for any $x\in M$ and $0<s<t\le\sqrt{T}.$
\end{proof}
Then using the doubling properties [Proposition \ref{prop:doubling}-ii)], we get the following uniform local Ahlfors regularity result:
\begin{cor}\label{cor:IntKatoBG}Let $(M^n,g)$ be a closed Riemannian manifold satisfying \eqref{eq:hyp_intKato} for some $ T, \Lambda>0$. Then there exists a constant $C_n>0$ depending only on $n$ such that for all $o,x \in X$ and  $0\le r\le \sqrt{T}$ then 
\begin{equation}\label{AR2}
 \frac{ \nu_g\left(B_{\sqrt{T}}(o)\right) }{T^{\frac n 2}} \le C_n^{(\Lambda+1)\frac{\dist(x,o)}{\sqrt{T}}}\  \frac{\nu_g(B_{r}(x))}{r^n}\end{equation}
\end{cor}

\begin{rem}
A metric measure space $(X,\dist,\mu)$ for which there exists $C_1,C_2>0$ such that $C_1 \le \mu(B_r(x))/r^n\le C_2$ for any $x \in X$ and $r>0$ is usually called Ahlfors $n$-regular. Thus \eqref{AR1} and  \eqref{AR2} tell that for any $R>0$  $(B_R(o),\dist_g,\nu_g)$ is Ahlfors $n$-regular with constants depending on $n,\Lambda, R, T$ and $\nu_g(B_{\sqrt{T}}(o))/T^{\frac n2}$.\end{rem}

\subsection{Non-collapsed strong Kato limits}
We introduce a last class of limit spaces that we are going to deal with, that is strong Kato limits with a non-collapsing assumption. In this case, the limit measure carries the local Ahlfors regularity described above. This will be important in proving that tangent cones are metric cones and for our stratification result.  

\begin{D}\label{def:NC} A Dirichlet space $(X,\dist,\mu, o,\cE)$ is called a \textbf{strong Kato limit} if it is obtained as a Mosco-Gromov-Hausdorff limit of pointed manifolds $(M_\alpha, g_\alpha, o_\alpha)$ with a strong uniform Kato bound. It is called \textbf{non-collapsed strong Kato limit} if moreover there exists $v>0$ such that for all $\alpha$
\begin{equation}
\label{NC}
\tag{NC}
\vol_{g_\alpha}(B_{\sqrt{T}}(o_\alpha))\geq vT^{\frac n2}, 
\end{equation}
where $T$ is given in Definition \ref{def:bounds}.
\end{D}

\begin{rem}
\label{rem-AR-ncStrongKatolim}
The convergence of the measure ensures that if the manifolds $(M_\alpha,g_\alpha)$ satisfy a strong uniform Kato bound, then inequalities \eqref{AR1} and \eqref{AR2} passes to the limit. In particular, if $(X,\dist,\mu, o,\cE)$ is a non-collapsed strong Kato limit, then there exists constants $C, \lambda$ such that we have for all $0 < r \leq s \leq \sqrt{T}$ and $x \in X$
$$\mu(B_r(x))\leq C r^n, \quad \frac{\mu(B_s(x))}{\mu(B_r(x))}\leq C \left(\frac s r \right)^{n}, $$
and the lower bound
$$\mu(B_r(x))\geq  v e^{-\lambda \frac{\dist(x,o)}{\sqrt{T}}}r^n.$$
As a consequence, for any $R>0$, $(B_R(o),\dist, \mu)$ is Ahlfors $n$-regular, with constants depending on $n,\Lambda, R, T$ and $v$. 
\end{rem}

\begin{rem}
\label{rem-TgConesSKL}
As in the previous cases, tangent cones of strong Kato limits are strong Kato limits. Under the non-collapsing assumption \eqref{NC}, the previous remark ensures a local Ahlfors $n$-regularity. Then as observed in Section 2, we can consider tangent cones as limits of re-scaled manifolds $\{(M_\alpha, \eps_\alpha^{-1}\dist_{g_\alpha}, \eps_\alpha^{-n}\nu_{g_\alpha},x_\alpha \}$, that is to say that we can replace the re-scaling factor $\nu_{g_\alpha}(B_{\sqrt{T}}(x_\alpha))$ of the measures  by $\eps_\alpha^{-n}$. This sequence of re-scaled manifolds also satisfies the non-collapsing condition. 

This also applies to limits $(Z,\dist_Z,\mu_Z,z,\cE_Z)$ of rescalings of non-collapsed strong Kato manifolds centered at varying but convergent points.
\end{rem}

\subsection{$L^p$-Kato condition}The strong Kato condition is implied by a uniform on the $L^p$-Kato constant for $p >1$.
Introduce
\begin{equation}\label{def:KatoLp}
\Kato_{p,T}(M,g):=\left(\sup_{x \in M}T^{p-1} \int_0^T\int_M H(s,x,y){\Ricm}(y)^p \di\nu_{g}(y) ds\right)^{\frac 1p}.
\end{equation}

When $p>1$, and using H\"older inequality, we obtain
$$\Kato_{t}(M,g)\le\left( \frac{t}{T}\right)^{1-\frac 1p} \Kato_{p,T}(M,g).$$

Hence a sequence  $\{(M_\alpha^n, g_\alpha)\}_{\alpha \in A}$  of closed manifolds satisfying 
$$\sup_{\alpha \in A} \Kato_{p,T}(M,g)<\infty$$  satisfies a strong uniform Kato bound.

As noticed in \cite[Proposition 3.15]{C16}, we can estimate the $L^2$ Kato constant in terms of the $\mathbf{Q}-$curvature.

Recall that if $(M,g)$ is Riemannian manifold of dimension $n\ge 4$, its $\mathbf{Q}-$curvature is defined by:
$$\mathbf{Q}_{g}=\frac{1}{2(n-1)}\Delta \mathrm{Scal}_{g}-\frac{2}{(n-2)^{2}}|\Ric|^{2}+c_{n}\mathrm{Scal}_{g}^{2},$$
where $c_{n}=\frac{n^{3}-4n^2+16n-16}{8n(n-1)^{2}(n-2)^{2}}$.

\begin{prop} Let $(M^{n},g)$ be a closed Riemannian manifold of dimension $n\ge 4$ such that :
$$-\kappa^4\le\mathbf{Q}_{g} \ \mathrm{and}\left|\mathrm{Scal}_{g}\right| \le \kappa^{2},$$
where $\kappa>0$.
Then 
$$\Kato_{2,T}(M,g)\le C(n) \kappa\sqrt{T} \left(1+ \kappa\sqrt{T}\right).$$
\end{prop}

\section{Analytic properties of manifolds with a Dynkin bound}

In this section we develop some analytic tools in the setting of manifolds with Dynkin bound on $\Kato_T(M^n,g)$, that is for which there exists $T>0$ such that 
\begin{equation}
\label{DB}
\tag{D}
\Kato_T(M^n,g) \leq \frac{1}{16n}.
\end{equation}

\subsection{Good cut-off functions} The existence of cut-off functions with suitably bounded gradient and Laplacian is a key technical tool in the theory of Ricci limit spaces (\cite{CheegerColdingI}) and $\RCD^*(K,N)$ spaces (\cite{MondinoNaber}). Our next proposition tells that such functions also exist in the context of manifolds with a Dynkin bound; it also provides an alternative proof to the one given by Cheeger and Colding on manifolds with  Ricci curvature bounded from below.

\begin{prop}\label{prop:coupure}
Let $(M^n,g)$ be a closed Riemannian manifold satisfying \eqref{DB} for some $T>0$. Then for any ball $B_{r+s}(x) \subset M$ there exists a function $\chi \in C^\infty(M)$ such that $0 \le \chi \le 1$ and
\begin{enumerate}
\item $\chi = 1$ on $B_r(x)$,
\item $\chi = 0$ on $M\backslash B_{r+s}(x)$,
\item there exists a constant $C=C(n)>0$ such that
\[
|\nabla \chi| \le \frac{C(n)}{\min(s,\sqrt{T})} \quad \text{and} \quad |\Delta \chi| \le \frac{C(n)}{\min(s^2,T)}\, \cdot
\]
\end{enumerate}
\end{prop}

We start with the following useful consequence of a Kato bound:
\begin{lemma} If $(M^n,g)$ be a closed Riemannian manifold satisfying \eqref{DB} for some $T>0$. Assume that $u\colon M\rightarrow \R$ is a $\Lambda$-Lipschitz function then for any $t\in (0,T]$ and $x\in M$:
$$\left|\nabla e^{-t\Delta}u\right|(x)\le 2\Lambda.$$
\end{lemma}
\begin{proof} It is well known (see \cite[Remark 1.3.2]{CarronRose}, or \cite[Proof of Theorem 1, step i)]{Voigt}, the Kato bound \eqref{DB} implies that for any $t\in (0,T]$
$$\left\|e^{-t(\Delta-\Ricm)}\right\|_{L^\infty\to L^\infty} \le \frac{16n}{16n-1}\le 2.$$
Then the result follows by the domination properties for the Hodge-Laplacian $\vec \Delta=d^*d+dd^*=\nabla^*\nabla+\Ric$ on $1$ forms:
$$\left|\nabla e^{-t\Delta}u\right|=\left|e^{-t\vec\Delta}\nabla u\right|\le e^{-t(\Delta-\Ricm)}|\nabla u|.$$
\end{proof}
The Li-Yau inequality have the following consequence which will be very useful.
\begin{lemma}\label{lem:estintnableH}
 If $(M^n,g)$ be a closed Riemannian manifold satisfying \eqref{DB} for some $T>0$. The heat kernel of $(M^n,g)$ satisfies for any $x\in M$ and $t\in (0,T]:$
  $$\int_M | \nabla_z H(t,x,z)\di\nu_g(z)\le  \sqrt{\frac{e^4n}{2t}}$$ and
  $$ \int_M  \frac{\left|\nabla_z H(t,x,z)\right|^2}{H(t,y,z)}\di\nu_g(z) \le \frac{e^4n}{2t}.$$
\end{lemma}
\begin{proof}
 Using Hölder's inequality and the stochastic completeness of $H$, we get 
\begin{align*}\int_M | \nabla_z H(t,x,z)\di\nu_g(z)&\le  \left( \int_M  \frac{\left|\nabla_z H(t,y,z)\right|^2}{H(t,y,z)}\di \nu_g(z) \right)^{1/2} \left(\int_M H(t,y,z) \di  \nu_g(z)\right)^{\frac 12}\\
&= \left( \int_M  \frac{\left|\nabla_z H(t,y,z)\right|^2}{H(t,y,z)}\di \nu_g(z) \right)^{1/2}.\end{align*} Hence the first estimate follows from the second one's.
By the Li-Yau estimate \eqref{eq:LYC},
\begin{equation}\label{eq:tointegrate}
\frac{\left|\nabla_z H(t,y,z)\right|^2}{H(t,y,z)}\le \frac{e^4n}{2t} H(t,y,z)+e^2\frac{\partial H(t,y,z)}{\partial t} \, \cdot
\end{equation}
Since $$\int_M \frac{\partial H(t,y,z)}{\partial t} \di \nu_g(z)= \frac{\partial}{\partial t}\underbrace{\int_M H(t,y,z) \di \nu_g(z)}_{=1} =0,$$ the second estimate follows.

\end{proof}
\begin{proof}[Proof of Proposition \ref{prop:coupure}]
Take $x \in M$ and $r,s>0$. Let $\rho_0$ be the distance function to $x$ (i.e.~$\rho_0(\cdot)=\dist_g(x,\cdot)$) and set
\[
\rho_t := e^{-t\Delta} \rho_0
\]
for any $t \in (0,T]$. 
Note that from the previous lemma, $\rho_t$ is $2$-Lipschitz.
Then for any $y \in M$,
\begin{align*}
\left|\frac{\partial \rho_t}{\partial t}(y)\right| &=\left|\Delta \rho_t(y)\right| \\
&=\left|\int_M \langle \nabla_z H(t,y,z),\nabla \rho_0(z)\rangle \di z\right|\\
& \le \int_M | \nabla_z H(t,y,z)|\underbrace{|\nabla \rho_0(z)|}_{=1 \, \text{$\nu_g$-a.e.}} \di z \\& \le  e^2 \sqrt{\frac{n}{2t}}
\end{align*}  
Where we used Lemma \ref{lem:estintnableH}. This estimate implies
\[
\left| \rho_t(y)-\rho_0(y)\right|\le e^2\sqrt{2n}\,\sqrt{t}.
\]
Therefore, if $y \in B_r(x)$ then
$$\rho_t(y)\le r+e^2\sqrt{2n}\,\sqrt{t}$$ while if $y \in $ $M\setminus B_{r+s}(x)$ then
$$\rho_t(y)\ge r+s-e^2\sqrt{2n}\,\sqrt{t}.$$
Hence defining $t_0=\left( \frac{s}{4e^2\sqrt{2n}}\right)^2$, we choose $t=t_o$
if $t_o\le T$ and  $t=T$ if $t_o>T$. 
So that \[
\rho_t(y) \le r+\frac{s}{4} \quad \text{if $y \in B_r(x)$}
\]
and
\[\rho_t(y) \ge r+\frac{3s}{4} \quad \text{if $y \in M\backslash B_{r+s}(x)$};
\]

Let $u\colon\R_+\rightarrow \R_+$ be a smooth function such that
\[
u = \begin{cases}
1 & \text{on $[0,1/4]$},\\
0 & \text{on $[3/4,+\infty)$}.
\end{cases}
\]
Set $$\chi(y):=u\left(\frac{\rho_t(y)-r}{s}\right)$$
for any $y \in M$.
Then 
\[
\chi = \begin{cases}
1 & \text{on $B_r(x)$},\\
0 & \text{on $M\backslash B_{r+s}(x)$}.
\end{cases}
\]
Since $$d\chi(y)=\frac{1}{r}u'\left(\frac{\rho_t(y)-r}{s}\right)d\rho_t(y)$$ and
$$\Delta \chi(y)=\frac{1}{s}u'\left(\frac{\rho_t(y)-r}{s}\right)\Delta\rho_t(y)-\frac{1}{s^2}u''\left(\frac{\rho_t(y)-r}{s}\right) |d\rho_t|^2(y)$$
for any $y \in M$, setting $L:=\sup_{\R} |u'|+|u''|$), we get
$$\| d\chi\|_\infty\le  \frac{2L}{s}$$
and
$$\| \Delta d\chi\|_\infty\le L\left( \frac{e^2\sqrt{n}}{s\sqrt{2t}}+ \frac{4}{s^2}\right).$$
\end{proof}

\begin{rem}[Complete Kato manifolds] 
Note that the above proof makes use of an integrated version of the Li-Yau inequality only. In this regard, it would be interesting to study whether the assumption \eqref{DB}  on a complete Riemannian manifold implies
$$\left(\int_M  \frac{\left|\nabla_z H(t,y,z)\right|^2}{H(t,y,z)}\di\nu_g(z)\right)^{\frac 12}\le \frac{C(n)}{\sqrt{t}} \,\cdot$$
This would ensure the existence of good cut-off functions which would in turn provide the Li-Yau inequality, and then make possible the study of limits of complete Riemannian manifolds satisfying the uniform bound \eqref{UD}.
\end{rem}

\subsection{Hessian estimates} Good cut-off functions are particularly relevant to deduce the following powerful Hessian  estimates, that we will use in  Section 8. 

\begin{prop}
\label{prop:HessEst}
Let $(M^n,g)$ be a closed Riemannian manifold satisfying \eqref{DB} for some $T>0$. Then there exists a constant $C=C(n)>0$ such that for any $u \in \cC^\infty(M)$ and any ball $B_r(x) \subset M$,
\begin{equation}\label{eq:hessest}\int_{B_{\frac r2}(x)} |\nabla d u|^2\di \nu_g\le C(n) \int_{B_r(x)}\left[(\Delta u)^2+\frac{1}{\min(r^2,T)} |du|^2\right]\di \nu_g.
\end{equation}
If $u$ is additionally harmonic, then
\begin{equation}\label{eq:hessest2}
\int_{B_{\frac r2}(x)} |\nabla d u|^2\di \nu_g\le \frac{C(n)}{\min(r^2,T)} \int_{B_r(x)}\left| |du|^2-\fint_{B_r(x)}|du|^2\di \nu_g\right|\di \nu_g.
\end{equation}
\end{prop}
 
\begin{proof}
\textbf{Step 1.} Using again \cite[Remark 1.3.2]{CarronRose}, or \cite[Proof of Theorem 1, step i)]{Voigt}, we obtain the following estimate:
\begin{equation}\label{eq:estimate1-1}
\left\|e^{-T(\Delta-2\Ricm)}\right\|_{L^1\to L^1} \le \frac{8n}{8n-1}\le 2.
\end{equation}
\hfill

\textbf{Step 2.} All the integrals in these steps are taken with respect to $\nu_g$, hence we skip the notation $\di \nu_g$ for the sake of brevity. Take $u \in C_c(B_r(x))$. The estimate \eqref{eq:estimate1-1} implies that the bottom of the spectrum of $\Delta-2\Ricm$ is bounded from below by $-\frac{\log 2}{T}$. Thus for any $v \in C^\infty(M)$,
\begin{equation}\label{eq:bottomspectrum}
\int_M \Ricm v^2  \le \frac 12 \int_M \left[ |dv|^2-\frac{\log(2)}{T}v^2 \right].
\end{equation}
Let $\chi$ be a cut-off function as built in Proposition \ref{coupure} such that $\chi = 1$ on $B_{\frac r2}(x)$ and $\chi=0$ on $M \backslash B_r(x)$. Apply Bochner's formula to $u$:
$$|\nabla d u|^2+\frac{1}{2}\Delta |du|^2+\Ric(du,du)=\langle d\Delta u,du\rangle,$$
multiply by $\chi^2$ and integrate over $M$. This gives
\begin{equation}\label{eq:Bochnerchi}
\int_M \chi^2|\nabla d u|^2+\frac12 \int_M \chi^2\Delta |du|^2\le \int_M \Ricm |du|^2\chi^2+\int_M \chi^2\langle d\Delta u,du\rangle .
\end{equation}
We control the second term in the right-hand side of \eqref{eq:Bochnerchi} as follows, using successively integration by parts, the Cauchy-Schwarz inequality, and the elementary fact $2ab \le a^2 + b^2$:
\begin{align}\label{eq:inter1}\int_M \langle d\Delta u,\chi^2du\rangle &=\int_M \chi^2(\Delta u)^2 -\int_M \Delta u \langle d\chi^2,du\rangle \nonumber \\
&\le \int_M \chi^2(\Delta u)^2 +\int_M 2\, \chi |\Delta u |\, | d\chi| |du| \nonumber\\
&\le  2\int_M \chi^2(\Delta u)^2 +\int_M | d\chi|^2 |du|^2 .
\end{align}
Now we control the first term in the right-hand side of \eqref{eq:Bochnerchi} as follows. Thanks to \eqref{eq:bottomspectrum} we have
\[ \int_M \Ricm |du|^2\chi^2\le \frac12 \left(\int_M |\nabla (\chi |du|)|^2 +\frac{\log 2}{T} \int_M  (\chi |du|)^2\right).\]
Since
\begin{align*}
\int_M |\nabla (\chi |du|)|^2 & = \int_M \langle \chi \nabla |du|, \nabla (\chi |du|)\rangle  +  \langle |du|\nabla \chi, \nabla (\chi |du|)\rangle\\
 & = \int_M  \chi^2 | \nabla |du||^2 + \chi |du| \langle \nabla \chi, \nabla |du|\rangle  +  |du|\langle \nabla \chi, \nabla (\chi |du|)\rangle\\
 & = \int_M  \chi^2 | \nabla |du||^2 + \int_M \langle \nabla \chi, \underbrace{\chi |du|  \nabla |du| + |du| \nabla (\chi |du|)}_{=\nabla (\chi |du|^2)} \rangle
\end{align*}
and $| \nabla |du||^2 \le | \nabla du|^2$, we get
\begin{equation}\label{eq:inter2}
\int_M \Ricm |du|^2\chi^2 \le  \frac12 \left(\int_M \chi^2 |\nabla du|^2 +\int_M (\Delta \chi) \chi |du|^2 + \frac{\log 2}{T} \int_M  (\chi |du|)^2\right).
\end{equation}
Combining \eqref{eq:Bochnerchi} with \eqref{eq:inter1} and \eqref{eq:inter2}, we get
$$\frac12 \int_M \chi^2|\nabla d u|^2\le  2 \int_M \chi^2 (\Delta u)^2+\int_M  \left[\, |d\chi|^2+\frac{\log 2}{2T}\chi^2+\frac12 |\Delta \chi|\chi - \frac{1}{2}(\Delta \chi^2)\right]|du|^2$$
which eventually leads to \eqref{eq:hessest} thanks to the properties of $\chi$.

The second estimate \eqref{eq:hessest2} is obtained in a similar way by replacing $\Delta |du|^2$ with $\Delta( |du|^2 - c)$ in \eqref{eq:Bochnerchi}, where $c=\fint_{B_r(x)}|du|^2 \di \nu_g$.
\end{proof}

\subsection{Gradient estimates for harmonic functions} We conclude with the following gradient estimates, that will also be useful in Section 8. 

\begin{lemma}
\label{lem:gradientEst_harmonic}
Let $(M^n,g)$ be a closed Riemannian manifold satisfying \eqref{DB} for some $T>0$, and let $h : B_r(x) \to \setR$ be a harmonic function. Then for some constant $c_n>0$ depending only on $n$:
\begin{enumerate}
\item $\displaystyle \sup_{B_{\frac r2}(x)} |\nabla h|\le c_n^{1+r/\sqrt{T}} \left(  \fint_{B_r(x)} |\nabla h|^2 \di \nu_g   \right)^{1/2}$,
\item $\displaystyle \sup_{B_{\frac r2}(x)} |\nabla h|\le \frac{c_n^{1+r/\sqrt{T}}}{r} \sup_{B_{\frac r2}(x)} |h|$.
\end{enumerate}
\end{lemma}

\begin{proof}
We first proof the result when $r \leq \sqrt{T}$.
Consider the operator 
$$A= \left(\Delta_g +\frac 1T\right)^{-1} \Ricm \, .$$
Assumption \eqref{DB} ensures that 
$$\left\| A \right\|_{L^{\infty} \to L^{\infty}} \leq \frac{1}{16n}\frac{1}{e-1}\leq \frac{1}{8n}.$$
See for example \cite[Lemma 3.18]{C16}. 

The same is true when replacing the Laplacian on $M$ by the Laplacian $\Delta_B$ on the ball $B=B_r(x)$ with the Dirichlet boundary conditions : that is introducing $$A_B= \left(\Delta_B +\frac 1T\right)^{-1} \Ricm \, ,$$ we get $\left\| A_B \right\|_{L^{\infty} \to L^{\infty}}\le  1/(8n).$  As a consequence, if $f=A_B(1)$ then we find a unique  $\varphi \in L^{\infty}(B)$ solving
$$\varphi=A_B\varphi+f.$$ Note that $A_B$ preserves the positivity: $
v\ge 0\Rightarrow A_B v\ge 0.$ Then 
 it is not difficult to show that $\varphi$ satisfies the inequality
$$0 \leq \varphi \leq \frac{1}{8n}\frac{1}{1-\frac{1}{8n}}\leq 1.$$
Moreover by construction $\varphi\in W^{1,2}(B)$ and is zero along $\partial B$.

Consider $J=1+\varphi$. By definition, $J$ solves the equation 
$$\left(\Delta_B+\frac 1T - \Ricm \right)J=\frac 1T,$$
and $1 \leq J \leq 2$. 
Now consider the Laplacian $\Delta_{J^2}$ associated to the quadratic form
$$\cE_{J^2}(\Psi)= \int_B |d\Psi|^2J^2\di \nu_g,$$
on the space $L^2(B,J^2\di \nu_g)$. Then for any $\Psi \in L^2(B,J^2\di \nu_g)$ we have
$$J^{-1}\left(\Delta+\frac{1}{T}-\Ricm\right)(J\Psi)=\Delta \Psi-2J^{-1}\langle dJ,d\Psi\rangle + \frac{\Psi}{J T}=\Delta_{J^2}\Psi + \frac{\Psi}{J T}.$$
When choosing $\displaystyle \Psi= \frac{|dh|}{J}$ we obtain
$$\Delta_{J^2}\left(\frac{|dh|}{J}\right) + \frac{|dh|}{J^2T}=J^{-1}\left(\Delta|dh|+\frac{|dh|}{T}-\Ricm |dh| \right) \leq \frac{|dh|}{JT},$$
where we used that $\Delta |dh|\leq \Ricm |dh|$ because of Bochner inequality. 
We then conclude that 
\begin{equation}
\label{eq-DeltaJ^2}
\Delta_{J^2}\Psi \leq \frac{1}{T}\Psi.
\end{equation}
Now for $\nu=e^2n$, the following Sobolev inequality holds for all $\upvarphi\in C_0^{\infty}(B)$
$$\left(\int_B |\upvarphi|^\frac{2\nu}{\nu-2}\di \nu_g\right)^{1-\frac 2\nu}\le \frac{C(n)r^2}{\nu_g(B)^{\frac 2\nu}}\left[ \int_B|d\upvarphi|^2\di \nu_g+\frac{1}{r^2}\int_B|\upvarphi|^2\di \nu_g\right],$$
this implying the analog Sobolev inequality for the measure $J^2\di \nu_g$:
$$\left(\int_B |\upvarphi|^\frac{2\nu}{\nu-2}J^2\di \nu_g\right)^{1-\frac 2\nu}\le \frac{C(n)r^2}{\nu_g(B)^{\frac 2\nu}}\left[ \int_B|d\upvarphi|^2J^2\di \nu_g+\frac{1}{r^2}\int_B|\upvarphi|^2J^2\di \nu_g\right].$$
Together with inequality \eqref{eq-DeltaJ^2} and De Giorgi-Nash-Moser iteration, this leads to 
$$\sup_{B_{\frac r2}(x)} \Psi\le C(n) \sqrt{ \frac{1}{\nu_g(B)}\int_{B_{3r/4}(x)} \Psi^2 J^2 \di \nu_g}.$$
Since $J$ is bounded between 1 and 2, we then obtain
\begin{equation}\label{eq:J12}
\sup_{B_{\frac r2}(x)}|dh| \le C(n) \sqrt{ \frac{1}{\nu_g(B)}\int_{B_{3r/4}(x)} |dh|^2 \di \nu_g},
\end{equation}
thus the first inequality. 

As for the second inequality, take $\xi \in C_c^\infty(M)$ such that $\xi \equiv 1$ on $B_{3r/4}(x)$ and $|d \xi|^2 \le C/r^2$ for some $C>0$.
The integration by parts formula applied to $\xi h$, that is
$$\int_B |d(\xi h)|^2\di \nu_g=\int_B |d\xi|^2h^2\di \nu_g+\int_B \xi^2h\Delta h \di \nu_g,$$
and the fact that $h$ is harmonic imply
$$\int_B |d(\xi h)|^2 \di \nu_g \leq \frac{C}{r^2}\int_B|h|^2\di \nu_g \le \frac{C}{r^2}\nu_g(B)\sup_B|h|^2.$$
The left-hand side is bounded from below by $\int_{B_{3r/4}(x)} |dh|^2 \di \nu_g$, hence the right-hand side of \eqref{eq:J12} is bounded from above by the previous right-hand side.

The proof when $r >  \sqrt{T}$ follows. Indeed the doubling property implies that 
when $y \in B_{\frac r2}(x)$ we have
$$\nu_g(B_r(x) \leq C_n^{\frac{r}{\sqrt{T}}}\nu_g(B_{\sqrt{T}}(y)).$$

 As a consequence we have:
\begin{align*}
\fint_{B_{\sqrt{T}}(y)}|dh|^2\di \nu_g&  \leq \left(\nu_g(B_{\sqrt{T}}(y)\right)^{-1}\int_{B_r(x)} |dh|^2\di \nu_g \\
& \leq  C_n^{\frac{r}{\sqrt{T}}}\fint_{B_r(x)} |dh|^2\di \nu_g.
\end{align*}
But we have already proved that 
$$|dh|^2(y)\le C\fint_{B_{\sqrt{T}}(y)}|dh|^2\di \nu_g.$$
\end{proof}

\begin{rem}
Notice that whenever $r \leq \sqrt{T}$ we do not need to consider the exponent $r/{\sqrt{T}}$ in the previous estimates. 
\end{rem}

\section{Curvature-dimension condition for Kato limits}

In this section we prove that the Cheeger energy built from the metric measure structure of a Kato limit space $(X,\dist,\mu,o,\cE)$ always coincides with the limit Dirichlet energy $\cE$.  Moreover, we show that tangent cones of a Kato limit space are all $\RCD(0,n)$ spaces, and that they are additionally weakly non-collapsed in case the space is a non-collapsed strong Kato limit. We obtain these two latter statements by establishing the Bakry-Ledoux gradient estimate $\mathrm{BL}(0,n)$ on a specific class of Kato limit spaces, namely those obtained from a sequence of closed Riemannian manifolds $\{(M^n_\alpha, g_\alpha)\}$ for which there exists $T>0$ such that $\Kato_T(M_\alpha,g_\alpha)$ tends to zero as $\alpha \to \infty$.

For these purposes and following \cite[Subsection 2.2]{AGS15},  we define the following quantity $A$ on a strongly local, regular Dirichlet space $(X,\dist, \mu, \cE)$. For all $t>0$ and $u \in \cD(\cE), \varphi \in L^2(X)\cap L^{\infty}(X)$ with $\varphi \geq 0$, we set
\begin{equation}
\label{def_A}
A_t(u, \varphi)(s):= \frac 12 \int_X (P_{t-s}u)^2 P_s \varphi \di \mu
\end{equation}
for any $s\in[0,t]$. As it is shown in \cite[Lemma 2.1]{AGS15}, the function $s \mapsto A_t(u,\varphi)(s)$ is continuous on $[0,t]$ and continuously differentiable on $(0,t]$ with derivative given by
\begin{equation}
\label{def_DerA}
\frac{\di}{\di s} A_t(u,\varphi)(s) = \int_X P_s\varphi \di \Gamma(P_{t-s} u).
\end{equation} 
for any $s \in (0,t]$. Whenever $\varphi$ additionally belongs to $\cD(\cE)$, the map $s \mapsto A_t(u,\varphi)(s)$ is in $C^1([0,t])$ and the previous formula is valid for $s=0$.

\subsection{Differential inequalities}

We first need to prove some differential inequalities on closed manifolds $(M^n,g)$ with a smallness condition on $\kato$.

\begin{theorem}\label{thm:Diffeq}Let $(M^n,g)$ be a closed Riemannian manifold satisfying
\begin{equation}
\label{eq:SK}\tag{SK}
\mbox{\emph{k}}_t(M^n,g) < \frac{1}{8}
\end{equation}
for some $t>0$.  Then for all $v \in W^{1,2}(M)$ with $\Delta_g v\in L^2(M) \cap L^{\infty}(M)$, for any non-negative $\varphi \in W^{1,2}(M)\cap L^\infty(M)$, we have the inequality 
{\small \begin{equation}\label{id3}
\frac 12 \int_M (P_t\varphi v^2 -\varphi (P_tv)^2)\di\nu_g \geq e^{-12\mbox{\tiny{\emph{k}}}_t(M,g)}\left(t\, \int_M\varphi |dP_tv|^2\di\nu_g + \frac{t^2}{n}  \int_M\varphi(\Delta_g P_tv)^2\di\nu_g \right).\end{equation}}
\end{theorem}

In the next subsection,  under the assumption $\Kato_T(M_\alpha,g_\alpha)\to 0$, we aim to pass inequality \eqref{id3} to a limit space in order to get the Bakry-Ledoux gradient estimate  $\mathrm{BL}(0,n)$.  To this aim, it is useful to rewrite this inequality in terms of $A$ and its derivative with respect to $s \in (0,t)$. 
Since on a closed manifold $(M^n,g)$, the derivative of $A$ is simply
$$\frac{\di}{\di s} A_t(u, \varphi)(s)= \int_M |dP_{t-s}u|^2P_s\varphi \di \nu_g,$$
we can rephrase Theorem \ref{thm:Diffeq} as follows. 
\begin{theorem}\label{thm:Aineq}Let $(M^n,g)$ be a closed Riemannian manifold satisfying
\eqref{eq:SK} for some $t>0$. Then for all $v\in W^{1,2}(M)$ with $\Delta_g v\in L^2(X) \cap L^{\infty}(X)$, for any non-negative $ \varphi \in W^{1,2}(M)\cap L^\infty(M)$ and any $s \in (0,t)$,
{ \small \begin{equation}
\label{id4}
A_t(v,\varphi)(t)-A_t(v,\varphi)(s)   \geq e^{-12\mbox{\tiny{k}}_t(M,g)}\left((t-s) \frac{\di}{\di s}A_t(v,\varphi)(s)  + \ \frac{2}{n}  (t-s)^2A_t(\Delta v,\varphi)(s) \right). 
\end{equation}}
\end{theorem}
\noindent Indeed,  with $A$ and its derivative, inequality \eqref{id3} writes as
$$ A_t(v,\varphi)(t)-A_t(v,\varphi)(0)\geq e^{-12\mbox{\tiny{k}}_t(M,g)}\left( t \int_M\varphi |dP_tv|^2\di\nu_g + \ \frac{2}{n}  t^2A_t(\Delta v,\varphi)(t) \right). $$
Now for any $s \in [0,t]$ we also have $\mbox{{k}}_{t-s}(M,g)\le \mbox{{k}}_{t}(M,g)<1/8$ so the previous holds with $t-s$ instead of $t$:
\begin{equation*}\begin{split} A_{t-s}(v,P_s\varphi)(t-s)-A_{t-s}(v,P_s\varphi)(0)&\geq e^{-12\mbox{\tiny{k}}_{t}(M,g)}\left[ (t-s)  \int_M P_s\varphi |dP_{t-s}v|^2\di\nu_g \right.\\
&\left. + \ \frac{2}{n}  (t-s)^2 A_{t-s}(\Delta v,P_s\varphi)(t-s) \right] \\
\end{split}\end{equation*}
and this rewrites easily as inequality \eqref{id4}. 

The proof of Theorem  \ref{thm:Diffeq} relies on a modified version of the function $s \mapsto \frac{\di}{\di s}A_t(u, \varphi)(s)$.  For any closed Riemannian manifold $(M^n,g)$, any $t>0$ and any positive ``gauging'' function $J: [0, t] \times M \rightarrow \R$, we  define
$$B_J(u,\varphi)(s):=\int_M |dP_{t-s}u|^2(P_s\varphi) J_{t-s}\di \nu_g$$
for any $s,u,\varphi$ as above, where $J_{t-s}(\cdot)=J(t-s,\cdot)$. For the sake of simplicity, from now on in this section we denote $P_{\tau}u,P_{\tau}\varphi$ by $u_\tau,\varphi_\tau$ for any $\tau>0$.

\begin{lemma}
Let $(M^n,g)$ be a closed Riemannian manifold, $t>0$ and $J: [0, t] \times M \rightarrow (0,+\infty)$ be a smooth function.  Then for any $\eps>0$,
\begin{equation}
\label{id1}
\begin{split}
\frac{\di}{\di s} B_J(u,\varphi)(s) & \geq \int_M \varphi_s \left( \left(-\Delta_g J - \dot{J}  -\frac{2}{\eps}\frac{|dJ|^2}{J}- 2J\Ric_{\mbox{\tiny{-}}} \right)|du_{t-s}|^2  \right. \\ & \left.  \qquad \qquad \qquad  \qquad \qquad \qquad +2(1-\eps)J\frac{(\Delta_g u_{t-s})^2}{n}\right) \di\nu_g
\end{split}
\end{equation}
for any $s \in [0,t]$, where $\dot{J}$ is a shorthand for $\frac{\di}{\di s} J$.
\end{lemma}

\begin{proof}
When deriving $B_J(u,\varphi)$ with respect to $s$ we obtain: 
\begin{align*}
\frac{\di}{\di s} B_J(u,\varphi)(s) & = \int_M \varphi_s \left( -\Delta_g (J |du_{t-s}|^2) -  \dot{J} |du_{t-s}|^2 + 2J\langle d \Delta_g u_{t-s}, d u_{t-s} \rangle \right) \di\nu_g \\
& = \int_M \varphi_s \left(-(\Delta_g J) |du_{t-s}|^2 - J \Delta_g |du_{t-s}|^2 + 2\langle d J, \nabla |du_{t-s}|^2 \rangle \right. \\ & \left.- \dot{J} |du_{t-s}|^2 +2J\langle d \Delta_g u_{t-s}, d u_{t-s} \rangle \right) \di\nu_g, \end{align*}where we have used  the Leibniz formula
\begin{align*}
 \Delta_g (J |du_{t-s}|^2)&= (\Delta_g J)  |du_{t-s}|^2+J\Delta_g ( |du_{t-s}|^2)-2\langle\nabla |du_{t-s}|^2,dJ\rangle\\
 &= (\Delta_g J)  |du_{t-s}|^2+J\Delta_g ( |du_{t-s}|^2)-4\nabla du_{t-s}( du_{t-s},dJ).
\end{align*}
Then by using Bochner's formula 
$$ -\Delta_g |df|^2 +2\langle d \Delta_g f, df \rangle = 2(|\nabla d f |^2 +\Ric(df, df)) $$
with $f =u_{t-s}$ we obtain
\begin{equation*}
\begin{split}
\frac{\di}{\di s} B_J(u,\varphi)(s) & = \int_M \varphi_s \left( (-\Delta_g J - \dot{J})|du_{t-s}|^2 + 2J (|\nabla d u_{t-s}|^2 + \Ric(d u_{t-s},d u_{t-s})) \right.
\\ & \left. +4\nabla du_{t-s} (d J, du_{t-s} ) \right)\di\nu_g. 
\end{split}
\end{equation*}
By using the fact that, for all $x\in M$, $\Ric_x \geq -\Ricm(x)$ we get the lower bound 
\begin{equation*}
\begin{split}
\frac{\di}{\di s} B_J(u,\varphi)(s) & \geq \int_M \varphi_s \left[ (-\Delta_g J - \dot{J} -2J \Ricm)|du_{t-s}|^2 + 2J |\nabla d u_{t-s}|^2 \right. \\  &\left. + 4\sqrt{J}\nabla du_{t-s}\left( \frac{d J}{\sqrt{J}}, d u_{t-s} \right)  \right]
\end{split}
\end{equation*}
Now, for any $\eps >0$ we have: 
$$ 2\sqrt{J}\nabla du_{t-s}\left( \frac{d J}{\sqrt{J}}, d u_{t-s} \right)    \geq - \left(\eps J|\nabla du_{t-s}|^2 +\frac{1}{\eps}\frac{|d J|^2}{J}|du_{t-s}|^2\right),$$
Therefore we get
\begin{equation*}\begin{split}
\frac{\di}{\di s} B_J(u,\varphi)(s) \geq& \int_M \varphi_s \left[ \left(-\Delta_g J - \dot{J}  -\frac{2}{\eps}\frac{|d J|^2}{J}-2J\Ricm \right)|du_{t-s}|^2\right.\\
&\left. +2(1-\eps)J|\nabla du_{t-s}|^2 \right] \di\nu_g.
\end{split}\end{equation*}
We conclude by using $|\nabla du_{t-s}|^2 \geq (\Delta_g u_{t-s})^2/n$. 
\end{proof}

\begin{lemma}
\label{lemmaJ}
Let $(M^n,g)$ be a closed Riemannian manifold satisfying \eqref{eq:SK} for some $t>0$.  Set $\eps:=4\Kato_t(M,g)$. Then there exists a unique solution $J\colon [0,t]\times M\rightarrow (0,+\infty)$ to the problem
\begin{equation}
\label{eqJ}
\tag{EJ}
\begin{cases}
\displaystyle \Delta_g J + \dot{J}+\frac{2}{\eps}\frac{|d J|^2}{J}+ 2J \Ric_{\mbox{\tiny{-}}} = 0 \\
J(0,x)=1 
\end{cases}
\end{equation}
which satisfies
\begin{equation*}
e^{-4 \Kato_t(M,g)} \leq J \leq 1. 
\end{equation*}
\end{lemma}

\begin{proof}
Consider $\delta = \frac{2}{\eps}-1\ge 1$,  we have 
$$\Delta_g(J^{-\delta})=- \delta J^{-\delta -1}\left( \Delta_g J+  (\delta +1)\frac{|d J|^2}{J}\right)= -\delta J^{-\delta -1}\left( \Delta_g J + \frac{2}{\eps}\frac{|d J|^2}{J}\right).$$
Define $I:=J^{-\delta}$. Then $J$ solves \eqref{eqJ} if and only if $I$ is a solution of 
$$ 
\begin{cases}\Delta_g I + \dot{I} - 2\delta I \Ricm =0 \\
I(0,x)=1
\end{cases}$$
This latter equation is equivalent to the following integral equation:
\begin{equation}
\label{intEq}
I(t,x)=1 +2\delta \int_0^t\int_M H(t-s,x,y)\Ricm(y)I(s,y)\di\nu_g(y)\di s.
\end{equation}
Consider the map $f\in L^{\infty}([0,t]\times M)\mapsto Tf\in L^{\infty}([0,t]\times M)$ defined by 
$$Tf(s,x)=  2\delta \int_0^t\int_M H(t-s,x,y)\Ricm(y)f(s,y)\di\nu_g(y)\di s.
$$

By definition of $\Kato_t(M,g)$ and of $\delta$, the operator norm of $T$ satisfies:
$$\left\| T\right\|_{L^\infty\to L^\infty}\le 2\delta \Kato_t(M,g)=2\left(\frac{2}{\eps}-1\right)\Kato_{t}(M,g).
$$ 
Recall that $\eps=4\Kato_t(M,g)$,  so
$$\left\| T\right\|_{L^\infty\to L^\infty}\le 1 -2\Kato_t(M,g) <  1.$$
As a consequence, the operator $\mathrm{Id}-T$ is invertible and  $I=(\mathrm{Id}-T)^{-1}\mathbf{1}$ is the unique solution of the equation \eqref{intEq} that satisfies 
$$\|I\|_{L^\infty}\le 2.$$  Then the integral equation (\ref{intEq}) implies that 
$$ 1 \leq I  \leq 1+2\delta \Kato_t(M,g)\|I\|_{L^\infty}  \leq 1+4\delta \Kato_t(M,g)\leq e^{4\delta \mbox{\tiny k}_t(M,g)}.$$
Therefore we get $ e^{-4\mbox{\tiny{k}}_t(M,g)} \leq J \leq 1$ for $J = I^{-1/\delta}$,  as we wished. \end{proof}

\begin{cor}
\label{cor:inForCheeger}
Let $(M^n,g)$ be a closed Riemannian manifold satisfying \eqref{eq:SK}  for some $t>0$. Then for all $u \in C^1(M)$, $\varphi \in C^{0}(M)$ with $\varphi \geq 0$ and $\tau\in (0,t]$: ,
\begin{equation}
\label{eq:inForCheeger}
\int_M \varphi \left| d P_\tau u\right|^2 \di\nu_g \leq e^{4\mbox{\tiny{k}}_t(M,g)} \int_M P_\tau\varphi\,  |du|^2\di\nu_g. 
\end{equation}
\end{cor}

\begin{proof} We only need to prove \eqref{eq:inForCheeger} for $\tau=t$ as condition \eqref{eq:SK} and our proof remains true if $t$ is replaced by any $\tau \in (0,t]$. First observe that inequality \eqref{id1} and the lower bound for $J$ given by Lemma \ref{lemmaJ} imply
\begin{align}
\label{ineq-proffDiffIn}
\frac{\di}{\di s} B_J(u,\varphi)(s) & \geq e^{-4 \mbox{\tiny{k}}_t(M,g)} \frac{2(1-4 \Kato_t(M,g))}{n}\int_M \varphi_s (\Delta_g u_{t-s})^2 \di\nu_g  \nonumber \\ 
& \geq \frac{e^{-12\mbox{\tiny{k}}_t(M,g)}}{n} \int_M \varphi_s (\Delta_g u_{t-s})^2 \di\nu_g, 
\end{align}
where we used that $1-x \geq e^{-2x}$ on $\left[ 0, \frac 12 \right]$, with $x=4\Kato_t(M,g)$. In particular the latter inequality yields that 
\begin{equation*}
\frac{\di}{\di s}B_J(u,\varphi)(s) \geq  0. 
\end{equation*}
Therefore, when integrating between $0$ and $t$, we get
\begin{equation*}
\int_M \varphi (P_t|du|^2-J|dP_t u|^2)\di\nu_g \geq 0, 
\end{equation*}
which leads to 
$$\int_M P_t\varphi |du|^2\di\nu_g \geq \int_M J |dP_tu|^2\di\nu_g.$$
Inequality \eqref{eq:inForCheeger} then immediately follows by using the lower bound $J \geq e^{-4\mbox{\tiny{k}}_t(M,g)}$. 
\end{proof}

\begin{rem}\label{eq:lipcontrol}
Corollary \ref{cor:inForCheeger} can be rephrased in the following way: if $(M^n,g)$ is a closed Riemannian manifold satisfying \eqref{eq:SK}, then for any $u \in W^{1,2}(M)$ and any $t>0$,
\[
|d P_tu|^2  \leq e^{4\mbox{\tiny{k}}_t(M,g)} P_t(|du|^2)
\]
holds in the weak sense. Of course, the right-hand side can be bounded from above by $e^{1/4n} P_t(|du|^2)$. Thus, as a direct consequence of $P_t$ being non-negative and sub-Markovian, if $u$ is $\kappa$-Lipschitz, then $P_tu$ is $e^{1/8n}\kappa$-Lipschitz.
\end{rem}

We are now in position to prove Theorem \ref{thm:Diffeq}.

\begin{proof}[Proof of Theorem \ref{thm:Diffeq}]
We consider again inequality \eqref{ineq-proffDiffIn}. By definition of $A$ and since $P_{t-s}\Delta_g u =\Delta_g P_{t-s}u$, we can write it as 
$$\frac{\di}{\di s} B_J(u,\varphi)(s) \geq  e^{-12 \mbox{\tiny{k}}_t(M,g)} \frac{2}{n} A(\Delta_g u,\varphi)(s).$$
Since $A$ is monotone non decreasing in $s$, $A(\Delta_g u, \varphi)(s)$ is bounded from below by its value in $s=0$. Then we get
$$\frac{\di}{\di s} B_J(u,\varphi)(s) \geq e^{-12 \mbox{\tiny{k}}_t(M,g)}\frac{2}{n} \int_M \varphi (\Delta_g u_{t})^2 \di\nu_g.$$
We integrate this latter inequality between $0$ and $t$, so that we get
$$ \int_M (\varphi_t |du|^2-\varphi J_t |du_t|^2) \di\nu_g \geq e^{-12\mbox{\tiny{k}}_t(M,g)}  \frac{2}{n} t\int_M \varphi(\Delta_g u_t)^2\di\nu_g.$$
Using the lower bound of Lemma \ref{lemmaJ} for $J$, we get 
\begin{align*}
\int_M \varphi_t |du|^2 \di\nu_g &\geq e^{-4\mbox{\tiny{k}}_t(M,g)}\int_M\varphi |du_t|^2\di\nu_g +\ e^{-12\mbox{\tiny{k}}_t(M,g)}  \frac{2}{n} t\int_M \varphi(\Delta_g u_t)^2\di\nu_g\\
&\geq e^{-12\mbox{\tiny{k}}_t(M,g)}\left( \int_M\varphi |du_t|^2\di\nu_g +\frac{2}{n} t\int_M \varphi(\Delta_g u_t)^2\di\nu_g\right)\cdot\end{align*}
We also have $\mbox{{k}}_s(M,g)\le \mbox{{k}}_t(M,g)<1/8$ for any $s\in (0,t]$. Hence if $s\in (0,t]$  then 
$$\int_M \varphi_s |du|^2 \di\nu_g \geq e^{-12\mbox{\tiny{k}}_t(M,g)}\left( \int_M\varphi |du_s|^2\di\nu_g +\frac{2}{n} s\int_M \varphi(\Delta_g u_s)^2\di\nu_g\right)\cdot$$
Apply this with $t=s$ and $u =P_{t-s}v= v_{t-s}$ to get
\begin{equation}\label{eq:th1.1fin}
\int_M \varphi_s |dv_{t-s}|^2\di\nu_g \geq e^{-12\mbox{\tiny{k}}_t(M,g)}\left(\int_M\varphi |dv_t|^2\di\nu_g + \frac{2}{n} s \int_M\varphi(\Delta_g v_t)^2\di\nu_g \right).
\end{equation}
 Observe that the left-hand side of the previous inequality can be rewritten as the following derivative with respect to $s$:
$$\int_M \varphi_s |dv_{t-s}|^2\di\nu_g = \frac{\di}{\di s}\int_M \varphi_s \frac{v_{t-s}^2}{2}\di\nu_g.$$
Taking this into account while integrating \eqref{eq:th1.1fin} between $0$ and $t$ yields \eqref{id3}. 
\end{proof}

\subsection{Convergence of the Energy}

Let us prove now that the Cheeger energy of a Kato limit space $(X,\dist,\mu,o,\cE)$ coincides with the limit Dirichlet energy $\cE$. 

\begin{theorem}
\label{thm:int_ChE}
Let $(X,\dist, \mu, \cE, o)$ be a Kato limit space. Then $\cE=\Ch_\dist$.
\end{theorem}

\begin{rem} Theorem \ref{thm:int_ChE} implies that for Kato limit spaces $(X,\dist,\mu,o)$, the pmGH convergence of the approximating sequence of manifolds implies the Mosco convergence of the associated energies. As a consequence, if $X$ is compact then we have convergence of the spectrum of the rescaled Laplacians of the approximating manifolds to the spectrum of the Laplacian associated with the Cheeger energy. This generalizes results of J.~Cheeger and T.~Colding \cite[Section 7]{CheegerColdingIII}, where a uniform lower bound on the Ricci curvature is assumed, and of K.~Fukaya \cite{Fukaya}, under a uniform bound on the curvature.
\end{rem}

\begin{proof}[Proof of Theorem \ref{thm:int_ChE}] Let $\{(M_\alpha,g_\alpha,o_\alpha)\}_\alpha$ be a sequence of pointed Riemannian manifolds satisfying a uniform Kato bound and such that $\{(M_\alpha,\dist_{g_\alpha},\mu_\alpha,o_\alpha,\cE_\alpha)\}_{\alpha \in A}$ converges in the Mosco-Gromov-Hausdorff sense to $(X,\dist, \mu, \cE, o)$. Let $T>0$ and $f:(0,T] \to [0,+\infty)$ be the non-decreasing function in Definition \ref{def:bounds}.  We know from Proposition \ref{converge3} that $\dist=\dist_\cE.$ Moreover, by Remark \ref{rem:KatoPI}, we get that $(X,\dist_\cE,\mu,\cE)$ is a PI$(R)$ Dirichlet space.  Therefore, thanks to Proposition \ref{KSZ}, we are left with showing that for any $u\in \cD(\cE)$, any non-negative $\varphi\in \cD(\cE)\cap \cC_c(X)$  and any $t\in [0,T],$
\begin{equation}\label{lip}\int_X \varphi\di \Gamma (P_tu)\le e^{4f(t)}\int_X P_t\varphi\di \Gamma (u).\end{equation}

Let $u, \varphi$ be as above. Set $L:=\|\varphi\|_{L^\infty}$ and let $R>0$ be such that $\supp \varphi\subset B_R(o)$.
Let $\{u_\alpha\}_\alpha, \{\varphi_\alpha\}$ be two sequences, where $u_\alpha \in \cD(\cE_\alpha)$ and $\varphi_\alpha\in \cC_c(X_\alpha)\cap \cD(\cE_\alpha)$ for any $\alpha$, such that
\begin{itemize}
\item $u_\alpha\stackrel{\mE}{\longrightarrow}u$,
\item the sequence $\{\varphi_\alpha\}_\alpha$ converges uniformly to $\varphi$,
\item  $0\le \varphi_\alpha\le L$ and $\supp \varphi_\alpha\subset B_{R+1}(o_\alpha)$ for any $\alpha$.
\end{itemize}
The Mosco convergence $\cE_\alpha\to \cE$ guarantees the existence of $\{u_\alpha\}$ while Proposition \ref{prop:approx} ensures the existence of $\{\varphi_\alpha\}$. Let $(P^\alpha_t)_{t\ge 0}$ (resp.~$(P_t)_{t\ge 0}$) be the heat semi-group of the Dirichlet space $(M_\alpha,\dist_\alpha,\mu_\alpha,\cE_\alpha)$ (resp.~$(X,\dist,\mu,\cE)$) for any $\alpha$.  Fix $t\in (0,T]$ and $s \in [0,t]$. Set 
\[
a(s):=A_t(u,\varphi)(s) \qquad \text{and} \qquad a_\alpha(s) :=A^{\alpha}_t(u_\alpha,\varphi_\alpha)(s) 
\]
for any $\alpha$. We claim that
\begin{equation}\label{claim}
\lim_\alpha a_\alpha(s)= a(s).
\end{equation}
Indeed, by (3) in Proposition \ref{prop:equivmosco}; the sequence $\{P^\alpha_{t-s}u_\alpha\}_\alpha$ converges strongly in $L^2$ to $P_{t-s}u$.  Moreover,  let us prove that $\{P^\alpha_{s}\varphi_\alpha\}_\alpha$ converges uniformly on compact sets to $P_{s}\varphi$.  Since $P_s\varphi$ is continuous, this follows from showing that for any $x \in X$ and any given sequence $\{x_\alpha\}$ such that $M_\alpha \ni x_\alpha \to x \in X$, 
\begin{equation}\label{lim}
\lim_\alpha P^\alpha_{s}\varphi_\alpha(x_\alpha)=P_s \varphi (x).
\end{equation}
Let $H_\alpha$ (resp.~$H$) be the heat kernel of $(M_\alpha,\dist_\alpha,\mu_\alpha,\cE_\alpha)$ (resp.~$(X,\dist,\mu,\cE)$) for any $\alpha$.  We have
\[
P_s^{\alpha} \varphi_\alpha (x_\alpha) = \int_{M_\alpha} H_\alpha(s,x_\alpha,y) \varphi_\alpha(y) \di \mu_\alpha(y)
\]
for any $\alpha$, and similarly
\[
P_s \varphi (x) = \int_X H(s,x,y) \varphi(y) \di \mu(y).
\]
The heat kernel estimate \eqref{eq:HKbound} ensures that we can apply  Proposition \ref{prop:convdom} to the continuous functions $\{H_\alpha(s,x_\alpha,\cdot) \varphi_\alpha(\cdot)\}_\alpha,H(s,x,y) \varphi(y)$: this directly establishes \eqref{lim}. We are then in a position to apply Proposition \ref{prop:cvproduit} to make the integrals
\[
a_\alpha(s)=\int_X (P^\alpha_{t-s}u_\alpha)^2 P_s^{\alpha} \varphi_\alpha
\]
converge to $a(s)$, as claimed in \eqref{claim}.

From \eqref{def_DerA}, we have $$a'_\alpha(s)= \int_{M_\alpha}  P^\alpha_{s}\varphi_\alpha \left|dP^\alpha_{t-s}u_\alpha\right|^2\di\mu_\alpha.$$
Assume from now on that $s\in (0,t/2)$. According to Corollary \ref{cor:inForCheeger}, we know that
$$a'_\alpha(s)\le e^{4f(t)}a'_\alpha(t-s).$$
Integrating this inequality between $0$ and $s$ and dividing by $s$ yields
$$\frac{a_\alpha(s)-a_\alpha(0)}{s}\le e^{4f(t)}\frac{a_\alpha(t)-a_\alpha(t-s)}{s}.$$
Letting $\alpha\to \infty$, then $s\to 0$, and using expression \eqref{def_DerA} for $a'(s)=A'_t(u,\varphi)(s)$ imply inequality \eqref{lip}.
\end{proof}

\begin{rems}\label{rem:plusgen}

Similarly to what we pointed out in Remark \ref{rem:plusgen1}, it is clear that Theorem \ref{thm:int_ChE} also holds when we assume the following more general assumption: there exists a non-decreasing function $f: (0,T] \to \R_+$ such that $\lim_0 f=0$ and
$$\limsup_{\alpha \to \infty} \Kato_t(M_\alpha, g_\alpha) \leq f(t) \ \mbox{ for all } t \in (0,T].$$
In particular, when $\lim_{\alpha} \Kato_T(M_\alpha, g_\alpha)=0$ we have both $\dist =\dist_{\cE}$ and $\cE=\Ch$. 

According to Remark \ref{rem:tgConeKato}, tangent cones (and rescalings $(Z,\dist_Z,\mu_Z,z,\cE_Z)$ centered at convergent points) of a Kato limit are obtained as limits of manifolds such that for all $t>0$, $\Kato_t(M_\alpha, g_\alpha) \to 0$ as $\alpha \to \infty$. Thus Theorem \ref{thm:int_ChE} then applies to such spaces.

\end{rems}

\subsection{The $\RCD$ condition for a certain class of Kato limit spaces}
 
In the next key result we prove that any Kato limit space obtained from manifolds $\{(M_\alpha,g_\alpha)\}$ with $\mbox{\emph{k}}_T(M_\alpha^n,g_\alpha) \to 0$ for some $T>0$ satisfies the Riemannian Curvature Dimension condition $\RCD(0,n)$. 

\begin{theorem}
Let $(X,\dist,\mu,o,\cE)$ be a Dirichlet space obtained as pointed Mosco-Gromov-Hausdorff limit of spaces $\{(M_\alpha, \dist_{g_\alpha},\mu_{\alpha},o_\alpha,\cE_\alpha)\}$, where $\{(M_\alpha^n,g_\alpha,o_\alpha)\}$ is a sequence of closed pointed Riemannian manifolds such that 
\begin{equation}
\label{eq:HypGen}
\tag{H}
\mbox{\emph{k}}_T(M_\alpha^n,g_\alpha) \to 0
\end{equation}
for some $T>0$.  Then $(X, \dist, \mu)$ is an $\RCD(0,n)$ space.
\end{theorem}

\begin{proof}

Our goal is to establish the Bakry-Ledoux $\mathrm{BL}(0,n)$ estimate: for $t \in (0,T)$ and $u\in \mathcal{D}(\mathcal{E})$ and $\varphi \in \cD(\cE)\cap \cC^0_c(X)$ with $\varphi \ge 0$,  we aim at showing
\begin{equation}
\label{intBL}
\frac 12 \int_X \varphi (P_t u^2 - (P_tu)^2 )\di\mu \geq \int_X \varphi \left( t \frac{\di \Gamma \left( P_tu\right)}{\di\mu} +\frac{t^2}{n}(\Delta P_t u)^2 \right) \di\mu,
\end{equation}
starting from inequality \eqref{id4} in Theorem \ref{thm:Aineq}.  As in the proof of  Theorem \ref{thm:int_ChE},  we let $\{u_\alpha\}_\alpha, \{\varphi_\alpha\}$ be two sequences where $u_\alpha \in \cD(\cE_\alpha)$, $\varphi_\alpha\in \cC_c(X_\alpha)\cap \cD(\cE_\alpha)$ for any $\alpha$ such that
\begin{itemize}
\item $u_\alpha\stackrel{\mE}{\longrightarrow}u$,
\item the sequence $\{\varphi_\alpha\}_\alpha$ converges uniformly to $\varphi$,
\item  $0\le \varphi_\alpha\le L$ and $\supp \varphi_\alpha\subset B_{R+1}(o_\alpha)$ for any $\alpha$, for some $L,R>0$.
\end{itemize}
Take $t \in (0,T]$ and $0\le s\le t$.  Like in the proof of the previous Theorem \ref{thm:int_ChE},  we set $a_\alpha(s):=A^{\alpha}_t(u_\alpha,\varphi_\alpha)(s)$ for any $\alpha$ and $a(s):=A_t(u,\varphi)(s)$. We know from there that
$$\lim_\alpha a_\alpha(s)=a(s).$$
For any $\alpha$,  we let $(P_t^\alpha)_{t\ge 0}$ be the heat semi-group associated with $\cE_\alpha$ and we set
$$c_\alpha(s)=\int_{M_\alpha}  P^\alpha_s\varphi_\alpha \left(\Delta_{\alpha} P^\alpha_{t-s}u_\alpha\right)^2 \di\mu_\alpha \quad \text{and}\quad
 c(s)=\int_{X}  P_s\varphi \left(\Delta
 P_{t-s}u\right)^2 \di\mu.$$
Thanks to (5) in Proposition \ref{prop:equivmosco}, we know that $\Delta_\alpha P_t^\alpha$ strongly converges to $\Delta P_t$ in the sense of bounded operators. 
In particular for any $s\in [0,t)$, we have the strong convergence $$\Delta_\alpha P_{t-s}^\alpha u_\alpha\stackrel{L^2}{
\longrightarrow}\Delta P_{t-s} u ,$$ hence the same argument to get the convergence $a_\alpha(s)\to a(s)$ gives
 $$\lim_\alpha c_\alpha(s)=c(s).$$
Now observe that dividing  inequality \eqref{id4} by $\nu_{g_\alpha}(B_{\sqrt{T}}(o_\alpha)$  implies, for any $0 < s < t < T$, 
 \begin{equation*}
 a_\alpha(0)-a_\alpha(s)\geq  e^{-12\mbox{\tiny{k}}_t(M_\alpha,g_\alpha)}\left( (t-s) \frac{\di a_\alpha}{\di s}(s) + \frac{2(t-s)^2}{n}  c_\alpha(s)\right).
\end{equation*}
Consider a non-negative test function $\xi \in C^{\infty}_0((0,t))$. When multiplying the previous inequality by $\frac{\xi}{t-s}$ and integrating over $[0,t]$ we obtain: 
\begin{equation*}\begin{split}
\int_0^t\frac{\xi(s)}{(t-s)}\di s\, a_\alpha(0)-\int_0^t\frac{\xi(s)}{(t-s)}a_\alpha(s)\di s&  \geq  e^{-12\mbox{\tiny{k}}_t(M_\alpha,g_\alpha)}\left( \int_0^t\xi(s) \frac{da_\alpha}{ds}(s) \di s\right.\\
&\left.+\int_0^t\xi(s) \frac{2(t-s)}{n}  c_\alpha(s)\di s \right).\end{split}
 \end{equation*}
Integrating by parts, we get 

\begin{equation*}\begin{split}
\int_0^t\frac{\xi(s)}{(t-s)}\di s\, a_\alpha(0)-\int_0^t\frac{\xi(s)}{(t-s)}a_\alpha(s)\di s&  \geq  e^{-12\mbox{\tiny{k}}_t(M_\alpha,g_\alpha)}\left( -\int_0^t\xi'(s) a_\alpha(s) \di s\right.\\
&\left.+\int_0^t\xi(s) \frac{2(t-s)}{n}  c_\alpha(s)\di s \right).\end{split}
 \end{equation*}

We remark that for any $s\in [0,t]\colon$
$$a_\alpha(s)\le L \|u_\alpha\|^2_{L^2}\text{ and } c_\alpha(s)\le  \frac{L}{(t-s)^2}\|u_\alpha\|^2_{L^2}.$$

Recall that $\xi$ has compact support in the open interval $(0,t)$ hence we can use the dominate convergence theorem to get that 
\begin{equation*}
\int_0^t\frac{\xi(s)}{(t-s)}ds\, a(0)-\int_0^t\frac{\xi(s)}{(t-s)}a(s)ds  \geq  -\int_0^t\xi'(s) a(s) ds
\int_0^t\xi(s) \frac{2(t-s)}{n}  c(s)ds.
 \end{equation*}
The functions $a,a'$ and $c$ are continuous on $[0,t)$. The continuity of $a$ and $c$ follows from \cite[Lemma 2.1]{AGS15}. The continuity of $a'$ follows similarly from the same lemma. Indeed, we  assumed that $\varphi$ is continuous with compact support, hence the function $s\in [0,t] \mapsto P_s\varphi\in \cC_0^0(X)$ is continuous.  Moreover, $s\in [0,t)\mapsto P_{t-s}u\in \cD(\cE)$ is then $s\in [0,t) \mapsto \di\Gamma\left(P_{t-s}u\right)\in \Rad(X)=\left(\cC_0^0(X)\right)'$  is also continuous, therefore $a'$ is continuous.

As a consequence, the function $a$ is $C^1$ on $[0,t)$, thus by integrating by parts we obtain
 \begin{equation*}
\int_0^t\frac{\xi(s)}{(t-s)}ds\, a(0)-\int_0^t\frac{\xi(s)}{(t-s)}a(s)ds  \geq  \int_0^t\xi(s) a'(s) ds
\int_0^t\xi(s) \frac{2(t-s)}{n}  c(s)ds.
 \end{equation*}
Moreover, since $a$ and $c$ are continuous on $[0,t)$, for any $s\in [0,t)$ we get
$$  a(0)-a(s)\geq   (t-s) \frac{\di a}{\di s}(s) + \frac{2(t-s)^2}{n}  c(s).$$
that is to say for any $s \in [0,t)$,  we obtain
\begin{equation*}\begin{split}
 \frac{1}{2}\left(\int_{X} \varphi P_t(u^2) \di\mu -\int_{X} P_s\varphi\, \left(P_{t-s}u\right)^2\di\mu\right)  \geq& (t-s) \int_{X} P_s\varphi\,\di \Gamma\left(P_{t-s}u\right)\\
+& \frac{(t-s)^2}{n} \int_{X} P_s\varphi (\Delta P_{t-s}u)^2 \di\mu. \end{split}
\end{equation*}
Inequality \eqref{intBL} coincides with this latter inequality when taking $s=0$.
\end{proof}

Thanks to Remarks \ref{rem:plusgen1} and \ref{rem:plusgen}, the two previous theorems allow us to get the following result for tangent cones of Kato limit spaces.

\begin{cor}
\label{cor:TgConeRCD}
Let $(X,\dist,\mu,o,\cE)$ be a (resp.~non-collapsed strong) Kato limit space. Then for any $x \in X$, any measured tangent cone $(X_x,\dist_x,\mu_x,x)$  is a (resp.  weakly non-collapsed) $\RCD(0,n)$ metric measure space. 
\end{cor}

The fact that tangent cones of non-collapsed strong Kato limits are weakly non-collapsed follows from the local Ahlfors regularity given in Remark \ref{rem-AR-ncStrongKatolim}.

\begin{rem}
\label{rem-RCDmoreGeneral}
The previous Corollary applies in particular to iterated tangent cones and to convergent rescalings (not necessarily centered at a same point) of a Kato limit.  Indeed, let $(X,\dist, \mu,o,\cE)$ be a Kato limit and assume that for some sequence $\{\eps_\alpha\} \subset (0,+\infty)$ such that $\eps_\alpha \downarrow 0$ and some points $\{x_\alpha\} \subset X$,  the sequence of pointed rescalings $\{(X, \eps_\alpha^{-1} \dist, \eps_\alpha^{-n}\mu, x_\alpha)\}$ converges to some pointed metric measure space $(Z, \dist_Z, \mu_Z,z)$. Thanks to Remark \ref{rem-TgConesSKL}, we know that $(Z, \dist_Z, \mu_Z)$ has the same properties as a tangent cone. As a consequence, $(Z,\dist_Z,\mu_Z)$ is a weakly non-collapsed $\RCD(0,n)$ space. 
\end{rem}

\section{Tangent cones are metric cones}

In this section, we prove that the tangent cones of a non-collapsed strong Kato limit space are all metric cones. To this aim, we introduce two crucial quantities. Let $(X,\dist,\mu,\cE)$ be a metric Dirichlet space admitting a heat kernel $H$. We recall that $n$ is a given positive integer. For any $x \in X$ and $s,t>0$, we define the \textbf{$\bf \Theta$-volume} by
$$\Theta_x(s):=(4\pi s)^{-\frac n2}\int_X e^{-\frac{\dist^2(x,y)}{4s}}\di \mu(y)$$
and similarly the \textbf{$\bf \theta$-volume} by
$$\theta_x(s,t):=(4\pi s)^{-\frac n2}\int_X e^{-\frac{U(t,x,y)}{4s}}d \mu(y)$$
where $U$ is defined by \begin{equation}\label{eq:HU}H(t,x,y)=(4\pi t)^{-\frac n2}e^{-\frac{U(t,x,y)}{4t}}
\end{equation}
for any $y \in X$. 
\begin{rem}
\label{rem-theta-vol} 
We point out the following properties of $\theta_x$ and $\Theta_x$.
\begin{itemize}
\item[(i)] $\theta_x(t,t)=\int_XH(t,x,y)\di\mu(y)$ is identically $1$ if the Dirichlet space is stochastically complete, for instance for $\mathrm{PI}$ Dirichlet space.
\item[(ii)]  The  Chapman-Kolmogorov property yields
\begin{equation}\label{rem:Htheta}
\theta_x\left(\frac{t}{4},\frac{t}{2}\right)=(4\pi t)^{\frac n2}H(t,x,x).
\end{equation}
\item[(iii)] We always have the lower bound: \begin{equation}\label{eq:minTheta}\Theta_x(s)\geq (4\pi s)^{-\frac n2}e^{-\frac{r^2}{4s}}\mu(B_r(x)),\end{equation} hence if $\mu$ is non trivial then $\Theta_x$ is always positive (but perhaps infinite).
\end{itemize}
\end{rem}

When it is necessary, we specify the space $X$ to which the previous quantities are associated by writing $\Theta^X_x(s), \theta^X_x(s,t)$. 

We start by recalling an elementary property of the $\Theta$-volume that relates it to the density of volume at a given point. 
\subsection{On a doubling space}
\begin{lemma}
\label{lem-Thetadoubling}
Let $(X,\dist,\mu)$ be $\upkappa$-doubling at scale $R$.  \begin{enumerate}[i)]
\item The function $(s,x)\in \setR_+^*\times X \mapsto \Theta_x(s)$ is continuous.
\item  For all $x \in X$ and $s>0$ we have that $\Theta_x(s)$ is finite. More precisely if $t\in (0,R^2]$, there is a constant $A$ depending only $\upkappa$ such that for any $x\in X$ and $s\in (0,R^2]$:
\begin{equation}\label{eq:majTheta}\Theta_x(s)\le A\, \frac{\meas (B_{\sqrt{s}}(x))}{s^{\frac n 2}}.\end{equation} 
\item If there exists $\vartheta \in (0, +\infty)$ such that
$$\lim_{r \rightarrow 0^+}\frac{\meas(B_r(x))}{\omega_n r^n}=\vartheta,$$
then $\displaystyle \lim_{s\rightarrow 0^+}\Theta_x(s)=\vartheta$.
\end{enumerate}
\end{lemma}

\begin{proof}

For any $x \in X$, by Cavalieri's formula we can write
$$\Theta_x(s)=\int_0^{+\infty}e^{-\frac{r^2}{4s}}\frac{r}{2s}\mu(B_r(x)) \frac{\di r }{(4\pi s)^{\frac n2}}$$
A simple computation using Proposition \ref{prop:doubling}-iii) ensures that this integral converges,  hence $\Theta_x(s)$ is well defined for any $s>0$. \par Moreover, Lebesgue's dominated convergence theorem implies that $\Theta_x(s)$ depends continuously on $(s,x)\in \setR_+^*\times X $. 
In addition, when $0<s\le R^2$, the estimate (\ref{eq:majTheta}) follows from
\begin{align*}
\Theta_x(s)&\le \frac{\meas (B_{\sqrt{s}}(x))}{(4\pi s)^{\frac n2}}+\int_{X\setminus B_{\sqrt{s}}(x)} \frac{e^{-\frac{\dist^2(x,y)}{4s}}}{(4\pi s)^{\frac n2}}\di \mu(y)\\
&\le \frac{\meas (B_{\sqrt{s}}(x))}{(4\pi s)^{\frac n2}}+\int_{\sqrt{s}}^{+\infty} e^{-\frac{\rho^2}{4s}}\frac{\rho}{2s}\mu(B_\rho(x)) \frac{\di\rho }{(4\pi s)^{\frac n2}}\\
&\le \frac{\meas (B_{\sqrt{s}}(x))}{(4\pi s)^{\frac n2}}\left(1+\int_{\sqrt{s}}^{+\infty} e^{-\frac{\rho^2}{4s}+\lambda\frac{\rho}{\sqrt{s}} }\frac{\rho}{2s} \di\rho \right) \\
&\le  \frac{\meas (B_{\sqrt{s}}(x))}{(4\pi s)^{\frac n2}}\left(1+\int_{1}^{+\infty} e^{-\frac{\rho^2}{4}+\lambda\rho }\,\frac{\rho}{2} \di\rho \right).
\end{align*}

\noindent Now assume that 
$$\lim_{r \rightarrow 0^+}\frac{\mu(B_r(x))}{\omega_n r^n}=\vartheta.$$
We have similarly 
$$\Theta_x(s)=\int_{B_{R}(x)} \frac{e^{-\frac{\dist^2(x,y)}{4s}}}{(4\pi s)^{\frac n2}}\di \mu(y)+\int_{M\setminus B_{R}(x)} \frac{e^{-\frac{\dist^2(x,y)}{4s}}}{(4\pi s)^{\frac n2}}\di \mu(y).$$
The same estimate yields
\begin{align*}
\int_{M\setminus B_{R}(x)} \frac{e^{-\frac{\dist^2(x,y)}{4s}}}{(4\pi s)^{\frac n2}}\di \mu(y) & \le  \frac{\meas (B_{\sqrt{R}}(x))}{(4\pi s)^{\frac n2}}\int_{R}^\infty  e^{-\frac{\rho^2}{4s}+\lambda\frac{\rho}{R} }\frac{\rho}{2s} \di\rho\\
& = \frac{\meas (B_{\sqrt{R}}(x))}{(4\pi s)^{\frac n2}}\int_{\frac{R}{\sqrt{s} }}^\infty  e^{-\frac{\rho^2}{4}+\lambda\frac{\rho \sqrt{s}}{R} }\frac{\rho}{2} \di\rho.
\end{align*}
We clearly have
$$\lim_{s \rightarrow 0^+}\int_{M\setminus B_{R}(x)} \frac{e^{-\frac{\dist^2(x,y)}{4s}}}{(4\pi s)^{\frac n2}}\di \mu(y)=0.$$
Moreover we have 
\begin{equation*}
\int_{B_{R}(x)} \frac{e^{-\frac{\dist^2(x,y)}{4s}}}{(4\pi s)^{\frac n2}}\di \mu(y)=
 \int_0^{\frac{R}{\sqrt{s} }}
\frac{e^{-\frac{\rho^2}{4}}\omega_n \rho^{n+1}}{2 (4\pi )^{\frac n2}} \frac{\mu(B_{\rho\sqrt{s}}(x))}{\omega_n (\rho \sqrt{s})^n}\di \rho. 
\end{equation*}
Note that there is some constant $C$ (depending on $x$) such that for $r\le R$:
$$\mu(B_r(x))\le C r^n.$$ Hence the dominated convergence theorem and the fact that$$ \int_0^{+\infty} 
\frac{e^{-\frac{\rho^2}{4}}\omega_n \rho^{n+1}}{2 (4\pi )^{\frac n2}}d\rho=\Theta^{\R^n}_{0}(1)=1$$ imply that 
$$\lim_{s \rightarrow 0^+}  \int_{B_{R}(x)} \frac{e^{-\frac{\dist^2(x,y)}{4s}}}{(4\pi s)^{\frac n2}}\di \mu(y))=\vartheta.$$
\end{proof}
We also have the following assertion about the comparison between the measure of balls and the $\Theta$-volume:
\begin{lemma}\label{lem:simple}
Let $(X,\dist,\mu)$ be a metric measure space, and $n>0$. 
\begin{enumerate}[i)]\item let  $x \in X$ and $c>0$, then
 $\Theta_x(s) =c$ for all $s>0$ if and only if $\meas(B_r(x)) = c \omega_n r^n$ for all $r>0$.
 \item If $(X,\dist,\mu)$ is $\upkappa$-doubling at scale $R$, then for any $s\in (0,R^2]$:
 \begin{equation}\label{eq:AhlforsTheta}
a\frac{\meas (B_{\sqrt{s}}(x))}{(\sqrt{s})^n}\le  \Theta_x(s)\le A\frac{\meas (B_{\sqrt{s}}(x))}{(\sqrt{s})^n}.
 \end{equation} where the constant $a>0$ depend only on 
 $n$ and the constant $A$ depends only on $\upkappa$.

\item Assume that for for some  $R>0$ and $x\in X$:$$ \forall s\in (0,R^2]\colon c\le \Theta_x(s)\le C$$ then for any $r\in (0,R]$:
 $$vr^n\le \meas(B_r(x)) \le V r^n.$$
where the constant $v,V$ depends only on $n,c,C$.
 \end{enumerate}
\end{lemma}
\begin{proof}\par
\noindent \textit{The first equivalence}  follows from Cavalieri's principle and some properties of the Laplace transform, see e.g.~the proof of \cite[Lem.~3.2]{CT19}.\\
\noindent\textit{The  second assertion}  is a consequence of the two inequalities (\ref{eq:minTheta}) and (\ref{eq:majTheta}).\\
\noindent\textit{For the third assertion :} The upper bound is a consequence of (\ref{eq:minTheta}). The upper bound follows from an argument used in the proof of \cite[Lem.~5.1]{CT19}. 
The estimate (\ref{eq:minTheta}) also implies that for any $\rho\ge \sqrt{t}$:
$$\mu( B_{\rho}(x))\le (4\pi s)^{\frac n 2} e^{\frac{\rho^2}{4s}}\,  C.$$
Hence for any $\kappa\in (0,1)$,
\begin{align*}
\Theta_x(\kappa s)&\le \frac{\meas (B_{\sqrt{s}}(x))}{(4\pi \kappa s)^{\frac n2}}+\int_{X\setminus B_{\sqrt{s}}(x)} \frac{e^{-\frac{\dist^2(x,y)}{4\kappa s}}}{(4\pi \kappa s)^{\frac n2}}\di \mu(y)\\
&\le \frac{\meas (B_{\sqrt{s}}(x))}{(4\pi \kappa s)^{\frac n2}}+\int_{\sqrt{s}}^{+\infty} e^{-\frac{\rho^2}{4\kappa s}}\frac{\rho}{2\kappa s}\mu(B_\rho(x)) \frac{\di\rho }{(4\pi \kappa s)^{\frac n2}}\\
&\le \frac{\meas (B_{\sqrt{s}}(x))}{(4\pi \kappa s)^{\frac n2}}+\frac{C}{\kappa^{\frac n2}} \int_{\sqrt{s}}^{+\infty} e^{-\left(\frac1\kappa -1\right)\frac{\rho^2}{4s} }\frac{\rho}{2\kappa s} \di\rho  \\
&\le \frac{\meas (B_{\sqrt{s}}(x))}{(4\pi \kappa s)^{\frac n2}}-\frac{C}{\kappa^{\frac n2}} \left[\frac{ e^{-\left(\frac1\kappa -1\right)\frac{\rho^2}{4s}}} {1-\kappa} \right]_{\sqrt{s}}^{+\infty} \\
&\le  \frac{\meas (B_{\sqrt{s}}(x))}{(4\pi  \kappa s)^{\frac n2}}+\frac{Ce^{-\left(\frac1\kappa -1\right)\frac{1}{4} }}{(1-\kappa) \kappa^{\frac n2}} .\end{align*}
Choosing $\kappa$ small enough so that $$\frac{e^{-\left(\frac1\kappa -1\right)\frac{1}{4} }}{(1-\kappa) \kappa^{\frac n2}} \le \frac{1}{2}\frac{c}{C}$$ yields that 
$$\frac{c}{2} (4\pi  \kappa s)^{\frac n2}\le \meas (B_{\sqrt{s}}(x)).$$
\end{proof}
The $\Theta$-volume  is continuous with respect to pointed measured Gromov-Hausdorff convergence: 
\begin{prop}
Let $\{(X_\alpha,\dist_\alpha,\meas_\alpha,o_\alpha)\}_\alpha,(X,\dist,\meas,o)$ be proper geodesic pointed metric measure spaces  $\upkappa$-doubling at scale $R$ such that  $(X_\alpha,\dist_\alpha,\meas_\alpha,o_\alpha) \to (X,\dist,\meas,o)$  in the pmGH sense.
 Let $x_\alpha\in X_\alpha\to x\in X$. Then for any $s>0$:
 $$\lim_{\alpha} \Theta^{X_\alpha}_{x_\alpha}(s)=\Theta^{X}_{x}(s).$$
 \end{prop}
 \begin{proof} This is a direct consequence of the Proposition \ref{prop:convdom}.  \end{proof}

\subsection{On a Dirichlet space}
The following gives a relationship between the $\Theta$- and $\theta$-volume on a $\mathrm{PI}$ Dirichlet space. 
\begin{prop}\label{prop:theta(s,t)} If $(X,\dist_\cE,\mu,\cE)$ is a  $\mathrm{PI}(R)$-Dirichlet space then when the $\Theta$-volume is defined with the intrinsic distance then for any $s>0$:
$$\Theta_x(s)=\lim_{t\to 0+} \theta_x(s,t).$$
\end{prop}

\begin{proof} According to the Varadhan's formula \eqref{eq:varadhan}, we get that for any $x,y\in X$
$$\dist_\cE^2(x,y)=\lim_{t\to 0+} U(t,x,y).$$ The Fatou's lemma implies that 
\begin{equation}\label{eq:liminfx}
\Theta_x(s)\le \liminf_{t\to 0} \theta_x(s,t).
\end{equation}

To get the limsup inequality,  we will use the heat kernel upper bound  \eqref{est:Gaussian14} (see \cite[Theorem 5.2]{GriRev}):
$$H(t,x,y)\le \frac{C}{\mu (B_R(x))} \frac{R^\nu}{t^{\frac \nu 2}} \left(1+\frac{\dist_\cE^2(x,y)}{t}\right)^{ \nu +1} \ e^{-\frac{\dist_\cE^2(x,y)}{4t}}$$
for any $x,y \in X $ and $t \in (0,R^2)$.

Since there exists $C>0$ such that $\xi \mapsto \xi^{\nu+1}e^{-\xi} \le C$ for any $\xi>0$, then choosing $\xi=\eps(1 + \dist_\cE^2(x,y)/t)$ where $\eps \in (0,1)$ leads to
$$H(t,x,y)\le \frac{C}{\mu (B_R(x))} \frac{R^\nu}{t^{\frac \nu 2}}  \eps^{- \nu -1} \ e^{\eps+(\eps-1)\frac{\dist_\cE^2(x,y)}{4t}},$$ which implies that there is a constant $C$ not depending  on $t\in (0,R^2]$ and $y\in X$ such that 
\begin{align}\label{eq:esti}
-\frac{U(t,x,y)}{4} \le t (C+\log(1/t))- (1-\eps)\frac{\dist_\cE^2(x,y)}{4}
\end{align}
and then, for any $s>0$,
\begin{equation}\label{eq:limsup}
\theta_x(s,t) \le \left(\frac{e^C}{t}\right)^{\frac{t}{s}} \Theta_x(s/(1-\eps)).
\end{equation}
So that for any $\eps\in (0,1):$
$$\limsup_{t \to 0} \theta_x(s,t) \le \Theta_x\left(\frac{s}{1-\eps}\right).$$
Then the assertion follows from the continuity of $\Theta$ with respect to $s$.
\end{proof}

\begin{rem}\label{eq:bound}
It is worth pointing out that if   $(X,\dist_\cE,\mu,\cE)$ is a PI$_{\upkappa, \upgamma}(R)$-Dirichlet space
then there is a constant $\upalpha$ such that for any $t\in (0,R^2)$ and any $x,y\in X$, we get the heat kernel bound
$$\frac{t^{\frac n2}}{\mu(B_{\sqrt{t}}(x))}\, 
\frac{1}{\upalpha t^{\frac n2}}e^{-\upalpha\frac{\dist^2(x,y)}{ t}}\le H(t,x,y)\le \frac{t^{\frac n2}}{\mu(B_{\sqrt{t}}(x))} \frac{\upalpha}{t^{\frac n2}}e^{-\frac{\dist^2(x,y)}{\upalpha t}}.$$

This easily implies that 
$$\left(\frac{(4\pi)^{\frac n2}}{\upalpha}\right)^{\frac t s}\Theta_x\left(\frac{s}{4\upalpha}\right)\le\left(\frac{t^{\frac n2}}{\mu(B_{\sqrt{t}}(x))}\right)^{-\frac ts}\theta_x(s,t)\le \left((4\pi)^{\frac n2}\upalpha \right)^{\frac t s}\Theta_x\left(\frac\upalpha 4s\right).$$
In particular, using Lemma \ref{lem:simple},  there is a positive constant $\upeta>0$ depending only $\upkappa, \upgamma,n$  such that for any $x\in X$ and $t,s>0$ with $t\le R^2$ and $s\le \frac{16}{\upalpha^2}R^2$ then 

\begin{equation}\label{eq:compathetavolu}
\upeta\, \frac{\mu(B_{\sqrt{s}}(x))}{s^{\frac n2}} \left(\frac{\mu(B_{\sqrt{t}}(x))}{t^{\frac n2}}\right)^{-\frac ts}\le \theta_x(s,  t)\le \upeta^{-1} \frac{\mu(B_{\sqrt{s}}(x))}{s^{\frac n2}}\left(\frac{\mu(B_{\sqrt{t}}(x))}{t^{\frac n2}}\right)^{-\frac ts} .\end{equation}
\end{rem}

The heat kernel upper bound implies similarly the continuity of the $\theta$-volume under the pointed Mosco Gromov-Hausdorff convergence of PI($R$)-Dirichlet space. 

\begin{prop}\label{prop:thetacont}Let $\{(X_\alpha,\dist_\alpha,\meas_\alpha,o_\alpha,\cE_\alpha)\}_\alpha, (X,\dist,\meas,o,\cE)$ be pointed PI$_{\upkappa, \upgamma}(R)$-Dirichlet spaces such that $(X_\alpha,\dist_\alpha,\meas_\alpha,o_\alpha,\cE_\alpha)\to (X,\dist,\meas,o,\cE)$ in the pointed Mosco-Gromov-Hausdorff sense.  Let $x \in X$ and $\{x_\alpha\}$ be such that $x_\alpha \in X_\alpha$ for any $\alpha$ and $x_\alpha \to x$. Then for any $s,t>0$:
 $$\lim_{\alpha} \theta^{X_\alpha}_{x_\alpha}(s,t)=\theta^{X}_{x}(s,t).$$
 \end{prop}

\begin{proof} Let $s,t>0$, we know that the sequence $f_\alpha(y)=H_\alpha(t,x_\alpha,y)$ converge uniformly on compact set to $f(y)=H(t,x,y)$. Hence the same is true for the integrand
$$h_\alpha(y)=(4\pi s)^{-\frac{n}{2}} \bigg( (4\pi t)^{\frac{n}{2}}H_\alpha(t,x_\alpha,y)\bigg)^{\frac{t}{s}}
=(4\pi s)^{-\frac{n}{2}}e^{-\frac{U_{X_\alpha}(t,x,y)}{4s} }$$ uniformly on compact set to $$h(y)=(4\pi s)^{-\frac{n}{2}}e^{-\frac{U_{X}(t,x,y)}{4s} }.$$
The uniform doubling and the uniform Poincaré inequality yields that there are positive constants $C,\nu$ depending only on $\upkappa$ and $\upgamma$ such that 

$$H_\alpha(t,x_\alpha,y)\le \frac{C}{\mu(B_{R^2}(x_\alpha))} \max\left\{1, \frac{R^\nu}{t^\frac{\nu}{2}} \right\} \, e^{-\frac{\dist_\alpha^2(x_\alpha,y)}{Ct}},$$ for any $y\in X$.
But $\lim_{\alpha}\mu_\alpha(B_{R^2}(x_\alpha))=\mu(B_{R^2}(x))$, hence we find a constant such that for any $\alpha$:
$$h_\alpha(y)\le C  e^{-\frac{\dist_\alpha^2(x_\alpha,y)}{Cs}}.$$ Hence the result follows also from Proposition \ref{prop:convdom}.
 \end{proof}

\subsection{A differential inequality}

We now study the properties of the $\theta$-volume on manifolds satisfying a Dynkin bound. Whenever this bound is improved to be an upper bound on the integral quantity
$$\int_0^T \frac{\sqrt{\Kato_t(M,g)}}{s}\di s \leq \Lambda,$$
for some $T, \Lambda >0$, we obtain a monotone quantity that will be crucial in the remainder of this section. 

\begin{prop}\label{prop:derivetheta}
Let $(M^n,g)$  be a closed Riemannian manifold satisfying $$\mbox{\emph{k}}_T(M^n,g)\le \frac{1}{16\, n}.$$
 for some $T>0$.  
For $\tau\le T$, set  $\Gamma_\tau(M^n,g):=e^{8\sqrt{n\Kato_\tau(M^n,g)}}-1$. Then for any $x \in M$, $t \in (0,\tau)$ and $s\le t/(2\Gamma_\tau(M^n,g))$

\begin{equation}\label{derive1}
t\frac{\partial \theta}{\partial t}+s\frac{\partial \theta}{\partial s}
+n \Gamma_\tau(M^n,g)\left(\frac{t}{s}-\frac{s}{t}\right)\theta\end{equation} has the same sign as $t-s$.

\end{prop}

\begin{rem}
Our proof shows that, when the Ricci curvature is non-negative (in which case $\Gamma_\tau(M^n,g)=0$), the map $\lambda \mapsto \theta_x(\lambda s, \lambda t)$ is monotone non-increasing for $s\geq t$ and monotone non-decreasing for $s \leq t$. 
\end{rem}

\begin{proof} Our proof will use the Li-Yau estimate (\cite[Proposition 3.3]{C16}): when $v\colon[0,\tau]\times M\rightarrow \R_+$ is a solution of the heat equation then 
\begin{equation}\label{eq:LY}
e^{-8\sqrt{n\Kato_\tau(M^n,g)}} \frac{|dv|^2}{v^2}-\frac{1}{v}\frac{\partial v}{ \partial t}\le e^{8\sqrt{n\Kato_\tau(M^n,g)} }\frac{n}{2t}.\end{equation}
For simplicity, let us write $\theta$, $\frac{\partial \theta}{\partial t}$, $\frac{\partial \theta}{\partial s}$, $U$, $H$, $\Kato$, $\Gamma$ instead of $\theta_x(s,t)$, $\frac{\partial \theta}{\partial t}(s,t)$, $\frac{\partial \theta}{\partial s}(s,t)$, $U(t,x,y)$, $H(t,x,y)$, $\Kato_\tau(M^n,g)$, $\Gamma_\tau(M^n,g)$, respectively. 

A direct computation implies
\begin{equation}\label{eq:derivet}
\frac{\partial \theta}{\partial t}=-\frac{1}{4s}\int_M \frac{\partial U}{\partial t}\,e^{-\frac{U}{4s}}\frac{\di\nu_g}{(4\pi s)^{\frac n2}}\end{equation}

\begin{equation}\label{eq:derives}
\frac{\partial \theta}{\partial s}=-\frac{n}{2s} \theta+\frac{1}{4s^2} J\end{equation}
where
$$J:=\int_M Ue^{-\frac{U}{4s}}\frac{\di\nu_g}{(4\pi s)^{\frac n2}} \, \cdot$$
From \eqref{eq:HU} and the fact that $H$ solves the heat equation, another computation gives
\begin{equation}\label{eq:heat1}
-\frac{n}{2t}- \frac{\partial }{\partial t}\left(\frac{U}{4t}\right)-\frac{1}{4t}\Delta U-\frac{1}{16 t^2} |\nabla U|^2=0
\end{equation}
\begin{equation}\label{eq:heat2}
-\frac{n}{2}-\frac 14 \frac{\partial U}{\partial t}+\frac{U}{4t}-\frac{1}{4}\Delta U-\frac{1}{16 t} |\nabla U|^2=0
\end{equation} where here and in the rest of the proof Laplacian and gradient are taken with respect to the $y$ variable.
Moreover the Li-Yau estimate provides (when $t\le \tau$):
\begin{equation}\label{ineq:LiYau1}
e^{-8\sqrt{n\Kato_t}}\frac{ |\nabla U|^2}{16t^2}+ \frac{\partial }{\partial t}\left(\frac n 2\log t+\frac{U}{4t} \right)  \le e^{8\sqrt{n\Kato_t}}\frac{n}{2t} \end{equation}
Adding (\ref{eq:heat1}) and (\ref{ineq:LiYau1}), together with the fact that $e^{-8\sqrt{n\mbox{\tiny{k}}_t}}-1=-e^{-8\sqrt{n\Kato_t}}\Gamma\ge -\Gamma$,  yields the estimate

\begin{equation}\label{ineq:Delta}
-\frac{\Delta U}{4}-\Gamma \frac{|\nabla U|^2}{16 t}\le \frac{n}{2}+\Gamma\frac{n}{2}.
\end{equation}

Combine \eqref{eq:derivet} and \eqref{eq:heat2},  to get
\begin{equation}\label{eq:prem}
\frac{\partial \theta}{\partial t}=\frac{n}{2s}\theta-\frac{1}{4ts} J+\frac{ 1}{4s}\int_M \Delta U\,e^{-\frac{U}{4s}}\frac{\di\nu_g}{(4\pi s)^{\frac n2}}+\int_M \frac{|\nabla U|^2 }{16ts}\,e^{-\frac{U}{4s}}\frac{\di\nu_g}{(4\pi s)^{\frac n2}}\end{equation}

Integration by parts implies:
$$\frac{ 1}{4s}\int_M \Delta U\,e^{-\frac{U}{4s}}\frac{\di\nu_g}{(4\pi s)^{\frac n2}}=-\int_M \frac{|\nabla U|^2 }{16s^2}\,e^{-\frac{U}{4s}}\frac{\di\nu_g}{(4\pi s)^{\frac n2}},$$
Hence from (\ref{eq:prem}) we get the two equalities:
\begin{equation}\label{eq:second}
\frac{\partial \theta}{\partial t}=\frac{n}{2s}\theta-\frac{1}{4ts} J+\left(\frac{ 1}{s}-\frac{1}{t}\right)\int_M  \frac{ \Delta U}{4}\,e^{-\frac{U}{4s}}\frac{\di\nu_g}{(4\pi s)^{\frac n2}}\, ,\end{equation}
\begin{equation}\label{eq:ter}
\frac{\partial \theta}{\partial t}=\frac{n}{2s}\theta-\frac{1}{4ts} J-\left(\frac{ 1}{s}-\frac{1}{t}\right)\int_M \frac{|\nabla U|^2 }{16s}\,e^{-\frac{U}{4s}}\frac{\di\nu_g}{(4\pi s)^{\frac n2}} \cdot\end{equation}


We use now inequality (\ref{ineq:Delta}) in (\ref{eq:second}) . Then if $t\geq s$ we obtain
$$\frac{\partial \theta}{\partial t}-\frac{n}{2s}\theta+\frac{J}{4ts}\geq \left( \frac 1s -\frac 1t\right)\left( -\frac{n}{2}(\Gamma +1)\theta-\Gamma \int_M\frac{|\nabla U|^2}{16t}\,e^{-\frac{U}{4s}}\frac{\di\nu_g}{(4\pi s)^{\frac n2}}\right),$$
and we have the opposite inequality if $t \leq s$. As a consequence, there exists a non-negative function $\uppi\colon \R_+\times [0,\tau]\times M\rightarrow\R_+$ such that:
\begin{equation}\label{eq:quart}
\frac{\partial \theta}{\partial t}=\frac{n}{2t}\theta-\frac{1}{4ts} J+\left(\frac{ 1}{s}-\frac{1}{t}\right)\uppi -\frac{n}{2}\left(\frac{ 1}{s}-\frac{1}{t}\right)\Gamma\theta-\Gamma\left(\frac{ 1}{s}-\frac{1}{t}\right) \frac{s}{t}\int_M \frac{|\nabla U|^2 }{16s}\,e^{-\frac{U}{4s}}\frac{\di\nu_g}{(4\pi s)^{\frac n2}} \cdot\end{equation}

We use now (\ref{eq:ter}) to express in a different way the last term of the right hand side of (\ref{eq:quart}):
\begin{equation}\label{eq:quint}\begin{split}
\frac{\partial \theta}{\partial t}&=\frac{n}{2t}\theta-\frac{1}{4ts} J+\left(\frac{ 1}{s}-\frac{1}{t}\right)\uppi -\frac{n}{2}\left(\frac{ 1}{s}-\frac{1}{t}\right)\Gamma\theta\\&
+\Gamma\frac{s}{t}\left(\frac{\partial \theta}{\partial t}-\frac{n}{2s}\theta+\frac{J}{4ts} \right)\cdot\end{split}\end{equation}
Using (\ref{eq:derives}) we get
$$-\frac{n}{2s}\theta+\frac{J}{4ts}=\frac{n}{2}\left(\frac1t-\frac 1s\right)\theta+\frac{s}{t}\frac{\partial \theta}{\partial s},$$
hence 
\begin{equation}\label{eq:six}\begin{split}
\frac{\partial \theta}{\partial t}&=-\frac{s}{t}\frac{\partial \theta}{\partial s}+\left(\frac{ 1}{s}-\frac{1}{t}\right)\uppi -\frac{n}{2}\left(\frac{ 1}{s}-\frac{1}{t}\right)\Gamma\theta\\&
+\frac{n}{2}\frac{ \Gamma}{t}\left(\frac{s}{t}-1\right)\theta+\Gamma\frac{s}{t}\left(\frac{\partial \theta}{\partial t}+\frac{s}{t}\frac{\partial \theta}{\partial s} \right)\cdot\end{split}\end{equation}
And we get the following differential equation:
\begin{equation}\label{eq:set}
\left(1-\Gamma\frac{ s}{t} \right)\left(t\frac{\partial \theta}{\partial t}+s\frac{\partial \theta}{\partial s}\right)=\left(\frac{ t}{s}-1\right)\uppi 
-\frac{n}{2} \Gamma\left(\frac{t}{s}-\frac{s}{t}\right)\theta\end{equation}

When we notice that 
$$s<\frac{t}{2 \Gamma}\Rightarrow \frac12\le 1-\Gamma\frac{ s}{t},$$ we get that 
$$t\frac{\partial \theta}{\partial t}+s\frac{\partial \theta}{\partial s} 
+n \Gamma\left(\frac{t}{s}-\frac{s}{t}\right)\theta$$ has the same sign as $t-s$.
\end{proof}

The above proposition has the following immediate consequence under a stronger integral bound. 
 
\begin{cor}\label{cor:monoKato} Let  $(M^n,g)$ be a closed manifold satisfying :
$$\int_0^{T}\frac{\sqrt{\Kato_s(M^n,g)}}{s} \di s \leq \Lambda \cdot$$
For any $\tau \in (0,T]$ we set  $$\Phi(\tau)=\int_0^{\tau}\frac{\sqrt{\Kato_s(M^n,g)}}{s} \di s.$$
Then there exists a positive constant $c_n$, depending only on $n$, such that the following holds. Let $t\in (0,T]$ and $s>0$.  Set
 $\overline{\lambda}:=\overline{\lambda}(s,t)=\min\left\{e^{-c_n\Lambda \frac{s}{t}}, e^{-4\sqrt{n}\Lambda}\right\}$. Then the  function 
 $$\lambda\in [0,\overline{\lambda}]\rightarrow \theta_x(\lambda s,\lambda t)\ e^{c_n\Phi(\lambda t)\left(\frac{t}{s}-\frac{s}{t}\right)}$$ is monotone:
 \begin{enumerate}[i)]
 \item it is non decreasing if $t\ge s$,
 \item it is non increasing if $t\le s$.
 \end{enumerate}
 \end{cor}
 \begin{proof}
For any $c\in \R$, $\lambda, s>0$ and $t\in (0,T]$, a direct computation of the derivative $\frac{\di}{\di \lambda} \left( \theta_x(\lambda s,\lambda t)e^{c\Phi(\lambda t)\left(\frac{t}{s}-\frac{s}{t} \right)}\right)$ shows that it has same sign as
\[
\lambda t\frac{\partial \theta_x(\lambda s,\lambda t)}{\partial t}+\lambda s\frac{\partial \theta(\lambda s,\lambda t)}{\partial s} +c\left(\frac{t}{s}-\frac{s}{t} \right)\sqrt{\Kato_{\lambda t}}\theta(\lambda s,\lambda t).
\]
If $\lambda \le e^{-4\sqrt{n} \Lambda}$, then $\Kato_{\lambda t}\le (16n)^{-1}$, and since there exists $b_n>0$ such that $e^{8\sqrt{nr}}-1\le b_n \sqrt{r}$ for any $r \in (0,(16n)^{-1}]$, this yields
\begin{equation}\label{eq:contr}
\Gamma_{\lambda t}\le b_n \sqrt{\Kato_{\lambda t}} \le b_n\frac{1}{\log(1/\lambda) }\int_{\lambda t}^t \frac{\sqrt{\Kato_s}}{s} \di s \le b_n\frac{\Lambda}{\log(1/\lambda) } \,\cdot\end{equation}
When choosing $c=c_n:=n b_n$ we get that $\frac{\di}{\di \lambda} \left( \theta_x(\lambda s,\lambda t)e^{c_n\Phi(\lambda t)\left(\frac{t}{s}-\frac{s}{t} \right)}\right)$ has same sign as 
\begin{align}
\label{concl}
& \lambda t\frac{\partial \theta_x(\lambda s,\lambda t)}{\partial t}+\lambda s\frac{\partial \theta(\lambda s,\lambda t)}{\partial s} + n\Gamma_{\bar{\lambda} t} \left(\frac{t}{s}-\frac{s}{t} \right)\theta_x(\lambda s, \lambda t) \nonumber\\
& \quad + \left(\frac{t}{s}-\frac{s}{t} \right)\left(c_n\sqrt{\Kato_{\lambda t}}-n\Gamma_{\bar{\lambda} t} \right)\theta_x(\lambda s,\lambda t).\end{align}
In addition, if $\lambda \le e^{-\frac{c_n}{n}\Lambda \frac{2s}{t}}$ in \eqref{eq:contr}, we get $\Gamma_{\lambda t} \le\frac{t}{2s}$.  In particular,  we have $s\le\frac{t}{2\Gamma_{\bar{\lambda} t}}$. Hence by Proposition \ref{prop:derivetheta} when $\lambda \le \bar{\lambda}$ the first three summands in the right-hand side of \eqref{concl} give a term of the same sign as $t-s$. Moreover,  with our choice of $c_n$ we have
$$c_n\sqrt{\Kato_{\lambda t}}-n\Gamma_{\bar{\lambda} t} >0,$$
hence the last term of \eqref{concl} also has the same sign as $t-s$. This concludes the proof.

 \end{proof}

\subsection{Consequences on non-collapsed strong Kato limits}

We are now in position to prove that tangent cones of strong non-collapsed Kato limits are measure metric cones. 

Throughout this section, we fix constants $T, \Lambda, v >0$ and a non-decreasing function $f : (0,T] \to \R_+$ such that 
$$ \int_0^T \frac{\sqrt{f(s)}}{s}\di s \leq \Lambda.$$
Without loss of generality, we assume that $f(T) \leq 1/(16n)$ and set 
$$\Phi(\tau)=\int_0^\tau \frac{\sqrt{f(t)}}{t}dt.$$
According to Definition \ref{def:NC}, a non-collapsed strong Kato limit $(X,\dist,\mu,o)$ is obtained as 
$$(M_\alpha, \dist_\alpha, \nu_\alpha, o_\alpha) \stackrel{pmGH}{\longrightarrow}(X, \dist, \mu,o),$$
where $\{(M_\alpha, g_\alpha)\}_\alpha$ are closed manifold satisfying the uniform estimates 
$$\mbox{for all } t \in (0,T], \ \sup_\alpha \Kato_t(M_\alpha, g_\alpha) \leq f(t),$$
$$\inf_\alpha \nu_{g_{\alpha}}(B_{\sqrt{T}}(o_\alpha))\geq vT^{\frac n2}.$$

\begin{theorem}\label{thm:metriccone}
Let $(X,\dist,\mu,o)$ be a non-collapsed strong Kato limit in the sense of Definition \ref{def:NC}, and $x \in X$. Then the following holds.
\begin{enumerate}
\item[\emph{(i)}] Any tangent cone of $X$ at $x$ is a measured metric cone.
\item[\emph{(ii)}] The volume density 
$$\vartheta_X(x)=\lim_{r\to 0} \frac{\mu(B_r(x))}{\omega_nr^n}$$ is well defined.
\item[\emph{(iii)}] We have the following relationship between the behavior of the $\theta$-volume and the volume density :
$$\lim_{\lambda\to 0} \theta_x(\lambda s,\lambda t)=\vartheta_X^{1-\frac t s}(x).$$
\end{enumerate}
\end{theorem}

The proof of this theorem and Remark \ref{rem-theta-vol} (ii) 
will imply:
\begin{cor}\label{cor:monoHeat} Let $(X, \dist, \mu,o)$ be a non-collapsed strong Kato limit. Then for any $x \in X$ we have
\begin{equation*}
\lim_{t \rightarrow 0} (4\pi t)^{\frac n2} H(t,x,x)= \frac{1}{\vartheta_X(x)}
\end{equation*}
Moreover there is a positive constant $\upeta$ and a increasing function $\Phi\colon (0,\upeta T]\rightarrow \R_+$ satisfying $\lim_{t\to 0+} \Phi(t)=0$ such that 
$$t\in (0,\upeta T]\mapsto \exp\left(\Phi(t) \right)(4\pi t)^{\frac n2} H(t,x,x)$$ is non decreasing. \end{cor}

\begin{proof}[Proof of Theorem \ref{thm:metriccone}]
Let $(X,\dist, \mu)$ be a non-collapsed strong Kato limit obtained as above. According to Theorem \ref{thm:int_ChE}, we know that we have Mosco convergence of the quadratic form
$$u\mapsto \int_{M_\alpha} |du|^2_{g_\alpha}\di\nu_{g_\alpha}$$ to the Cheeger energy $\Ch$ of $(X, \dist, \mu)$. As a consequence, $(X, \dist, \mu,\Ch)$ is a PI$_{\upkappa, \upgamma}(\sqrt{T})$ Dirichlet space. Therefore, the $\Theta$ and $\theta$-volume are well defined on $X$. Moreover we know that if $x_\alpha\to x$ then 
 $$\Theta^X_x(s)=\lim_{\alpha\to+\infty} \Theta^{M_\alpha}_{x_\alpha}(s)\text{ and } \theta^X_x(s,t)=\lim_{\alpha\to+\infty} \theta^{M_\alpha}_{x_\alpha}(s,t).$$
According to Corollary \ref{cor:monoKato}, we also know that for any $t,s>0$ and $t>0$ with $t\le T$ there is a constant $\eps>0$ and a constant $\kappa$ both depending on $s$ and $t$, such that the function
 $$\lambda\in (0,\eps]\mapsto \theta^{M_\alpha}_{x_\alpha}(\lambda s,\lambda t)e^{\kappa \Phi(\lambda t)}$$ is monotone.
Hence the same is true for the function
 $$\lambda\in (0,\eps]\mapsto \theta^{X}_{x}(\lambda s,\lambda t)e^{\kappa \Phi(\lambda t)}.$$
Recall that, as we observed in Remark \ref{rem-AR-ncStrongKatolim} and since $(X,\dist, \mu)$ is a non-collapsed strong Kato limit, the measure $\mu$ has a local Ahlfors regularity and satisfies for all $x \in X$ and $0\le r\le s\le \sqrt{T}/2$
 $$\mu(B_s(x))\le C s^n\text{ and } \frac{\mu(B_s(x))}{\mu(B_r(x))}\le C\left(\frac{s}{r}\right)^n.$$
We also have the uniform lower bound:
$$vs^n\, e^{-\lambda\frac{\dist(o,x)}{\sqrt{T}}}\le \mu(B_s(x)).$$
Hence by the estimate (\ref{eq:compathetavolu}), there exist positive constants $c,C$ depending on $d(o,x)$ and on $t/s$ such that
$$c\le \theta_x(\lambda s,\lambda t)\le C.$$ As a consequence, this monotone function has a well defined limit at $\lambda=0$, that we denote by $\vartheta_x(s,t)$. Moreover when $t\le s$:
 $$\vartheta_x(s,t):=\lim_{\lambda\to 0+} \theta^{X}_{x}(\lambda s,\lambda t)=\sup_{\lambda\in (0,\eps)} \theta^{X}_{x}(\lambda s,\lambda t)e^{\kappa \Phi(\lambda t)},
 $$
 when $t\ge s$:
 $$\vartheta_x(s,t):=\lim_{\lambda\to 0+} \theta^{X}_{x}(\lambda s,\lambda t)=\inf_{\lambda\in (0,\eps)} \theta^{X}_{x}(\lambda s,\lambda t)e^{\kappa \Phi(\lambda t)}\cdot
 $$
 By construction the function $(s,t)\mapsto \vartheta_x(s,t)$ is $0$-homogeneous :
 $$\forall \lambda\in (0,1)\colon  \vartheta_x( \lambda s, \lambda t)= \vartheta_x(s,t).$$
 
Let $x \in X$ and $(Y,\dist_Y,\mu_Y, y)$ be a tangent cone of $(X,\dist,\mu)$ at $x$. Then there exist a sequence $\{\eps_\beta\}_\beta \in (0,+\infty)$ with $\eps_\beta \downarrow 0$ and a limit measure $\mu_Y$ on $Y$ such that
\begin{equation}\label{eq:notat}
(X_\beta=X,\dist_\beta:=\eps_\beta^{-1} \dist, \mu_\beta:=\eps_\beta^{-n} \mu,x) \stackrel{pmGH}{\longrightarrow} (Y,\dist_Y,\mu,y).
\end{equation}
Moreover the sequence  $\left( (X,\dist_\beta:=\eps_\beta^{-1} \dist, \mu_\beta:=\eps_\beta^{-n} \mu, \Ch_\beta=\eps_\beta^{2-n} \Ch, x)\right)_\beta$ Mosco-Gromov-Hausdorff converges to $(Y,\dist_Y,\mu,\Ch, y)$.
Let $H_{\beta}$ be the heat kernel of the scaled Dirichlet spaces $(X_{\beta}, \dist_{\beta}, \mu_{\beta},\Ch_\beta )$ and $U_{\beta}$ the corresponding function such that
$$H_{\beta}(t,x,y)=(4\pi t)^{-\frac{n}{2}} e^{-\frac{U_\beta(t,x,y)}{4t}}.$$
The scaling property of the heat kernel leads to
$$U_\beta(t,x,y)=\eps_\beta^{-2}U(\eps_\beta^2t,x,y),$$
so that we have
$$\theta_x^{X_\beta}(s,t)=\theta_x^{X}(\eps_\beta^2s,\eps_\beta^2 t).$$
Moreover, Proposition \ref{prop:thetacont} ensures that for any $s,t>0$ we have
$$\theta_y^{Y}(s,t)=\lim_{\beta\to +\infty} \theta_x^{X_\beta}(s,t)=\lim_{\beta\to+\infty}\theta_x^{X}(\eps_\beta^2s,\eps_\beta^2 t)=\vartheta_x(s,t).$$
Hence the function $(s,t)\to \theta_y^{Y}(s,t)$ is also $0$-homogeneous, but according to 
Proposition \ref{prop:theta(s,t)}, we get that 
$$\Theta^Y_y(s)=\lim_{t\to 0} \theta^Y_y(s,t).$$ Hence
$s\mapsto \Theta_y^Y(s)$ is $0$-homogeneous hence there is some $c=\Theta_y^Y(1)>0$ such that for any $s>0$:
$$\Theta_y^Y(s)=c.$$

 As a consequence, Lemma \ref{lem:simple} implies for any $r>0$
$$\mu_Y(B(y,r))=c \omega_n r^n.$$
From Corollary \ref{cor:TgConeRCD}, we know that $(Y,\dist_Y,\mu_Y)$ is a weakly non collapsed $\RCD(0,n)$ space, then according to Proposition \ref{prop:DPG},
we know that $(Y,\dist_Y,\mu_Y)$ is a measured metric cone at $y$. This shows the first assertion (i). 
 
We have also shown that for any point $x \in X$ and for any tangent cone $(Y, \dist_Y, \mu_Y, y)$ of $X$ at $x$, the functions $\Theta_y^Y$ and $\theta_y^Y$ do not depend on the tangent cone $Y$.  

Recall that, as it was pointed out earlier in the proof, for any fixed $x\in X$ the function $r\in (0,\sqrt{T}/2]\, \mapsto \mu(B_r(x))/\omega_n r^n$ is bounded above and below by positive constants, hence it admits limit points as $r \downarrow 0$. Let $\varpi$ be one of these limits points and $(r_\alpha)_{\alpha} \subset (0,+\infty)$ a sequence such that $r_\alpha \downarrow 0$ and $\varpi = \lim_\alpha \mu(B(x,r_\alpha))/\omega_n r_\alpha^n$. 
We can assume, extracting a subsequence, that the sequence of rescaled space $(X,\dist_\alpha:=r_\alpha^{-1} \dist, \mu_\alpha:=r_\alpha^{-n} \mu,  x)$ converges for the pointed measure Gromov-Hausdorff topology to some tangent cone $(Y, \dist_Y, \mu_Y, y)$. In particular
$$\varpi=\frac{\mu_Y(B_1(y))}{\omega_n}= \Theta^Y_y(1)=\lim_{t\to 0} \vartheta_x(s,t).$$
All limit points are hence equals and 
 we obtain that the volume density is well defined.
 
It remains to show that 
$$\lim_{\lambda\to 0} \theta_x(\lambda s,\lambda t)=\vartheta_X^{1-\frac t s}(x),$$
that is to  verify that $$\vartheta_x(s,t)=\vartheta_X^{1-\frac t s}(x).$$ If $(Y, \dist_Y, \mu_Y, y)$ is a tangent cone of $X$ at $x$, then we have shown that
 $$\theta_y^Y(s,t)=\vartheta_x(s,t).$$
Since $(Y, \dist_Y, \mu_Y, y)$ is a  measure metric cone at $y$, we get that for any $z\in Y$ and $t>0$:
\begin{equation*}
H_{Y}(t,y,z)=\frac{1}{\vartheta_Y(y) (4\pi t)^{\frac n2}}e^{\frac{-\dist^2_{Y}(y,z)}{4t}}, 
\end{equation*}
where we recall that, since $Y$ is a measure metric cone at $y$ and a tangent cone of $X$ at $x$, for any $r>0$: $$\vartheta_Y(y)=\frac{\mu(B_r(y))}{\omega_n r^n}=\vartheta_X(x).$$
Therefore the function $U_{Y}$ associated to $H_{Y}$ equals 
$$U_{Y}(t,y,z)=\dist^2_{Y}(y,z) +4t\vartheta_Y(y)=\dist^2_{Y}(y,z) +4t\vartheta_X(x).$$
 
When using this equality in the definition of $\theta^{Y}_y(s,t)$ we obtain
\begin{align*}
\theta^{Y}_y(s,t)&= \int_{Y}e^{-\frac{U_{Y}(t,y,z)}{4s}}\frac{\di\mu_{Y}(z)}{(4\pi s)^{\frac n2}} \\
& = \vartheta_Y(y)^{-\frac{t}{s}} \int_{Y}e^{-\frac{\dist^2_{Y}(y,z) }{4s}}\frac{\di\mu_{Y}(z)}{(4\pi s)^{\frac n2}} \\
&=\vartheta_Y(y)^{1-\frac{t}{s}} \int_{Y} H_Y(s,y,z)\di\mu_{Y}(z)\\
&=\vartheta_Y(y)^{1-\frac{t}{s}}=\vartheta_X(x)^{1-\frac{t}{s}},
\end{align*}
where we have used the stochastic completeness of $Y$.
 \end{proof}
 
 This theorem has also the following useful consequence. 
 \begin{cor}\label{cor:Hausdorffvsmu1}
 Let $(X,\dist,\mu,o)$ be a non-collapsed strong-Kato limit in the sense of Definition \ref{def:NC}. Then at every point $x\in X$ the volume density satisfies:
 $$\vartheta_X(x)\le 1.$$
 \end{cor}

 \begin{proof} Let $(X,\dist, \mu,o)$ be a non-collapsed strong Kato limit defined as above and recall that we defined for all $t \in (0,T]$
$$\Phi(\tau)= \int_0^t \frac{\sqrt{f(s)}}{s} \di s.$$
Let $x\in X$.  We only need to show that $$\lim_{t\to 0} \theta^X_x(t/4,t/2)=\vartheta_X(x)^{-1}\ge 1.$$ 
Using Corollary \ref{cor:monoHeat} we know that for some $\upeta>0$, the function
$$t\in (0,\upeta T]\mapsto \exp\left( \frac{\Phi(t)}{\upeta} \right)(4\pi t)^{\frac n2} H_{M_\alpha}X(t,x_\alpha,x_\alpha) $$ is non decreasing.
But $$\lim_{t\to 0+} (4\pi t)^{\frac n2} H_{M_\alpha}X(t,x_\alpha,x_\alpha) =1$$ hence for any $t\in [0,\upeta T]$:
$$\theta_{x_\alpha}^{M_\alpha}(t/4,t/2)=(4\pi t)^{\frac n2} H_{M_\alpha}X(t,x_\alpha,x_\alpha)\ge   \exp\left(-\frac{\Phi(t)}{\upeta}\right).$$
By Proposition \ref{prop:thetacont}, we also have
$$\theta_{x}^{X}(t/4,t/2)\ge   \exp\left( -\frac{\Phi(t)}{\upeta} \right).$$
Hence the result. \end{proof}
Our next result concerns the measure of balls in a non-collapsed strong Kato limit and the comparison between the $n$-Hausdorff measure and $\mu$.
  \begin{cor}\label{cor:Hausdorffvsmu2} 
 Let $(X,\dist,\mu,o)$ be a non-collapsed strong-Kato limit in the sense of \ref{def:NC}, then for any $\rho>0$ and $\eps>0$ there is some $\delta>0$ such that $r\le \delta$ and $x\in B_\rho(o)$:
$$ \mu(B_r(x))\le \omega_nr^n(1+\epsilon).$$
As a consequence, $\mu$ is absolutely continuous with respect to the $n$-Hausdorff measure and $$\mu\le \cH^n.$$
\end{cor}
To prove this corollary we will use, as in \cite{CheegerPisa}, the spherical Hausdorff measure defined  for any $s>0$ and any Borel set $A$ in a metric space $(Z,\dist_Z)$ by
$$\cH^s(A):=\lim_{\delta\to 0+} \cH_\delta^s(A),$$ where for any $\delta \in (0,+\infty]$, $$\cH^s_\delta(A):=\inf\left\{\sum_i \omega_s r_i^s : A\subset \cup_i B_{r_i}(x_i) \text{ and } \forall i\colon r_i<\delta\right\}.$$Following \cite[Theorem 3.6]{Simon} or \cite[Theorem 6.6]{Mattila}, we have 
the following result : if $\cH^s(A)<\infty$ then 
\begin{equation}\label{Hausdorff}
\limsup_{r\to 0} \frac{\cH^s(B_r(x)\cap A)}{\omega_s r^s}\le 1
\end{equation}
for $\cH^s-\text{a.e. }x\in A$.

\begin{proof}[Proof of Corollary \ref{cor:Hausdorffvsmu2}]
Indeed if the estimate were not true, then we would find $\eps>0$ and sequences $r_\alpha \downarrow  0$, $x_\alpha\in B_\rho(o)$, such that the sequence of re-scaled spaces $(X, r_\alpha^{-1} \dist, r_\alpha^{-n}\mu, x_\alpha)$ converges to some pointed metric measure space $(Z, \dist_Z, \mu_Z,z)$ with $$\mu_Z(B_1(z))\ge \omega_n(1+\eps).$$ 
By Remark \ref{rem-TgConesSKL}, we know that $(Z, \dist_Z, \mu_Z,z)$ is a non-collapsed strong Kato limit as well, then by Corollary \ref{cor:Hausdorffvsmu1} its volume density is smaller than 1. Moreover, $(Z, \dist_Z, \mu_Z,z)$ is obtained as a limit of re-scaled manifolds $(M_\alpha, \tilde{g}_\alpha)$ such that for all $t>0$
$$\lim_{\alpha \to \infty} \Kato_t(M_\alpha,\tilde g_\alpha) =0.$$
Then according to Remark \ref{rem-RCDmoreGeneral}, $(Z, \dist_Z, \mu_Z,z)$ is a weakly non collapsed $\RCD(0,n)$ space. Thus the Bishop-Gromov comparison theorem holds on $(Z, \dist_Z, \mu_Z,z)$ and we get
$$\mu_Z(B_1(z))\le \vartheta_Z(z)\omega_n,$$ hence a contradiction.
The comparison with  the Hausdorff measure is then straightforward, because if $A\subset B_\rho(o)$ and $\eps>0$ we find $\delta\in (0,1)$ such that $x\in B_{\rho+1}(o)$ and $r<\delta$ yields $ \mu(B_r(x))\le \omega_nr^n(1+\epsilon),$ so that $$\mu(A)\le (1+\epsilon) \cH^n(A).$$
 \end{proof}
 \begin{rem} The above volume estimate can in fact be quantified on smooth closed manifolds. 
\label{rem:estiBishop} Let $v,T,\Lambda>0$ and $f:(0,T] \to \R_+$ be  a non-decreasing function such that
$$f(T)\le \frac{1}{16n}\text{ and }\int_{0}^T \frac{\sqrt{f(s)}}{s}\di s \leq \Lambda.$$
 Then for any $\rho>0$ and $\eps>0$, there exists  $\delta>0$ depending only on $n$, $f$, $v$, $T$ such that if $(M^n,g)$ is a closed Riemannian manifold such that
 $$v\le \frac{\nu_g\left(B_{\sqrt{T}}(o)\right)}{T^{\frac{n}{2}}}\text{ and }\forall t\in (0,T], \ \Kato_t(M, g) \leq f(t),$$
 then for any $x\in B_{\rho}(o)$ and $r<\delta$:
 $$\nu_g(B_r(x))\le \omega_nr^n(1+\epsilon).$$
\end{rem}

\section{Stratification}

In this section, we prove a stratification theorem for non-collapsed strong Kato limit spaces.  To state this result, we first give a useful definition.  From now on we equip $\setR^k$ with the classical Euclidean distance which we write $\dist_e$, whatever $k \in \setN \backslash \{0\}$.

\begin{D} For $k \in \N \backslash \{0\}$, a pointed metric measure space $(Y,\dist_Y,\mu_Y,y)$ is called \textbf{metric measure $k$-symmetric} (mm $k$-symmetric for short) if there exists a metric measure cone $(Z,\dist_Z,\mu_Z)$ with vertex $z$ such that
\[
(Y,\dist_Y,\mu_Y,y) = (\setR^k\times Z, \dist_{\setR^k \times Z}, \mathcal{H}^k \otimes \mu_Z,(0_k,z)),
\]
where $\dist_{\setR^k \times Z}$ is the classical Pythagorean product distance,  $0_k$ is the origin of $\setR^k$, and the equality sign means that there exists an isometry $\varphi:Y \to \setR^k \times Z$ such that $\varphi_\# \mu_Y =  \cH^n \otimes \mu_Z$ and $\varphi(y)=(0_k,z)$.
\end{D}

Let $\dim_\cH A$ be the Hausdorff dimension of a subset $A$ of a metric space $(X,\dist)$. Then our stratification theorem writes as follows.

\begin{theorem}\label{thm:stratification}
Let $(X,\dist,\mu,o)$ be a non-collapsed strong Kato limit.  We set for any $x \in X$
\[
d(x):=\sup\{k \in \setN : \text{one tangent cone at $x$ is mm $k$-symmetric}\} \in \{0,\ldots,n\}
\]
and $S^k:=\{x \in X : d(x) \le k\}$ for any $k \in \{0,\ldots,n\}$. Then the sets $S^k$ define a filtration of $X$
$$S^0 \subset S^1 \subset \ldots \subset S^{n-1} \subset S^n,$$
and the following holds. 
\begin{itemize}
\item[(i)] The set $S_0$ is countable. 
\item[(ii)] For any $k \in \{1,\ldots,n\}$ we have
$$\dim_\cH S^k \le k.$$
\item[(iii)] For $\mu$-a.e.~$x \in X$ the set of tangent cones of $(X,\dist,\mu)$ at $x$ is reduced to $\{(\setR^n,\dist_e,\vartheta_X(x)\cH^n,x)\}$.
\end{itemize}
\end{theorem}

Our proof of Theorem \ref{thm:stratification} is based on two intermediary results and an argument originally due to B.~White \cite{White}. The key point in this argument consists in dealing with an appropriate upper or lower semi-continuous function, hence we begin with showing that the volume density $\vartheta_X$, that is well defined at any point of a non-collapsed strong Kato limit thanks to Theorem \ref{thm:metriccone}, is lower semi-continuous. 

For the sake of clarity,  like in the previous subsection, we fix constants $T, \Lambda, v >0$ and a non-decreasing function $f : (0,T] \to \R_+$ such that $f(T) \leq 1/(16n)$ and 
$$\int_0^T \sqrt{f(\tau)}\di \tau/\tau \leq \Lambda.$$
Any non-collapsed strong Kato limit space $(X,\dist,\mu,o)$ considered in this section is the pmGH limit of a sequence of pointed Riemannian manifolds $\{(M_\alpha,g_\alpha,o_\alpha)\}$ satisfying 
$$\mbox{for all } t \in (0,T], \quad \Kato_t(M_\alpha, g_\alpha) \leq f(t)$$
$$\mbox{for all }  \alpha, \quad \nu_{g_{\alpha}}(B_{\sqrt{T}}(o_\alpha))\geq vT^{\frac n2}.$$

\begin{prop}
\label{prop:lowerSemicontinuity}
Let $(X,\dist,\mu,o)$ be a non-collapsed strong Kato limit.  Then the function $\vartheta_X$ is lower semicontinuous.  Moreover, if $\{(X_\alpha,\dist_\alpha,\mu_\alpha,o_\alpha)\}_\alpha$  is a pmGH convergent sequence of non-collapsed strong Kato limit spaces with limit $(X,\dist,\mu,o)$, then for any $x \in X$ and any sequence $\{x_\alpha\}$ where $x_\alpha \in X_\alpha$ for any $\alpha$ such that $x_\alpha \to x$,
\[
\vartheta_{X}(x) \le \liminf_{\alpha \to +\infty} \vartheta_{X_\alpha}(x_\alpha).
\]
\end{prop}

\begin{proof}
Recall that the infimum of a family of continuous functions is upper semi-continuous. The proposition is then a consequence of the fact that $\vartheta^{-1}$ is the infinimum of a family of continuous function. 
Indeed from Corollary \ref{cor:monoHeat} and Remark \ref{rem-theta-vol} (ii) we know that there exists $\upeta>0$ and an increasing function $\Phi\colon (0,\upeta T]\rightarrow \R_+$ satisfying $\lim_{t\to 0+} \Phi(t)=0$ such that the function
$$t\in (0,\upeta T]\mapsto \exp\left(\Phi(t) \right)(4\pi t)^{\frac n2} H(t,x,x)$$  is non decreasing. We also know that
$$\vartheta^{-1}_X(x)=\inf_{t\in (0,\upeta]}  \exp\left( \Phi(t) \right) \theta_x^X(t/4,t/2).$$
The result then follows from Proposition \ref{prop:thetacont}.
\end{proof}

The next additional result deals with the volume density of weakly non-collapsed  $\RCD(0,n)$ measure metric cones and was implicitly present in \cite[Lem.~2.9]{DPG}.

\begin{prop}
\label{prop:splitting}
Let $(Y,\dist,\mu)$ be a weakly non-collapsed $\RCD(0,n)$ space which is a $n$-metric measure cone with vertex $y \in Y$. Then $\vartheta_Y(y') \ge \vartheta_Y(y)$ for any $y' \in Y$. Moreover there exists $k \in \N$ such that the level set $\{\vartheta_Y(\cdot) = \vartheta_Y(y)\}$ is isometric to the Euclidean space $\setR^k$ and $(Y,\dist,\mu,y)$ is mm $k$-symmetric but not mm $(k+1)$-symmetric.
\end{prop}

\begin{proof}
By the Bishop-Gromov theorem for $\RCD(0,n)$ spaces, the volume ratio is non-increasing, hence we know that for any $y' \in Y$ and $r>0$,
$$\vartheta_Y(y')\geq \frac{\mu(B_r(y'))}{\omega_n r^n}\ge \inf_{s>0} \frac{\mu(B_s(y'))}{\omega_n s^n}=\lim_{s\to+\infty} \frac{\mu(B_s(y'))}{\omega_n s^n} \, \cdot$$
The Bishop-Gromov theorem classically implies that the asymptotic volume ratio $ \lim\limits_{s\to+\infty} \frac{\mu(B_s(y'))}{\omega_n s^n}$ does not depend on $y'\in Y.$ Thus
$$\vartheta_Y(y')\geq \lim_{s\to+\infty} \frac{\mu(B_s(y))}{\omega_n s^n}\, \cdot$$
Since $(Y,\dist,\mu)$ is a $n$-metric measure cone with vertex $y$, the function $s \mapsto \frac{\mu(B_s(y))}{\omega_n s^n}$ is constantly equal to $\vartheta_Y(y)$.  As a consequence,
$$\vartheta_Y(y') \geq  \lim_{s\to+\infty} \frac{\mu(B_s(y))}{\omega_n s^n} =\vartheta_Y(y).$$

Since for any $r>0$,
$$
\vartheta_Y(y') \ge  \frac{\mu(B_r(y'))}{\omega_n r^n} \ge \vartheta_Y(y),
$$
then  $\vartheta_Y(y') = \vartheta_Y(y)$  if and only if the function $r \mapsto \frac{\mu(B_r(y'))}{\omega_n r^n}$ is constantly equal to $\vartheta_Y(y)$, what occurs if and only if $Y$ is a $n$-metric measure cone at $y'$ thanks to Proposition \ref{prop:DPG}.

If $y' \neq y$, this implies that $Y$ is mm 1-symmetric along the geodesic connecting $y$ and $y'$. Indeed, since $Y$ is a metric cone at $y'$, it must be isometric to any tangent cone at $y'$. But $Y$ is a metric cone at $y$, say $(Y, \dist, \mu)=(C(Z), \dist, \mu)$: therefore, any tangent cone $Y_{y'}$ at $y'=(s,z) \neq y$ is of the form
$$(\R \times Z_z, \dist_{\R \times Z_z}, s^{n-1}dt d\nu_z,y'),$$
where $s=\dist(y,y')$, and $(Z_z, \dist_z,\nu_z)$ is a tangent cone of $Z$ at $z$.  
\end{proof}

Recall that, thanks to Corollary \ref{cor:TgConeRCD} and to Theorem \ref{thm:metriccone}, any tangent cone of a non-collapsed strong Kato limit space is a weakly non-collapsed $\RCD(0,n)$ $n$-metric measure cone. Then in our setting we directly obtain the following reformulation of Proposition \ref{prop:splitting}. 

\begin{cor}
\label{cor:upthetaSplitting}
Let $(X,\dist,\mu)$ be a non-collapsed strong Kato limit and $x \in X$.  Let $(X_x,\dist_x,\mu_x,x)$ be a tangent cone. Then $\vartheta_{X_x}(z)\geq \vartheta_{X_x}(x)$ for any $z \in X_x$, and there exists $k \in \N$ such that the level set $\{ \vartheta_{X_x}(\cdot) =\vartheta_{X_x}(x) \}$ is isometric to the Euclidean space $\R^k$ and $(X_x,\dist_x,\mu_x, x)$ is mm $k$-symmetric but not mm $(k+1)$-symmetric.
\end{cor}

Before proving Theorem \ref{thm:stratification}, we recall the definition of $\cH_\infty^s$ from the previous section and provide a couple of classical properties. For any $s \in \R_+$ and any subset $E$ of a metric space $(X,\dist)$,
$$\cH^s_{\infty}(E):= \inf \left\{\sum_i \omega_s r_i^s \, : \, E \subseteq \bigcup_i B_{r_i}(x_i) \ \right\}.$$

\begin{lemma}\label{lemHaus}
The function $\cH^s_{\infty}$ satisfies the following properties. 
\begin{enumerate}
\item $\mbox{dim}_{\cH}(E)=\sup\{s >0 :  \cH^s_{\infty}(E)>0\}=\inf\{ s >0:  \cH^s_{\infty}(E)=0 \}$; 
\item if $\cH^s_\infty(E)>0$, then for $\cH^s$-almost every $x\in E \colon $
$$ \limsup_{r\to 0+} \frac{\cH^s_\infty(E\cap B_r(x))}{\omega_sr^s}\ge 2^{-s};$$
\item if $E$ is a countable union of sets $\{E_j\}$, then $\cH^s_\infty(E)>0$ if and only if there exists $j$ such that $\cH^s_\infty(E_j)>0$;
\item $\cH^s_\infty$ is upper semi-continuous with respect to the Gromov-Hausdorff convergence of compact metric sets.
\end{enumerate}
\end{lemma}

We are now in a position to prove the existence of a well-defined stratification for non-collapsed strong Kato limits. 

\begin{proof}[Proof of Theorem \ref{thm:stratification}] The proof is divided in three cases: first the case $k=0$, then the case $k\in \{1,\ldots,n-1\}$, and eventually the case $k=n$.\\

\textbf{Case I, $k=0$}.

Our argument to prove the assertion about $S_0$ is inspired by \cite[Proposition 3.3]{White}. It suffices to prove the inclusion
\begin{equation*}
S_0\subseteq \cM:=\{x\in X, \ \vartheta_X(x)<\liminf_{y\to x} \vartheta_X(y)\}.
\end{equation*}
Indeed, $\cM$ is countable as it can be written as the union over $\ell \in \N\backslash \{0\}$ of the sets
$$\cM_\ell=\left\{x \in X \, : \, \forall y \in B_{1/\ell}(x)\setminus \{x\}, \ \vartheta_X(x)+\frac 1\ell < \vartheta_X(y) \right\}$$
which are all discrete and countable, since whenever two disjoint points $x,y$ are in $\cM_\ell$ they satisfy $\dist(x,y)\ge 1/\ell$.

To show the inclusion $S_0 \subseteq \cM$,  let us take $x\not\in \cM$. Then there exists a sequence $\{y_\ell\}_\ell \subset X$ such that $0<r_\ell:=\dist(x,y_\ell)<1/\ell$ and $\vartheta_X(x)+1/\ell \ge  \vartheta_X(y_\ell)$ for any $\ell \in  \N\backslash \{0\}$. In particular,  $y_\ell\to x$ and
\begin{equation*}
\lim_\ell \vartheta_X(y_\ell)=\vartheta_X(x).
\end{equation*}
Consider the sequence of rescalings $\{(X, r_\ell^{-1}\dist, r_\ell^{-n}\mu, x)\}_\ell$. Since $r_\ell \downarrow 0$, there exists a subsequence $\{(X,r_{\ell'}^{-1}\dist,r_{\ell'}^{-n}\mu, x)\}_{\ell'}$ which pmGH converges to a tangent cone $(X_x,\dist_x,\mu_x,x)$. Moreover, the points $\{y_{\ell'}\}_{\ell'}$ converge to some $y\in X_x$ such that $\dist_x(x,y)=1$. The lower semi-continuity of $\vartheta_X$ through pmGH convergence, together with the choice of $\{y_\ell\}_\ell$, ensures that
$$\vartheta_{X_x}(y)\le \liminf_{\ell'} \vartheta_X(y_{\ell'})=\vartheta_X(x)=\vartheta_{X_x}(x).$$ 
Thanks to Corollary \ref{cor:upthetaSplitting}, this implies that $\vartheta_{X_x}(y) = \vartheta_{X_x}(x)$ and $(X_x,\dist_x,\mu_x,x)$ is mm $1$-symmetric. Hence $d(x)\ge 1$ and $x\not\in S_0$.  \\

\textbf{Case II, $k\in \{1,\ldots,n-1\}$}.

From (1) in Lemma \ref{lemHaus},  we only need to prove that for any $s>0$, $\cH^s_\infty(S^k)>0$ implies $s \leq k$. Thus we assume 
\begin{equation}\label{eq:départ}
\cH^s_{\infty}(S^k) >0.
\end{equation}

\textbf{Step 1.} Let us write $S^k$ as a countable union of closed sets. From Remark \ref{rem-AR-ncStrongKatolim},  we know that there exists $C,\lambda>0$ such that for all $x \in X$ and all $0 < s < r \le R$,
 \begin{equation}\label{aprioriAhlfors}
ve^{-C \frac{\dist(x,o)}{R}} r^n\le  \mu(B_r(x))\le Cr^n \quad \text{and}\quad \frac{\mu(B_r(x))}{\mu(B_s(x))}\le C\left(\frac r s\right)^n\cdot
\end{equation}
 Arguing as in \cite[Proof of Theorem 10.20]{CheegerPisa}, we write $S^k$ as the countable union over $j \in \setN\backslash \{0\}$ of the closed sets
 \begin{equation*}
\begin{split}
S^{k,j}:= & \left\{x  \in X: \,\, \mathrm{D}(B_r(x), B^Z_r(z))\geq r/j, \, \, \,  \forall r \in (0,1/j),  \, \, \, \forall (Z, \dist_Z,\mu_Z,z) \in \mathrm{Adm}_{k+1} \right\},
\end{split}
\end{equation*}
where $\mathrm{Adm}_{k+1}$ is the set of mm $(k+1)$-symmetric spaces $(Z, \dist_Z,\mu_Z,z)$  satisfying \eqref{aprioriAhlfors} and $\mathrm{D}(B_r(x), B^Z_r(z))$ is the sum of $|\mu(B_r(x))-\mu_Z(B^Z_r(x))|$ and the $L^2$-transportation distance  \cite[p.~69]{sturm2006I} between the normalized metric measure spaces $(B_r(x),\dist,\mu(B_r(x))^{-1}\mu \measrestr B_r(x))$ and $(B_r^Z(x),\dist_Z,\mu_Z(B^Z_r(x))^{-1}\mu_Z \measrestr B^Z_r(x))$. Moreover, for any $j$
$$S^{k,j}=\bigcup_{N \in \setN}S^{k,j}\cap \overline{B_N(o)},$$
so (3) in Lemma \ref{lemHaus} ensures from \eqref{eq:départ} that there exists $j,N \in \N \backslash \{0\}$ such that
$$\cH^s_{\infty}(S^{k,j}\cap \overline{B_N(o)})>0.$$

\textbf{Step 2.} Let us write $S^{k,j}\cap \overline{B_N(o)}$ as a countable union of closed sets.  Take $\eps>0$. For any $x \in X$, since $\vartheta_X(x)<+\infty$,  there exists $\eta(x,\eps)>0$ such that for all $r \in (0, \eta(x,\eps)]$,
$$\left| \frac{\mu(B_r(x))}{\omega_n r^n} - \vartheta_X(x) \right| \leq \eps,$$
and we can define
$$\delta(x,\eps):=\sup\left\{r>0 : \left| \frac{\mu(B_\sigma(x))}{\omega_n \sigma^n}- \frac{\mu(B_\rho(x))}{\omega_n\rho^n} \right|\leq \eps,  \, \, \,  \forall \sigma, \rho\in (0,r] \right\} >0.$$
Then for all $c >0$ the set $A_{\eps,c} \subset X$ defined by
$$A_{\eps,c}:=\{x \in X : \delta(x,\eps)\geq c\}=\bigcap_{0 < \sigma\leq \rho \leq c}\left\{x \in X : \left| \frac{\mu(B_\sigma(x))}{\omega_n \sigma^n}- \frac{\mu(B_\rho(x))}{\omega_n \rho^n} \right|\leq \eps \right\}$$
is closed.  Hence for any $p,q \in \mathbb{Q}$ with $q<p<q+\eps$,  the set
\begin{equation*}
S_{\eps, p,q} := A_{\eps,2(p-q)} \cap \left\{ x \in S^{k,j}\cap \overline{B_N(o)} \, : \, q \leq \frac{ \mu(B_{p-q}(x))}{\omega_n (p-q)^n} \leq p \right\}
\end{equation*}
is compact since it is a closed subset of the compact set $\overline{B_N(o)}$. Observe that for any $x \in S_{\eps, p,q}$ and $\rho \in (0,2(p-q)]$ we have
\begin{equation}\label{uniformdensity} q-\eps \leq \frac{\mu(B_\rho(x))}{\omega_n \rho^n} \leq p+\eps,\end{equation}
and then $q-\eps \leq \vartheta_X(x) \leq p+\eps$ as $\rho \downarrow 0$. Finally, note that $$S^{k,j}\cap \overline{B_N(o)} = \bigcup_{\substack{p,q \in \mathbb{Q}\\q<p<q+\eps}}S_{\eps, p,q}\, .$$

\hfill

\textbf{Step 3.} Now let us consider the sequence $\{\eps_\ell:=2^{-\ell}\}_{\ell\in \N\backslash \{0\}}$. By (3) in Lemma \ref{lemHaus}, for any $\ell$ there exist $p_\ell, q_\ell \in \mathbb{Q}$ with $q_\ell < p_\ell$ such that $\cH^s_{\infty}(S_{\eps_{\ell},p_\ell,q_\ell})>0$, hence by (2) in Lemma \ref{lemHaus} there exist $x_\ell \in S_{\eps_{\ell},p_\ell,q_\ell}$ and $r_\ell>0$ small such that 
\begin{equation}
\label{eq:x_ell}
\frac{\cH^s_\infty\left(S_{\eps_{\ell},p_\ell,q_\ell}\cap B_{r_\ell}(x_\ell)\right)}{\omega_sr_\ell^s}\ge 4^{-s}.
\end{equation} 
As the pointed metric measure spaces $\{(X, r_{\ell}^{-1}\dist, r^{-n}_\ell \mu, x_\ell)\}_\ell$ all satisfy the volume estimates \eqref{aprioriAhlfors}, up to extracting a subsequence we can assume that they pmGH converge as $\ell \to +\infty$ to a pointed metric measure space $(Z,\dist_Z,\mu_Z,z)$.  Since the sets $\{S_{\eps_\ell, p_\ell, q_\ell}\}_\ell$ are compact,  up to extracting another subsequence we can assume that the compact sets $\{S_{\eps_{\ell}, p_{\ell}, q_{\ell}} \cap B_{r_{\ell}}(x_{\ell})\}_\ell$ GH converge to some compact set $K \subset B^Z_1(z)$ containing $z$.  Because of the upper semi-continuity of $\cH_\infty^s$ with respect to GH convergence (i.e.~(4) in Lemma \ref{lemHaus}) and because of \eqref{eq:x_ell}, we have $\cH^s_{\infty}(K) \geq \omega_s 4^{-s}.$ In particular,
\[
\dim_\cH K \ge s.
\]
Finally, up to extracting a further subsequence, we can assume that the bounded sequence of rational numbers $\{q_{\ell}\}$ tends to some number $Q >0$ as $\ell \to +\infty$. 

\hfill

\textbf{Step 4.} Now we take $y\in K$ and we let $y_{\ell}\in S_{\eps_{\ell}, p_{\ell}, q_{\ell}}$ for any $\ell$ be such that $y_\ell \to y$. Let $\ell$ be fixed. Take $r \in [0,2^\ell)$ and set $\rho :=r r_\ell$. With no loss of generality we can assume $ r_\ell \leq 2^{-\ell+1}(p_\ell-q_\ell)$, so that $\rho \in (0,2(p_\ell - q_\ell)]$. Then the triangle inequality and \eqref{uniformdensity} lead to
\[
\left| \frac{\mu(B_\rho(y_\ell))}{\omega_n \rho^n}-Q\right| \leq 2\eps_\ell+|q_{{\ell}}-Q|,
\]
which rewrites as
\begin{equation}\label{eq:50}
\left| \frac{\mu(B_{r r_\ell}(y_\ell))}{\omega_n r_\ell^n}-Qr^n\right| \leq r^n \left( 2\eps_\ell+|q_{{\ell}}-Q| \right).
\end{equation}
Since 
\[
\lim_{\ell \to +\infty}\frac{\mu(B_{r_{\ell}r} (y_{\ell}))}{r_{\ell}^n} = \mu_Z(B(y,r)),
\]
inequality \eqref{eq:50} yields $\mu_Z(B_r(y))=\omega_n Q r^n$ as $l \to +\infty$. 

Because of Remark \ref{rem-RCDmoreGeneral}, we know that $(Z,\dist_Z,\mu_Z)$ is a weakly non-collapsed $\RCD(0,n)$ metric measure space. In particular, its volume density $\vartheta_Z$ is well defined at all points, and thanks to Proposition \ref{prop:DPG} the previous computation shows that for any $y \in K$, $(Z,\dist_Z,\mu_Z)$ is a metric measure cone at $y$ and  $\vartheta_Z(y)=Q$ . By Proposition \ref{prop:splitting}, this means that there exists an integer $k' \ge \dim_\cH K$ such that $(Z, \dist_Z, \mu_Z,z)$ is metric measure $k'$-symmetric.  In particular,
\[
k' \ge s.
\]

\textbf{Step 5.} To conclude, let us show that $k \ge k'$. Since $(X,\dist,\mu,o)$ satisfies the volume estimates \eqref{aprioriAhlfors},  so do the rescalings $\{(X, r_{\ell}^{-1}\dist, r_{\ell}^{-n}\mu, x_\ell)\}_{\ell \in \N \backslash \{0\}}$. As $(Z, \dist_Z, \mu_Z,z)$ is the pmGH limit of these rescalings, this implies that $(Z, \dist_Z, \mu_Z,z)$ belongs to $\mathrm{Adm}_{k'}$.  Since for $\ell$ large enough we have $r_{\ell} < 1/j$ and
$$ \mathrm{D}(B_{r_\ell}(x_{\ell}), B^Z_{\ell}(z))< \frac{r_{\ell}}{j} \, ,$$
this means that $x_\ell$ does not belong to $S^{k',j}$. But $x_\ell \in S^{k,j}$. As a consequence, by the very definition of $S^{k,j}$, the integer $k'$ is necessarily smaller than $k+1$. 

\hfill

\textbf{Case III, $k=n$}. 
Lebesgue differentiation theorem holds on locally doubling spaces (see \cite[Sect.~3.4]{HKST}) so $\mu$-a.e.~$x \in X$ is a Lebesgue point of the locally integrable function $\vartheta_X$.  Thus it is enough to show that whenever $x \in X$ is a Lebesgue point of $\vartheta_X$, that is
\[
\lim\limits_{r \to 0} \fint_{B_r(x)} \vartheta_X \di \mu = \vartheta_X(x),
\]
then any tangent cone at $x$ is equal to $(\setR^n,\dist_e,\vartheta_X(x)\cH^n,0_n)$. 

Let $x$ be a Lebesgue point of $\vartheta_X$, $(X_x,\dist_x,\mu_x,x)$ be a tangent cone and $\{r_\alpha\}_\alpha \subset (0,+\infty)$ be such that $r_\alpha \downarrow 0$ and $(X,\dist_\alpha:=r_\alpha^{-1}\dist,\mu_\alpha:=r_\alpha^{-n} \mu,x) \to (X_x,\dist_x,\mu_x,x)$ in the pmGH sense.  According to Corollary \ref{cor:monoHeat}, we know that if we set
\[
\beta_X(z,t):=\frac{1}{\exp(\Phi(t))(4\pi t)^{n/2}H(t,z,z)}
\]
for any $z \in X$ and any $t$ small enough, then
\[
\vartheta_X(z)=\lim_{t \to 0} \beta_X(z,t)
\]
and $t \mapsto  \beta_X(z,t)$ is non increasing. The same is true if we define $\beta_{X_\alpha}$ (resp.~$\beta_{X_x}$) in a similar way on the rescaled space $(X,\dist_\alpha,\mu_\alpha,x)$ (resp.~on the tangent cone $(X_x,\dist_x,\mu_x,x)$)  for any $\alpha$, and we have $\beta_{X_\alpha}(\cdot,t) \to \beta_{X_x}(\cdot,t)$ uniformly on compact sets for any $t$ small enough; this implies
\begin{align*}
\fint_{B_1^{\dist_x}(x)} \beta_{X_x}(z,t) \di \mu_x(z) & = \lim_\alpha \fint_{B_1^{\dist_\alpha}(x)} \beta_{X_\alpha}(z,t) \di \mu_\alpha(z) \\
& = \lim_\alpha \fint_{B_{r_\alpha}^{\dist}(x)} \beta_{X}(z,r_\alpha^2 t) \di \mu_(z)\\
& \le \lim_\alpha \fint_{B_{r_\alpha}^{\dist}(x)} \vartheta_X(z) \di \mu_(z) = \vartheta_X(x).
\end{align*}
By monotone convergence, letting $t \downarrow 0$ gives 
\[
\fint_{B_1^{\dist_x}(x)} \vartheta_{X_x}(z) \di \mu_x(z) \le \vartheta_X(x).
\]

By the first statement in Proposition \ref{prop:splitting},
\[
\fint_{B_1^{\dist_x}(x)} \vartheta_{X_x}(z) \di \mu_x(z) \ge  \vartheta_{X_x}(x) = \vartheta_X(x),
\]
hence $\vartheta_{X_x}$ is constantly equal to $\vartheta_X(x)$ on $B^{\dist_x}_1(x)$. The second statement of Proposition \ref{prop:splitting} implies that $X_x$ is isometric to $\setR^n$ equipped with the Euclidean distance and $\mu_{X_x}$ is given by $c \cH^n$ for some $c \in (0,1]$. But since for all $r>0$ we have $\mu_{X_x}(B_r(x))=\vartheta_X(x) \omega_n r^n$, with $\cH^n(B_r(x))=\omega_n r^n$ in $\R^n$, we get $c=\vartheta_X(x)$.
\end{proof}

\section{Volume continuity}

This section is devoted to proving the following analog of volume continuity for Ricci limit spaces \cite[Theorem 0.1]{Col97}, \cite[Theorem 5.9]{CheegerColdingI}.

\begin{theorem}
\label{thm:muHaus}
Let $(X,\dist, \mu,o, \cE)$ be a non-collapsed strong Kato limit. Then $\mu$ coincides with the $n$-dimensional Hausdorff measure $\cH^n$. 
\end{theorem}

The proof of the previous is a direct consequence of the next key result, of \cite[Theorem 6.9]{Mattila} and of the fact that we already know $\mu \leq \cH^n$.

\begin{theorem}
\label{thm:density1}
Let $(X,\dist, \mu,o, \cE)$ be a non-collapsed strong Kato limit and $x \in X$ such that the set of tangent cones at $x$ is reduced to $\{(\R^n, \dist_e,\vartheta_X(x)\cH^n, x)\}$. Then $\vartheta_X(x)=1$. 
\end{theorem}

As a consequence, we also obtain the following corollary, which generalizes \cite[Theorem 9.31]{CheegerPisa} for manifolds with Ricci curvature bounded below. 

\begin{cor}
\label{cor:Hausdorffvsmu3}
Let $n\geq 1$, $T,v, \Lambda>0$ and $f: (0,T] \to \R$ a function such that
$$\int_0^T \frac{\sqrt{f(s)}}{s}\di s \leq \Lambda.$$
Then for all $\eps >0$ there exists $\delta =\delta(\eps,n,\Lambda, v,T, f)$ such that the following holds. Assume that for all $t \in (0,T]$
$$ \Kato_t(M^n,g) \leq f(t), \quad \frac{\nu_{g}(B_{\sqrt{T}}(x))}{T^{\frac n2}}\geq v,$$
and for $r \in (0, \delta \sqrt{T}]$
$$\mbox{d}_{\mbox{\tiny{GH}}}(B_r(x), \mathbb{B}(r)) \leq \delta r.$$
Then
$$1-\eps \leq \frac{\nu_g(B_r(x))}{\omega_n r^n} \leq 1+\eps.$$
\end{cor}

\begin{proof}
Assume by contradiction that there exist $\eps_0$ such that for all $\delta$ the conclusion of the corollary is false. Than we can consider a sequence $\delta_i$ tending to zero and manifolds $(M_i^n,g_i)$ satisfying the assumptions above, for which there exist $r_i \in (0, \delta_i \sqrt{T}]$ and $x_i \in M_i$ such that
\begin{equation}
\label{eq-GHclose}
\mbox{d}_{\mbox{\tiny{GH}}}(B(x_i,r_i), \mathbb{B}(r_i)) \leq \delta_i r_i,
\end{equation}
and 
\begin{equation}
\label{eq-contr}
\left| \frac{\nu_{g_i}(B(x_i,r_i))}{\omega_n r_i^n} - 1\right|\geq \eps_0.
\end{equation}
The re-scaled sequence $(M_i^n,r_i^{-2}g_i, r_i^{-n}v_{g_i},x_i)$ is a non-collapsing sequence satisfying the strong Kato bound \eqref{StrongKato}, with $\Kato_t(M_i,\eps_i^{-1})$ tending to zero. As a consequence, up to a sub-sequence, it converges to a pointed  metric measure space $(X,\dist, \mu,x)$. Because of \eqref{eq-GHclose} and \eqref{eq-contr} the unit ball $B_1(x)$ is isometric to the unit Euclidean ball $\mathbb{B}(1)$ and satisfies
\begin{equation}
\label{eq-contr2}
\left|\frac{\mu(B_1(x))}{\omega_n} -1 \right| \geq \eps_0.
\end{equation} 
But according to Theorem \ref{thm:muHaus}, $\mu=\cH^n$ hence  in particular $\mu$ coincides with the Lebesgue measure on $B_1(x)$, this contradicting \eqref{eq-contr}. 
\end{proof}

In the remainder of this section, we prove Theorem \ref{thm:density1}. In order to do this, we start by proving the existence of GH-isometries with the appropriate regularity properties. 

\subsection{Existence of splitting maps} One of the most powerful tools in the study of Ricci limit spaces and $\RCD$ spaces is given by $\eps$-splitting maps, see for example Definition 4.10 in \cite{CJN}. We are going to show that whenever a point $x$ in a non-collapsed strong Kato limit admits a Eudlidean tangent cone, we can construct an $\eps$-splitting map from a ball around $x$ to a Euclidean ball.  To this aim, we need an approximation result for harmonic functions defined on $\PI$ Mosco-Gromov-Hausdorff limits, which is proven in the Appendix, together with the gradient and Hessian estimates shown in Section 4. 

In the following, we denote a Euclidean ball of radius $r$ centered at $0^n$ as $\mathbb{B}(r)$. 

\begin{theorem}
\label{thm-splitting-maps}
Let $(X,\dist, \mu)$ be a strong non-collapsed Kato limit obtained from a sequence $\{(M_\alpha, g_\alpha)\}_\alpha$ and $x \in X$ a point that admits $(\R^n,\dist_e, \vartheta_X(x)\cH^n,0^n)$ as a tangent cone. Then there exist sequences $\{r_\alpha\}, \{\eps_\alpha\} \subset (0,\infty)$ tending to zero, $x_\alpha$ in $M_\alpha$ and maps $H_\alpha=(h_{1,\alpha}, \ldots, h_{n,\alpha})$, $H_\alpha : B_{r_\alpha}(x_\alpha) \rightarrow \mathbb{B}(r_\alpha)$, such that $h_{i,\alpha}$ is harmonic on $B_{r_\alpha}(x)$ for all $i =1, \ldots, n$. Moreover, the following holds.
\begin{enumerate}
\item[\emph{(i)}] $H_\alpha$ is an $(\eps_\alpha r_\alpha)$-GH isometry. 
\item[\emph{(ii)}]$H_\alpha$ is $(1+\eps_\alpha)$-Lipschitz. 
\item[\emph{(iii)}] $\displaystyle \fint_{B_{r_\alpha}(x)}|{}^t dH_\alpha \circ dH_\alpha - \mbox{Id}_n| \di \nu_{g_\alpha}\leq \eps_\alpha.$
\item[\emph{(iv)}] $\displaystyle r_\alpha^2\fint_{B_{r_\alpha}(x)}|\nabla dH_\alpha|^2 \di \nu_{g_\alpha}\leq \eps_\alpha.$ 
\item[\emph{(v)}]  $\displaystyle \lim_{\alpha}\frac{\nu_{g_\alpha}(B_{ t r_\alpha}(x))}{\omega_n(tr_\alpha)^n}=\vartheta_X(x)$ for all $t >0$.
\end{enumerate}
\end{theorem}

Before proving the previous theorem, we show an improvement of the Lipschitz constant of Lipschitz harmonic functions whose gradient is suitably close to 1. The argument we use is originally due to J.~Cheeger and A.~Naber, see \cite[Lem.~3.34]{CheegerNaber15}, and it relies on the existence of good cut-off functions, Bochner formula and the appropriate estimates for the heat kernel.  

\begin{prop}\label{Lipoptimal}
Let $(M^n,g)$ be a closed Riemannian manifold and $u:B_r(x) \to \setR$ a $\kappa$-Lipschitz harmonic function, for some $\kappa\ge1$. Assume that there exists $\delta>0$ such that
\[
\Kato_{r^2}(M^n,g) \le \delta \le \frac{1}{16n} \quad \text{and} \quad \fint_{B_r(x)}\left| |du|^2-1\right|\di \nu_g\le \delta^2.
\]
Then $u \restr_{B_{r/2}(x)}$ is $1+C(n,\kappa)\delta$-Lipschitz.
\end{prop}

\proof Let $\chi\in C_c^\infty(M)$ be a cut-off function as constructed in Proposition \ref{prop:coupure} such that:
\begin{enumerate}[i)]
\item $\chi=1$ on $B_{3r/4}(x)$,
\item $\chi=0$ on $M\setminus B_r(x)$,
\item $|d \chi| \le C(n)/r$ and $|\Delta \chi| \le C(n)/r^2$ on $B_r(x) \backslash B_{3r/4}$.
\end{enumerate}
Apply Bochner's formula on $B_r(x)$ to the $\kappa$-Lipschitz harmonic function $u$ in order to get
\[
|\nabla du|^2+\frac12 \Delta \left(|du|^2-1\right) = -\Ric(\nabla u, \nabla u).
\]
Since $|\nabla du|^2\ge 0$ and $-\Ric(\nabla u, \nabla u) \le \Ricm \kappa^2$, this leads to
\[
\frac12 \Delta \left(|du|^2-1\right)\le \Ricm \kappa^2.
\]
Take $y \in B_{r/2}(x)$ and multiply the previous inequality evaluated at some $z \in B_r(x)$ by $2H(t,y,z)\chi(z)$, where $t \in [0,r^2]$, then integrate with respect to $z$ and $t$:
\begin{align*}
& \phantom{=} \phantom{=} \iint_{[0,r^2]\times B_r(x)} H(t,y,z)\chi(z) \Delta \left(|du|^2(z)-1\right)\di \nu_g(z)\di t \nonumber\\
 & \le 2 \iint_{[0,r^2]\times B_r(x)} H(t,y,z)\chi(z) \Ricm \kappa^2 \di \nu_g(z)\di t \nonumber.
\end{align*}
As immediately seen, the previous right-hand side is not greater than  $2 \Kato_{r^2}(M^n,g)\kappa^2$ which is not greater than $2 \delta \kappa^2$. Thus
\begin{align}\label{estimee1}
\iint_{[0,r^2]\times B_r(x)} H(t,y,z)\chi(z) \Delta \left(|du|^2(z)-1\right)\di \nu_g(z)\di t \le 2 \delta \kappa^2.
\end{align}
Use integration by parts to rewrite the left-hand side, with simplified notations, as follows:
\begin{align*}
\int H \, \chi\,  \Delta \left(|du|^2-1\right) & = \int \Delta( H \, \chi )\,  \left(|du|^2-1\right)\\
& = \int (\Delta H) \, \chi \,  \left(|du|^2-1\right) - 2 \int \scal{\nabla H}{\nabla \chi} \left(|du|^2-1\right)\\
& + \int H \, (\Delta\chi) \,  \left(|du|^2-1\right).
\end{align*}
Now 
\begin{align*}
& \phantom{=} \phantom{=} \iint_{[0,r^2]\times B_r(x)} \Delta_z\left(H(t,y,z)\right)\chi(z)\left(|du|^2(z)-1\right)\di \nu_g(z)\di t\\
& =-\int_{B_r(x)} \left( \int_0^{r^2}\frac{\partial H(t,y,z)}{\partial t} \di t \right) \chi(z)\left(|du|^2(z)-1\right) \di \nu_g(z)\\
&=-\int_{B_r(x)}H(r^2,y,z)\chi(z)\left(|du|^2(z)-1\right)\di \nu_g(z) + \underbrace{\chi(y)}_{=1}\left(|du|^2(y)-1\right) .
\end{align*}
Combining these three last estimates with the properties of the cut-off function $\chi$, we get
$$|du|^2(y)-1\le 2\kappa^2\delta+I(y)+II(y)+III(y)$$
where
$$I(y)=\int_{ M}H(r^2,y,z)\chi(z)\left(|du|^2(z)-1\right)\di \nu_g(z),$$
$$II(y)=\frac{C(n)}{r}\iint_{[0,r^2]\times [B_r(x)\setminus B_{3r/4}(x)]} \left|\nabla_z H(t,y,z)\right|\,\left||du|^2(z)-1\right|\di \nu_g(z) \di t,$$
$$III(y)=\frac{C(n)}{r^2}\iint_{[0,r^2]\times [B_r(x)\setminus B_{3r/4}(x)]}H(t,y,z) \left||du|^2(z)-1\right|\di \nu_g(z)\di t,$$
and we are going to establish the following estimates:
\begin{equation*}
I(y) \le C(n) \delta^2, \qquad II(y) \le C(n) \delta \sqrt{1+\kappa^2}, \qquad III(y) \le C(n) \delta^2,
\end{equation*}
which are enough to complete the proof.

We recall the upper bound for the heat kernel, with $\nu=e^2n$ and $t\le r^2, y,z\in M$:
$$H(t,y,z)\le\frac{C(n)}{\nu_g(B_r(y))}\frac{r^\nu}{t^{\frac \nu 2}} e^{-\frac{d^2(y,z)}{5t}}.$$
Moreover, the doubling condition implies for $y \in B_{r/2}(x)$
\begin{equation}
\label{eq-boundHK}
H(t,y,z)\le\frac{C(n)}{\nu_g(B_r(x))}\frac{r^\nu}{t^{\frac \nu 2}} e^{-\frac{d^2(y,z)}{5t}}.
\end{equation}
Therefore, for any $z \in B_r(x) \setminus B_{3r/4}(x)$ and $y \in B_{r/2}(x)$ we obtain
$$H(r^2,y,z)\le \frac{C(n)}{\nu_g(B_r(x))}.$$ 
Using this inequality and the assumption on $|du|$ leads to the estimate for $I(y)$:
$$I(y) \leq  \frac{C(n)}{\nu_g(B_r(x))} \int_{B_r(x)}||du|^2(z)-1|\di \nu_g(z) \leq C(n)\delta^2.$$
We now obtain the estimate for $III(y)$. Consider $z \in B_r(x) \setminus B_{3r/4}(x)$ and $y \in B_{r/2}(x)$ as above. Inequality \eqref{eq-boundHK} and the fact that 
$$\int_0^{r^2} \frac{r^\nu}{t^{\frac \nu 2}} e^{-\frac{r^2}{80t}}\di t= r^2 \int_0^{1} \frac{1}{t^{\frac \nu 2}} e^{-\frac{1}{80t}}\di t$$  imply that
$$\int_0^{r^2}H(t,y,z)\di t\le \frac{C(n)}{\nu_g(B_r(x))}\int_0^{r^2} \frac{r^\nu}{t^{\frac \nu 2}} e^{-\frac{r^2}{80t}}\di t= \frac{C(n)r^2}{\nu_g(B_r(x))},$$
and as a consequence
$$III(y)\le C(n)\fint_{B_r(x)}\left||du|^2(z)-1\right|\di \nu_g(z)\le C(n)\delta^2.$$
As for $II(y)$ we use the Cauchy-Schwarz inequality twice, first in $\di \nu_g$ and then $dt$, together with the result of Lemma \ref{lem:estintnableH}:
$$\int_M  \frac{\left|\nabla_z H(t,y,z)\right|^2}{H(t,y,z)}\di \nu_g(z)\le \frac{C(n)}{t},$$
thus we obtain
\begin{align*}
II(y) &= \frac{C(n)}{r}\int_0^{r^2}\int_{B_r(x) \setminus B_{3r/4}(x)}\frac{|\nabla_z H(t,y,z)|}{\sqrt{H(t,y,z)}} \sqrt{H(t,y,z)}  \left||du|^2(z)-1\right|\di \nu_g(z) \di t \\
& \leq \frac{C(n)}{r}\int_0^{r^2}\frac{1}{\sqrt{t}} \left( \int_{B_r(x)\setminus B_{3r/4}(x)}  H(t,y,z)\left||du|^2(z)-1\right|^2\di \nu_g(z)\right)^{\frac 12} \di t \\
& \leq C(n)\left(\int_{[0,r^2]\times (\left(B_r(x)\setminus B_{3r/4}(x)\right)} \frac{H(t,y,z)}{t} \left||du|^2(z)-1\right|^2 \di \nu_g(z) \di t \right)^{\frac{1}{2}}
\end{align*}
Using  
$$\int_0^{r^2} \frac{r^{\frac{\nu}{2}}}{t^{1+\frac{\nu}{2}}}e^{-\frac{r^2}{80t}}\di t = \int_0^{1} \frac{1}{t^{1+\frac \nu 2}} e^{-\frac{1}{80t}}\di t,$$ we have
$$\int_0^{r^2}  \frac{H(t,z,y)}{t} \di t \leq \int_0^{r^2}\frac{C(n)}{\nu_g(B_r(x))}\frac{r^{\frac{\nu}{2}}}{t^{1+\frac{\nu}{2}}}e^{-\frac{r^2}{80t}}\di t \leq \frac{C(n)}{\nu_g(B_r(x))}.$$
We then obtain
$$II(y) \leq C(n)(1+\kappa^2)^{\frac 12}\left(\fint_B \left||du|^2(z)-1\right| \di \nu_g(z) \right)^{\frac 12 }\leq C(n)(1+\kappa^2)^{\frac 12} \delta.$$
This allows us to obtain the desired bound on $|dh|$ over $B_{r/2}(x)$. 
\endproof 

We are now in a position to prove Theorem \ref{thm-splitting-maps}. 

\begin{proof}[Proof of Theorem \ref{thm-splitting-maps}]

Let $x \in X$ and assume that $(\R^n, \dist_e, \vartheta_X(x)\cH^n, 0^n)$ is a tangent cone at $x$. Then by definition of tangent cones of strong Kato limit spaces, there exist sequences $(r_\alpha)_\alpha \subset (0,+\infty)$, $r_\alpha \downarrow 0$ and $x_\alpha \in M_\alpha$  such that
$$(M_\alpha, r_\alpha^{-1}\dist_{g_\alpha}, r_\alpha^{-n}\di \nu_{g_\alpha}, x_\alpha) \longrightarrow  (\R^n, \dist_e, \vartheta_X(x)\cH^n, 0^n),$$
and property (v) holds. Denote by $\tilde g_\alpha = r_\alpha^{-2} g_\alpha$ and $\mu_\alpha=r_\alpha^{-n}\di \nu_{g_\alpha}$. Balls with respect to $\tilde g_\alpha$ are denoted by $\tilde B_s(y)$. Notice that is is enough to prove the existence of a map $H_\alpha : \tilde{B}_1(x_\alpha) \to \mathbb{B}^n(1)$ satisfying properties (i) to (iv) with respect to the re-scaled metric $\tilde g_\alpha$ and with $r_\alpha$ replaced by $1$. Then the map on $B_{r_\alpha}(x_\alpha)$ is simply obtained by re-scaling $H_\alpha$ by a factor $r_\alpha$. As a consequence, in the rest of the proof, we only work with the re-scaled manifolds $(M_\alpha, \tilde{g}_\alpha)$. 

Consider the coordinate maps $x_i: \R^n \to \R$ for all $i=1, \ldots, n$. Then $x_i$ are harmonic and we can apply Proposition \ref{prop:approxharmonic}: for all $\alpha$ there exist harmonic functions $h_{i,\alpha} : \tilde B_1(x_\alpha) \to \mathbb{B}(1)$ such that 
\begin{itemize}
\item[(i)] $h_{i,\alpha} \to x_i|_{\mathbb{B}(1)}$ uniformly; 
\item[(ii)] For all $s \leq 1$
$$\lim_{\alpha \to \infty} \int_{\tilde B_s(x_\alpha)}|d h_{i,\alpha}|_{\tilde g_\alpha}^2 \di \mu_\alpha = \int_{\mathbb{B}(s)} |dx_i|^2 \vartheta_X(x) \di \cH^n = \vartheta_X(x) \omega_n.$$
\end{itemize}
Define $$H_\alpha=(h_{1,\alpha}, \ldots, h_{n, \alpha}): \tilde{B}_1(x_\alpha) \to \mathbb{B}(1).$$
Since $H_\alpha$ converges uniformly to the identity $\mbox{Id}_n =(x_1, \ldots, x_n)$, it is not difficult to show that $H_\alpha$ is an $\eps_\alpha$-GH isometry where  $(\eps_\alpha)_\alpha \subset (0,+\infty)$ is a sequence tending to zero. In the continuation of the reasoning, we will take the freedom  of modifying this sequence tending to zero while keeping its notation. 

Since $\mu_\alpha(\tilde{B}_1(x_\alpha))$ tends to $\vartheta_X(x)\omega_n$, the second property implies that 
$$\lim_{\alpha \to \infty} \fint_{\tilde{B}_1(x_\alpha)} |d h_{i,\alpha}|_{\tilde g_\alpha}^2 \di \mu_\alpha  = 1.$$
Using the first estimate in Lemma \ref{lem:gradientEst_harmonic} we then deduce that there exists $C(n)>0$ such that
$$\sup_{\tilde{B}_{3/4}(x_\alpha)} |d h_{i,\alpha}|_{\tilde g_\alpha} \leq C(n),$$
that is $h_{i,\alpha}$ is $C(n)$-Lipschitz on $\tilde B_{1/2}(x_\alpha)$.
We can then apply Proposition \ref{prop:HessEst} and get some uniform estimates 
$$\fint_{\tilde{B}_{1/2}(x_\alpha)}|\nabla d h_{i,\alpha}^2|_{\tilde{g}_\alpha}^2  \di \mu_\alpha \leq C_n $$
Then, Remark \ref{rem:localisationcv} in the Appendix ensures that $|dh_{i,\alpha}|_{\tilde{g}_\alpha}$ tends to 1 in $L^2$. Then by additionally using the fact that $(M_\alpha, \tilde{g}_\alpha)$ is doubling, we have

$$\lim_{\alpha \to \infty} \fint_{\tilde{B}_{1/2}(x_\alpha)}||dh_{i,\alpha}|_{\tilde{g}_\alpha} -1|^2 \di \mu_\alpha = 0.$$

This, together with $h_{i,\alpha}$ being Lipschitz, implies 

$$\lim_{\alpha \to \infty} \fint _{\tilde{B}_{1/2}(x_\alpha)}||dh_{i,\alpha}|^2_{\tilde{g}_\alpha} -1| \di \mu_\alpha =0.$$

Recall that, as observed in Remark \ref{rem:tgConeKato}, a sequence of re-scaled manifolds converging to a tangent cone of a strong Kato limit is such that for all $t>0$ 
$$\lim_{\alpha \to \infty}\Kato_t(M_\alpha, \tilde{g}_\alpha)=0.$$
Modifying the sequence $(\eps_\alpha)_\alpha$ if necessary, we have 
$$\Kato_{1/2}(M_\alpha, \tilde{g}_\alpha) < \eps_\alpha, \quad \fint _{\tilde{B}_{1/2}(x_\alpha)}||dh_{i,\alpha}|_{\tilde{g}_\alpha}^2 -1| \di \mu_\alpha  < \eps_\alpha.$$

This means that the assumptions of Proposition \ref{Lipoptimal} are satisfied, therefore for $\alpha$ large enough $h_{i,\alpha}$ is $(1+\eps_\alpha)$-Lipschitz on the ball $\tilde{B}_{1/4}(x_\alpha)$. 

Applying Proposition \ref{prop:HessEst} and using again the doubling property, we also obtain 
$$\fint_{\tilde{B}_{1/4}(x_\alpha)}|\nabla d h_{i,\alpha}^2|_{\tilde{g}_\alpha}^2  \di \mu_\alpha \leq C_n \fint_{\tilde{B}_{1/2}(x_\alpha)}\left||d h_{i,\alpha}^2|_{\tilde{g}_\alpha}^2 -\fint_{\tilde{B}_1(x_\alpha)}|d h_{i,\alpha}^2|_{\tilde{g}_\alpha}^2 \di \mu_\alpha\right| \di \mu_\alpha < \eps_\alpha.$$
Here we used that $\Kato_{r_\alpha^{-2}T}(M_\alpha, \tilde{g}_\alpha) \leq 1/16n$, thus for $\alpha$ large enough we have $\min\{1/2, r_\alpha^{-2}T\}=1/2$. As a consequence, up to replacing $r_\alpha$ by $r_\alpha/4$, $h_{i, \alpha}$ is $(1+\eps_\alpha)$-Lipschitz on $\tilde{B}_1(x_\alpha)$ and satisfies 
$$ \fint_{\tilde{B}_1(x_\alpha)} ||dh_{i,\alpha}|^2_{\tilde{g}_\alpha} -1| \di \mu_\alpha < \eps_\alpha, \quad \fint_{\tilde{B}_1(x_\alpha)} |\nabla dh_{i,\alpha}|^2_{\tilde{g}_\alpha} \di \mu_\alpha < \eps_\alpha.$$
In order to obtain properties (iii) and (iv) for $H_\alpha$, one can consider the function $x_i+x_j$. Since we know that $h_{i,\alpha}+h_{j,\alpha}$ converges uniformly to $x_i+x_j$, by arguing as above we get 
$$\lim_{\alpha \to \infty}\fint_{\tilde{B}_1(x_\alpha}|\langle dh_{i,\alpha}, dh_{j,\alpha}\rangle_{\tilde{g}_\alpha} -\delta_{ij}|^2\di \mu_\alpha =0.$$
Then the same argument that we used for $h_{i,\alpha}$  finally leads to properties (iii) and (iv) for $H_\alpha$. 
\end{proof}

\begin{rem}
The same argument as above shows that if $(X,\dist,\mu)$ is a strong non-collapsed Kato limit and $x \in X$ admits an mm $k$-symmetric tangent cone, then there exist harmonic $\eps$-splitting maps from a ball around $x$ to a Euclidean ball of the same radius in $\R^k$.
\end{rem}

\subsection{Proof of Theorem \ref{thm:density1}}

Our proof of Theorem \ref{thm:density1} is inspired by the argument illustrated in \cite[Theorem 9.31]{CheegerPisa} and \cite[Theorem 1.6]{Gallot}. Both proofs are based on degree theory. 

\begin{proof}

Let $x$ be a point in $X$ that admits an Euclidean tangent cone $(\R^n, \dist_e, \vartheta_X(x) \cH^n,0^n)$. Let $r_\alpha, \eps_\alpha$ and $H_\alpha : B_{r_\alpha}(x) \to \R^n$ be as in Theorem \ref{thm-splitting-maps}. Let $\rho_\alpha : B_{r_\alpha}(x) \to \R$ be defined by $\rho_\alpha(x)=||H_\alpha(x)||^2$. Fix $\tau_\alpha \in \left( \frac 14 r_\alpha -2\eps_\alpha r_\alpha, \frac 14 r_\alpha -\eps_\alpha r_\alpha \right)$. Since $\rho_\alpha$ is smooth, by Sard's theorem $\tau_\alpha$ can be chosen so that $\tau_\alpha^2$ is a regular value of $\rho_\alpha$. Now define the compact set
$$\Omega_\alpha= \{ x \in B_{r_\alpha}(x), \ ||H_\alpha(x)|| \leq \tau_\alpha\}.$$
Since $H_\alpha$ is an $(\eps_\alpha r_\alpha)$-GH isometry, we have 
$$H_\alpha(B_{\frac{r_\alpha}{4} -3r_\alpha \eps_\alpha}(x)) \subset \mathbb{B}\left(\frac{r_\alpha}{4} -2\eps_\alpha r_\alpha\right).$$
Also, for any $x$ such that $||H_\alpha(x)|| \leq r_\alpha/4 -\eps_\alpha r_\alpha$, $x \in B_{ r_\alpha/4}(x)$. Then with our choice of interval and $\tau_\alpha$ we have
\begin{equation}
\label{eq-inclusions}
B_{ \frac{r_\alpha}{4} -3 \eps_\alpha r_\alpha}(x) \subset \Omega_\alpha \subset B_{\frac{r_\alpha}{4}}(x).
\end{equation}
We claim that if for $\alpha$ large enough, $H_\alpha : \Omega_\alpha \rightarrow \mathbb{B}(\tau_\alpha)$ is surjective, then $\vartheta_X(x) =1$. Indeed, if $H_\alpha$ is surjective, the estimates on $dH_\alpha$ and on the Lipschitz constant of $H_\alpha$ imply
$$\cH^n(\mathbb{B}(\tau_\alpha)) \leq (1+\eps_\alpha)^n \nu_{g_\alpha}(\Omega_\alpha),$$
which, together with the inclusion above leads to
$$\omega_n \left( (1-12\eps_\alpha)\frac{r_\alpha}{4} \right)^n \leq (1+\eps_\alpha)^n \nu_{g_\alpha}\left(B_{\frac{r_\alpha}{4}}\left(x\right)\right).$$
Together with (6) in the previous Theorem, this shows that $\vartheta_X(x) \geq 1$. But since $\vartheta_X$ is lower semi-continuous, we already know that $\vartheta_X(x)\leq 1$, then $\vartheta_X(x)=1$. 

In the rest of the proof we show by contradiction that $H_\alpha : \Omega_\alpha \rightarrow \mathbb{B}(\tau_\alpha)$ is surjective for $\alpha$ large enough. We first assume that $M_\alpha$ is oriented and $H_\alpha$ is not surjective. We let $\Theta_\alpha$ be the unit volume form of  $(M_\alpha,g_\alpha)$ We  Then, since $\Omega_\alpha$ is compact the set $\mathbb{B}(\tau_\alpha) \setminus H_\alpha(\Omega_\alpha)$ is open and there exists an open ball $\mathbb{B}(p,\eta) \subset \mathbb{B}(\tau_\alpha) \setminus H_\alpha(\Omega_\alpha)$. We claim that there exists an $(n-1)$-form $\gamma_\alpha$ on the set $\mathbb{B}(\tau_\alpha) \setminus \mathbb{B}(p,\eta)$  such that $\di \gamma_\alpha$ equals the $n$-volume form $\omega_\alpha = \di x_1 \wedge \ldots \wedge \di x_n$ and $\iota^* \gamma_\alpha =0$, where $\iota : \partial \mathbb{B}(p,\eta) \to \mathbb{B}(p,\eta)$ is the inclusion map. 

Now observe that $H_\alpha(\partial \Omega_\alpha) = \partial \mathbb{B}_{\tau_\alpha}$: if $x \in \partial \Omega_\alpha$, then $||H_\alpha(x)||=\tau_\alpha$ and $H_\alpha(x) \in \partial \mathbb{B}(\tau_\alpha)$  because $\tau_\alpha$ is a regular value of $\rho_\alpha= ||H_\alpha||^2$. Then Stokes' theorem implies 
$$\int_{\Omega_\alpha} H_\alpha^*(\omega_\alpha) = \int_{\Omega_\alpha} \di H_\alpha^*(\gamma_\alpha) = \int_{\partial \Omega_\alpha} H_\alpha^* \gamma_\alpha =0.$$
But we also know 
$$H_\alpha^*(\omega_\alpha) =H^*_\alpha(\di x_1 \wedge \ldots \wedge \di x_n) = \di h_{1, \alpha}\wedge \ldots \wedge \di h_{n, \alpha} = f_\alpha\Theta_\alpha,$$
where for $x \in \Omega_\alpha$, $$f_\alpha(x)=\mbox{det}(\di_x H_\alpha(e_1),\dots,\di_x H_\alpha(e_n) ), $$ where $(e_1,\dots,e_n)$ is a direct orthonormal basis of $(T_xM_\alpha,g_\alpha(x))$. Then 
$$\int_{\Omega_\alpha} f_\alpha \di \nu_{g_\alpha} = 0.$$
Denote by $B_\alpha$ the ball $B\left(x,\frac{r_\alpha}{2}\right)$. Let $x \in  B_\alpha$ and $\lambda_1, \ldots \lambda_n$ the eigenvalues of ${}^t\di_x H_\alpha \circ \di_x H_\alpha$. Then we have
$$|f_\alpha(x)|=\sqrt{\lambda_1 \cdot \ldots \cdot \lambda_n},$$
and 
$$|\lambda_i - 1| \leq |{}^t\di_x H_\alpha \circ \di_x H_\alpha - \mbox{Id}_n|:= D_\alpha.$$
Observe that property (iii) of $H_\alpha$ implies that for $\alpha$ large enough $D_\alpha < 1$. Therefore we have
$$|f_\alpha | \geq \left(1- D_\alpha\right)^{\frac n2} \geq 1-\frac{n}{2} D_\alpha.$$
As a consequence and using the doubling condition
$$\fint_{B_\alpha}|f_\alpha| \di \nu_{g_\alpha} \geq 1 - \frac{n}{2} \fint_{B\left(x,\frac{r_\alpha}{2}\right)}D_\alpha \di \nu_{g_\alpha} \geq 1 - C\fint_{B_{r_\alpha}(x)} D_\alpha \di \nu_{g_\alpha}.$$
Then for $\alpha$ large enough we have 
\begin{equation}
\label{eq-nonnul}
m_\alpha := \fint_{B_\alpha}|f_\alpha| \di \nu_{g_\alpha} \geq \frac 12.
\end{equation}
We also have for $x \in B_\alpha$
$$ |\nabla f_\alpha| \leq n(1+\eps_\alpha)^{n-1}|\nabla d H_\alpha|.$$
We then use Cauchy-Schwartz and Poincaré's inequalities to get 
$$\fint_{B_\alpha}| f_\alpha -m_\alpha|\di \nu_{g_\alpha} \left(\fint_{B_\alpha} | f_\alpha -m_\alpha|^2 \di \nu_{g_\alpha} \right)^{\frac 12}\leq Cr_\alpha \left(\fint_{B_\alpha}|\nabla f_\alpha|^2 \di \nu_{g_\alpha} \right)^{\frac{1}{2}} $$
Then using the doubling condition and property (iv) in the previous Proposition we get
$$\fint_{B_\alpha} | f_\alpha -m_\alpha|\di \nu_{g_\alpha} \leq Cr_\alpha \left(\fint_{B_\alpha}|\nabla  d H_\alpha|^2 \di \nu_{g_\alpha} \right)^{\frac{1}{2}} \leq C' \sqrt{ \eps_\alpha},$$
and for $\alpha$ large enough 
$$\fint_{B_\alpha} | f_\alpha -m_\alpha|\di \nu_{g_\alpha} \leq \frac 18.$$
Then combining this last inequality with \eqref{eq-nonnul} we obtain
\begin{equation}
\label{eq-nonnul2}
m_\alpha \geq \frac{3}{8}.
\end{equation}
We are going to contradict the previous lower bound by using the properties of $H_\alpha$ and Poincaré's inequality. 

We can write
\begin{align*}
|m_\alpha| &= \left| m_\alpha - \int_{\Omega_\alpha}f_\alpha\right| \\
& =\left| \fint_{\Omega_\alpha \times B_\alpha} (f_\alpha(x)-f_\alpha(y))\di \nu_{g_\alpha}(x) \di \nu_{g_\alpha (y)}\right| \\
& \leq \frac{\nu_{g_\alpha}(B_\alpha)}{\nu_{g_\alpha}(\Omega_\alpha)}\fint_{B_\alpha \times B_\alpha} |f_\alpha(x)-f_\alpha(y)|\di \nu_{g_\alpha}(x) \di \nu_{g_\alpha} (y) \\
&   \leq \frac{\nu_{g_\alpha}(B_\alpha)}{\nu_{g_\alpha}(\Omega_\alpha)} \left( \fint_{B_\alpha \times B_\alpha} |f_\alpha(x)-f_\alpha(y)|^2\di \nu_{g_\alpha}(x) \di \nu_{g_\alpha} (y)\right)^{\frac 12}
\end{align*}
Thanks to the first inclusion in \eqref{eq-inclusions}, for $\alpha$ large enough $B_{r_\alpha/8}(x)\subset \Omega_\alpha$. Then the ratio $\nu_{g_\alpha}(B_\alpha)/\nu_{g_\alpha}(\Omega_\alpha)$ has an upper uniform bound, since $(M_\alpha,g_\alpha)$ is doubling. Moreover
$$\fint_{B_\alpha \times B_\alpha} |f_\alpha(x)-f_\alpha(y)|^2\di \nu_{g_\alpha}(x) \di \nu_{g_\alpha} (y)=2 \fint_{B_\alpha}|f_\alpha-m_\alpha|^2\di \nu_{g_\alpha},$$
then thanks to the estimate on the hessian of $H_\alpha$ we can conclude 
$$|m_\alpha| \leq C r_\alpha^2 \fint_{B_\alpha}|\nabla dH_\alpha|^2\di \nu_{g_\alpha}\leq C' \eps_\alpha.$$
Since $\eps_\alpha $ tends to zero, this contradicts for $\alpha$ large enough inequality \eqref{eq-nonnul2}. Then if $M_\alpha$ is oriented, $H_\alpha$ is surjective from $\Omega_\alpha$ to $\mathbb{B}(\tau_\alpha)$. 

If $M_\alpha$ is not oriented, we consider the two-fold orientation covering $\hat{\pi}_\alpha : \hat{M}_\alpha to M_\alpha$ and choose $\hat x \in \hat \pi_\alpha^{-1}(x)$. We observe that $\hat M_\alpha$ endowed with the pull-back metric $\hat g_\alpha= \hat \pi_\alpha^*g_\alpha$ satisfies for all $s>0$
$$\Kato_s(\hat M_\alpha, \hat g_\alpha)=\Kato_s(M_\alpha, g_\alpha).$$
Then $(\hat M_\alpha, \hat g_\alpha)$ is PI at the same scale as $(M_\alpha,g_\alpha)$. Moreover, the map 
$$\hat H_\alpha = H_\alpha \circ \hat \pi_\alpha : B(\hat x,r_\alpha) \to \mathbb{B}(\tau_\alpha)$$ is $(1+\eps_\alpha)$-Lipschitz and satisfies properties (iii) and (iv) of the previous Proposition. Then we can apply the same argument as above and show that $\hat H_\alpha : \hat \pi_\alpha^{-1}(\Omega_\alpha) \to \mathbb{B}(\tau_\alpha)$ is surjective. It finally follows that $H_\alpha : \Omega_\alpha \to \mathbb{B}(\tau_\alpha)$ is also surjective. 
\end{proof} 

\appendix
\section*{Appendix}
\renewcommand{\thesubsection}{\Alph{subsection}}

\makeatletter
\renewcommand{\thetheorem}{\thesubsection.\arabic{theorem}}
\@addtoreset{theorem}{subsection}
\makeatother

In this appendix, we provide a proof of Theorem \ref{th:improvedKS} and of several other useful convergence results.  

For the two next subsections, we put ourselves in the following setting: we let $\{(X_\alpha, \dist_\alpha, \meas_\alpha, o_\alpha)\}_\alpha, (X,\dist,\mu,o)$ be proper geodesic pointed metric measure spaces such that
\begin{itemize}
\item
$(X_\alpha, \dist_\alpha, \meas_\alpha, o_\alpha) \stackrel{pmGH}{\longrightarrow} (X,\dist,\mu,o)$, and we use the sequences $\{R_\alpha\}$, $\{\eps_\alpha\}$ and $\{\Phi_\alpha\}$ given by Characterization \ref{chara},
\item there exists $\upkappa\ge 1$ and $ R>0$ such that the spaces $\{(X_\alpha, \dist_\alpha, \meas_\alpha)\}_\alpha$ are all $\upkappa$-doubling at scale $R$, hence so is $(X,\dist,\mu)$.
\end{itemize}

  \subsection{Approximation of functions} The following result is known by experts. It says that the space $\cC_c(X)$ is somehow the limit of the spaces $\left\{\cC_c(X_\alpha)\right\}$, in the sense that any $\varphi\in \cC_c(X)$ can be nicely approximated by functions  $\varphi_\alpha\in\cC_c(X_\alpha)$.
\begin{prop}\label{prop:approx}
For any $r>0$ and any $\alpha$ large enough,  we can build a linear map
$$\mathcal{A}_\alpha\colon \cC_c(B_r(o) )\rightarrow  \cC_c(X_\alpha)$$
such that the following holds for any $\varphi \in \cC_c(B_r(o))$.
\begin{enumerate}[i)]
\item If $\varphi\ge 0$, then $\mathcal{A}_\alpha\varphi\ge 0$ for any $\alpha$.
\item If $0\le \varphi \le L$, then $0\le\mathcal{A}_\alpha\varphi \le L$ for any $\alpha$.
\item  The convergence $\mathcal{A}_\alpha\varphi\stackrel{\cC_c}{\to}\varphi$ holds.
\item The functions $\{\mathcal{A}_\alpha\varphi \}_\alpha$ are uniformly equicontinuous.
\item There exists a constant $\bar{C}>0$ depending only on $\upkappa$ such that if $\varphi$ is $\Lambda$-Lipschitz, then $\mathcal{A}_\alpha\varphi$ is $\bar{C} \Lambda$-Lipschitz.
\end{enumerate}
\end{prop}

\begin{proof}

Let $r>0$. With no loss of generality we assume that $\sup_\alpha \eps_\alpha \le r/8$ and that $2r\le R_\alpha$; if this is not true, we let $\mathcal{A}_\alpha$ be the zero map for all $\alpha$ such that $\eps_\alpha > r/8$ or $2r> R_\alpha$. Let $\varphi\in \cC_c(B_r(o) )$\\

\textbf{Step 1.} [Construction of $\mathcal{A}_\alpha \varphi$] Let $\alpha$ be arbitrary. Let $\cD_\alpha\subset X_\alpha$ be a maximal $2\eps_\alpha$-separated set of points, i.e.~a maximal set such that $X_\alpha=\bigcup_{p\in \cD_\alpha} B_{2\eps_\alpha}(p)$ and for any $p, q\in \cD_\alpha$,
$$p \neq q \quad \Rightarrow \quad B_{\eps_\alpha}(p)\cap B_{\eps_\alpha}(q)=\emptyset.$$
For any $x \in X$, we set $\cV(x):=\cD_\alpha\cap B_{4\eps_\alpha}(x)$ and we point out that
\begin{equation}\label{cV}
\#\cV(x)\le \upkappa^4.
\end{equation}
Indeed we have $\cup_{p\in \cV(x)}B_{\eps_\alpha}(p)\subset B_{5\eps_\alpha}(x)$ and $\mu_\alpha\left(B_{5\eps_\alpha}(x)\right)\le \mu_\alpha\left(B_{9\eps_\alpha}(p)\right)\le \upkappa^4 \mu_\alpha\left(B_{\eps_\alpha}(p)\right)$ for any $p\in \cV(x)$.

Let us consider the $1$-Lipschitz function $\chi:[0,+\infty) \to [0,1]$ defined by 
\begin{equation}\label{coupure}
\chi(t)=
\begin{cases}
1& \mbox{ if } t \le 1, \\
2-t&  \mbox{ if } t \in [1,2], \\
0&  \mbox{ if } t \in [2,+\infty).
\end{cases}
\end{equation}
For any $p \in \cD_\alpha$,  we define $\hat{\xi}^\alpha_p,\sigma^\alpha,\xi^\alpha_p\colon X_\alpha\rightarrow \R$ by
\[
\hat{\xi}^\alpha_p(x) = \chi \left( \frac{\dist_\alpha(p,x)}{2\eps_\alpha} \right), \qquad \sigma^\alpha(x) = \sum_{p \in \cD_\alpha} \hat{\xi}_{p}^\alpha(x), \qquad \xi_p^\alpha(x) = \hat{\xi}_p^\alpha(x)/\sigma^\alpha(x),
\]
for any $x\in X$.  By construction, $1\le\sigma^\alpha\le \upkappa^4$ (the upper bound actually follows from \eqref{cV}), the function $\xi_p^\alpha$ is $(1+\upkappa^4)(2\eps_\alpha)^{-1}$-Lipschitz,  and
$$\sum_{p\in \cD_\alpha}  \xi_p^\alpha=1.$$
Then we define
$$\varphi_\alpha = \mathcal{A}_\alpha\varphi:=\sum_{p\in \cD_\alpha} \varphi\left(\Phi_\alpha(p)\right)\, \xi_p^\alpha.$$
As a linear combination of compactly supported Lipschitz functions, $\varphi_\alpha \in \Lip_c(X_\alpha)$. Moreover,  $\supp\varphi_\alpha\subset B_{r+5\eps_\alpha}(o_\alpha)$.  Linearity of the map $\mathcal{A}_\alpha$ is clear from the construction. Finally, properties ii) and iii) are trivially respected. \\

\textbf{Step 2.} [Convergence $\varphi_\alpha\stackrel{\cC_c}{\to} \varphi$]
Let us show that if $\omega_\varphi$ is the modulus of continuity of $\varphi$, defined by $\omega_\varphi(\delta):=\sup_{\dist(x,y)\le \delta} \left|\varphi(x)-\varphi(y)\right|$ for any $\delta>0$, then
\begin{equation}\label{eq:quantconv}
\|\varphi \circ \Phi_\alpha - \varphi_\alpha\|_{L^\infty(X_\alpha,\mu_\alpha)} \le \upkappa^4 \omega_\varphi(5\eps_\alpha)
\end{equation}
for any $\alpha$.  Take  $x \in X_\alpha$. Then
\begin{align*}
\left|\varphi \circ \Phi_\alpha(x) - \varphi_\alpha(x)\right| & =\left| \sum_{p \in \cD_\alpha } \left(\varphi \left(\Phi_\alpha(x)\right) - \varphi\left(\Phi_\alpha(p)\right) \right) \xi_{p}^\alpha(x)\right|\\
& \le \sum_{p\in \cV(x)} \omega_\varphi\left(\dist\left(\Phi_\alpha(x),\Phi_\alpha(p)\right)\right) \xi_{p}^\alpha(x)\\
&\le \upkappa^4\omega_\varphi\left(5\eps_\alpha\right).
\end{align*}
Hence \eqref{eq:quantconv} is proved, and consequently $\varphi_\alpha\stackrel{\cC_c}{\to} \varphi$.\\

\textbf{Step 3.} [Equicontinuity and Lipschitz estimate] Take $x,y\in X_\alpha$. Then 
\begin{align*}
\varphi_\alpha(x)-\varphi_\alpha(y)&=\sum_{p \in \cD_\alpha } \left[ \varphi\left(\Phi_\alpha(p)\right)- \varphi\left(\Phi_\alpha(x)\right)\right]\left(\xi_{p}^\alpha(x)-\xi_{p}^\alpha(y)\right)\\
&=\sum_{p \in \cV(x)\cup \cV(y) } \left[ \varphi\left(\Phi_\alpha(p)\right)- \varphi\left(\Phi_\alpha(x)\right)\right]\left(\xi_{p}^\alpha(x)-\xi_{p}^\alpha(y)\right).\end{align*}
Observe that when $p \in \cV(x)\cup \cV(y)$, then $\dist_\alpha(x,p)\le  4\eps_\alpha+\dist_\alpha(x,y)$, hence  we have $$\left| \varphi\left(\Phi_\alpha(p)\right)- \varphi\left(\Phi_\alpha(x)\right)\right|\le \omega_\varphi\left(\dist_\alpha(x,p)+\eps_\alpha\right)\le \omega_\varphi\left(5\eps_\alpha+\dist_\alpha(x,y)\right).$$
Using  $$\sum_{p \in \cD_\alpha } \left|\xi_{p}^\alpha(x)-\xi_{p}^\alpha(y)\right|\le \min\left\{2, \frac{1+\upkappa^4}{\eps_\alpha}\dist_\alpha(x,y)\right\},$$ we obtain the 
estimate
$$\left|\varphi_\alpha(x)-\varphi_\alpha(y)\right|\le\begin{cases}
\frac{(1+\upkappa^4)\omega_\varphi\left(6\eps_\alpha\right)}{\eps_\alpha}\dist_\alpha(x,y)&\text{ if }\dist_\alpha(x,y)\le \eps_\alpha,\\
2\omega_\varphi\left(6\dist_\alpha(x,y) \right)&\text{ if }\dist_\alpha(x,y)\ge \eps_\alpha.
\end{cases}$$
In particular, if $\varphi$ is $\Lambda$-Lipschitz, then $\varphi_\alpha$ is $6\left(1+\upkappa^4\right)\Lambda$ -Lipschitz. 
This estimate also implies the equicontinuity of the sequence: if $\delta\in (0,1)$ and
$\dist_\alpha(x,y)\le \delta$ then
 $$\left|\varphi_\alpha(x)-\varphi_\alpha(y)\right|\le \omega_\varphi\left(6\sqrt{\delta}\right)+2(1+\upkappa^4)\|\varphi\|_{L^\infty} \sqrt{\delta}.$$
\end{proof}

\subsection{Convergence of integrals}
In this subsection, we prove two results about convergence of integrals under pmGH convergence that are used repeatedly in this article. We recall our setting: $\{(X_\alpha, \dist_\alpha, \meas_\alpha, o_\alpha)\}_\alpha, (X,\dist,\mu,o)$ are $\upkappa$-doubling at scale $R$ proper geodesic pointed metric measure spaces such that $(X_\alpha, \dist_\alpha, \meas_\alpha, o_\alpha) \to (X,\dist,\mu,o)$ in the pmGH sense,  and we use the notations of Characterization \ref{chara}.

It is easy to prove the first convergence result.

\begin{prop}
\label{prop:cvproduit}
Let $u \in C(X), v \in L^2(X,\mu)$ and $u_\alpha \in C(X_\alpha), v_\alpha \in L^2(X_\alpha,\mu_\alpha)$ for any $\alpha$ be such that:
\begin{itemize}
\item $\sup_\alpha \|u_\alpha\|_{L^\infty}<\infty$,
\item $u_\alpha \to u$ uniformly on compact subsets,
\item $v_\alpha \to v$ strongly in $L^2$.
\end{itemize} 
Then
\[
\lim_\alpha \int_{X_\alpha} u_\alpha v_\alpha^2\di\mu_\alpha=\int_{X_\alpha} u v^2\di\mu .
\]
\end{prop}
\begin{proof}
The result follows from establishing the weak convergence $u_\alpha v_\alpha\stackrel{L^2}{ \weakto }u v$.
Remark first that the hypotheses of the Proposition imply $$\sup_\alpha \|u_\alpha v_\alpha\|_{L^2}<+\infty.$$
Moreover, when $\varphi_\alpha\stackrel{\cC_c}{\to}\varphi$,  then obviously $\varphi_\alpha u_\alpha\stackrel{\cC_c}{\to}\varphi u$, and as $v_\alpha\stackrel{L^2}{ \weakto }v$ we get
$$\lim_{\alpha \to+\infty} \int_{X_\alpha}  \varphi_\alpha u_\alpha v_\alpha\di\mu_\alpha=\int_{X}  \varphi u v\di\mu .$$ \end{proof}
We make now  a few useful remarks.
The first point is that for any $r>0$:
\begin{equation}\label{cvmB}
x_\alpha\in X_\alpha\to x\in X \quad \Longrightarrow  \quad \lim\limits_{\alpha\to+\infty} \mu_{\alpha}\left(B_r(x_\alpha)\right)=\mu\left(B_r(x)\right).\end{equation}
In full generality, this convergence result holds when $\mu\left(\partial B_r(x)\right)=0$ \cite[Theorem 2.1]{billingsley2013convergence}, 
and this condition is guaranteed by the doubling condition [\ref{prop:doubling}-v)]. We also have that for any $r>0$:
 \begin{equation}\label{cvmF}
\varphi_\alpha\stackrel{\cC_c}{\to} \varphi, \, \,\,  x_\alpha\in X_\alpha\to x\in X \,\,\,\Longrightarrow  \,\,\, \lim_{\alpha\to+\infty} \int_{B_r(x_\alpha)} \varphi_\alpha\di \mu_{\alpha}= \int_{B_r(x)} \varphi\di \mu.\end{equation}
Even better, the convergence result takes place as soon as $\varphi_\alpha\in \cC( X_\alpha)$ converges uniformly on compact set to $\varphi\in \cC( X)$.
 
The above convergence results \eqref{cvmB} and \eqref{cvmF} imply, by definition, that when $r>0$, $p>1$ and
 $x_\alpha\in X_\alpha\to x\in X$, then
 \begin{equation}\label{cvLpun}\mathbf{1}_{B_r(x_\alpha)}\stackrel{L^p}{\to} \mathbf{1}_{B_r(x)}.\end{equation}
 This implies the following criterion for $L^p$ weak convergence:
 \begin{lemma}\label{lem:criterion} For $p\in(1,\infty)$, let $u_\alpha \in L^p(X_\alpha,\mu_\alpha)$ for any $\alpha$ and $u \in L^p(X,\mu)$ be given. Then $u_\alpha\stackrel{L^p}{ \weakto} u$ if and only if
\begin{equation}\label{eq:criterion}
\begin{cases} \sup_\alpha \| u_\alpha\|_{L^p}<\infty& \\
 x_\alpha \in X_\alpha \to x \in X,  \,\, r>0 \quad \Rightarrow \quad  \displaystyle \lim_\alpha \fint_{B_r(x_\alpha)} u_\alpha \di\mu_\alpha= \fint_{B_r(x)} u\di\mu.&
\end{cases}
\end{equation}
\end{lemma}
\begin{proof} The direct implication follows from \eqref{cvLpun}.  The converse one is a consequence of the fact that if $B\subset A$ is such that $\{u_\beta\}_{\beta \in B}$ converges weakly in $L^p$ to $v$ then $\int_{B_r(x)} u\di\mu=\int_{B_r(x)} v\di\mu$ for any $x\in X$ and any $r>0$,  and so $$\fint_{B_r(x)} u\di\mu=\fint_{B_r(x)} v\di\mu.$$
By Lebesgue differentiation theorem (true on any doubling space), this implies that $u=v$ $\mu$-a.e.
\end{proof}

Our second convergence result is the following.

\begin{prop}\label{prop:convdom} 
  Let $u \in \cC(X)$ and $u_\alpha \in \cC(X_\alpha)$ for any $\alpha$ be such that
\begin{itemize}
\item $u_\alpha \to u$ uniformly on compact subsets,
\item there exists $C, \beta>0$ such that for any $\alpha$ and $\mu_\alpha$-a.e.~$x \in X_\alpha$,
\begin{equation}\label{Cb}|u_\alpha(x)|\le C\, e^{-\beta \dist^2_\alpha(o_\alpha,x)}.\end{equation}
\end{itemize} 
Then the functions $u_\alpha$ and $u$ are $L^p$-integrable for any $p\ge 1$ and
\begin{enumerate}
\item\label{1} $\int_{X_\alpha} u_\alpha \di \mu_\alpha \to \int_Xu \di \mu$,
\item\label{2} $u_\alpha \to u$ strongly in $L^p$ when $p>1$.
\end{enumerate}
\end{prop}

For the proof of this proposition,  we use the following lemma which is a consequence of the ideas of the proof of \eqref{eq:majTheta}.

\begin{lemma}
\label{lemma:estimeeVol}
Let $(X,\dist,\mu)$ be $\upkappa$-doubling at scale $R$. Then for any $c>0$ there exists $A>0$ depending only on $c$, $\upkappa$ and $R$ such that for any $o \in X$,
\begin{equation}\label{doub1}
\int_X e^{-c \dist^2(o,x)}\di \mu(x) \le A \mu(B_{R}(o_\alpha)).
\end{equation}
Moreover, there exists $\upbeta:(0,+\infty) \to (0,+\infty)$ depending only on $c$, $C$ and $R$ such that $\upbeta(\rho)\to 0$ when $\rho \to +\infty$ and for any $\rho>0$,
\begin{equation}\label{doub2}
\int_{X\backslash B_{\rho}(o)} e^{-c\dist^2(o,x)}\di \mu(x) \le \upbeta(\rho)\mu(B_{R}(o)).
\end{equation}
\end{lemma}
\begin{proof}
We have
$$\int_X e^{-c \dist^2(o,x)}\di \mu(x)\le \mu(B_R(o))+ \int_{X\backslash B_{R}(o)} e^{-c\dist^2(o,x)}\di \mu(x).$$
Moreover, using Cavalieri's formula and Proposition \ref{prop:doubling}-ii) we get that for any $\rho\ge R$,
\begin{align*}
\int_{X\backslash B_{\rho}(o)} e^{-c\dist^2(o,x)}\di \mu(x)&=\int_\rho^{+\infty} 2cre^{-cr^2} \left(\mu(B_{r}(o)\right) \di r\\
&\le  \mu\left(B_{R}(o)\right)\int_\rho^{+\infty} 2cre^{-cr^2+\lambda\frac{r}{R} }\di r.
\end{align*}\end{proof}

We can now prove Proposition \ref{prop:convdom}.

\begin{proof}[Proof of Proposition \ref{prop:convdom}]

As a consequence of the previous lemma, for any $p\ge 1$, we get
\begin{equation}\label{queueI}\int_{X_\alpha\backslash B_{\rho}(o_\alpha)}|u_\alpha|^p\di\mu_\alpha\le  \upbeta(\rho)\mu_\alpha(B_{R}(o_\alpha))\end{equation}
where $\upbeta$ depend only on $p,\beta,\upkappa$. The discussion above implies that:
$$u_\alpha\stackrel{L^p_{loc}}{\to}  u.$$
With the estimate \eqref{queueI}, we get that that the sequence $\{\|u_\alpha\|_{L^p}\}$ is bounded hence  $u_\alpha\stackrel{L^p}{\weakto}  u.$ But the estimate  \eqref{queueI} implies the validity of the intervention on limits :
$$\lim_{\rho\to+\infty}\lim_\alpha \int_{ B_{\rho}(o_\alpha)}|u_\alpha|^p\di\mu_\alpha=\lim_\alpha \lim_{\rho\to+\infty}\int_{ B_{\rho}(o_\alpha)}|u_\alpha|^p\di\mu_\alpha;$$ that is to say 
$$\lim_\alpha \int_{ X_\alpha}|u_\alpha|^p\di\mu_\alpha=\int_{ X}|u|^p\di\mu.$$ Thus
$u_\alpha\stackrel{L^p}{\to}  u.$ 
The statement  Proposition \ref{prop:convdom}-\eqref{1} is proven in the same way.

\end{proof}
\begin{rem} When the function $u_\alpha$ are only assumed to be measurable, the conclusion  Proposition \ref{prop:convdom}-\eqref{1} holds assuming $u_\alpha\stackrel{L^2_{loc}}{\weakto}  u$ in place of the uniform convergence on compact sets. Indeed this hypothesis implies that for any $R>0$:
$$\lim_{\alpha} \int_{B_R(o_\alpha)} u_\alpha\di\mu_\alpha= \int_{B_R(o)} u\di\mu.$$
And the proof of  Proposition \ref{prop:convdom} can be applied. But we won't need this refinement here.
\end{rem}

\subsection{Heat kernel characterization of $\mathrm{PI}$-Dirichlet spaces.}

In this subsection, we provide a set of conditions on the heat kernel of a metric Dirichlet space $(X,\dist,\mu,\cE)$ for him to be regular, strongly local and with $\dist_\cE$ being a distance bi-Lipschitz equivalent to $\dist$.  We use this result in the next subsection to prove Theorem  \ref{th:improvedKS}. Wde let $R>0$ be fixed throughout this subsection.

\subsubsection{Heat kernel bound} We need an important statement about regular, strongly local Dirichlet spaces. It is the combination of several well-known theorems \cite{Grigoryan92,Saloff-Coste,sturm1996analysis}. If $(X,\dist,\mu,\cE)$ is a metric measure space equipped with a Dirichlet form with associated operator $L$, we call local solution of the heat equation any function $u$ satisfying $(\partial_t + L)u=0$ in the sense of \cite{sturm1994analysis} (see also \cite[Def.~2.3]{CT19}).

\begin{theorem}
\label{thm:eq}
Let $(X,\cT,\mu,\cE)$ be a regular, strongly local Dirichlet space with $\dist_\cE$ being a distance compatible with $\cT$.
Then the following are equivalent:
\begin{itemize}
\item[(c1)] $(X,\dist_\cE, \mu, \cE)$ is a $\mathrm{PI}_{\upkappa,\upgamma}(R)$ Dirichlet space,
\item[(c2)] $\cE$ admits a heat kernel $H$ satisfying Gaussian bounds: there exists $\upbeta>0$ such that
\begin{equation}
\frac{\upbeta^{-1}}{\mu(B_{\sqrt{t}}(x))} e^{-\upbeta\frac{\dist_\cE^2(x,y)}{ t}}
\le H(t,x,y)\le \frac{\upbeta}{\mu(B_{\sqrt{t}}(x))} e^{-\frac{\dist_\cE^2(x,y)}{\upbeta t}}
\end{equation}
for all $x,y\in X$ and $t\in (0,R^2]$,
\item[(c3)] the local solutions of the heat equation satisfy a uniform Hölder regularity estimate: there exist constants $\alpha \in (0,1]$, $A >0$ such that if $B$ is a ball of radius $r \leq R$ and $u:(0,r^2)\times 2B \rightarrow(0,\infty)$ is a local solution of the heat equation then for any $s,t\in \left(r^2/4,3r^2/4\right)$ and $x,y\in B$,
\begin{equation}\label{locHold}
\left|u(s,x)-u(t,y)\right|\le \frac{A}{r^\upalpha}\left( \sqrt{|t-s|}+\dist_\cE(x,y)\right)^\upalpha \sup_{(0,r^2)\times B} |u|.
\end{equation}
\end{itemize}
\end{theorem}

As a corollary, the heat kernel of a regular, strongly local $\mathrm{PI}$ Dirichlet space satisfies the following properties:

\begin{prop}\label{prop:HKPI}
Let $(X,\dist_\cE,\mu,\cE)$ be a regular, strongly local,  $\mathrm{PI}_{\upkappa,\upgamma}(R)$ Dirichlet space for some $\upkappa \ge 1$ and $\upgamma>0$.  Let $\dist$ be a distance on $X$ bi-Lipschitz equivalent to $\dist_\cE$.  Then $\cE$ admits a heat kernel $H$ such that the following holds:
\begin{itemize}
\item[(a)] $H$ is stochastically complete \eqref{eq:stocom},
\item[(b)] there exists $\upbeta\ge1$ such that the following Gaussian double-sided bounds hold:
\begin{equation}\label{eq:gaussian}
\frac{\upbeta^{-1}}{\mu(B_{\sqrt{t}}(x))} e^{-\upbeta\frac{\dist^2(x,y)}{ t}}
\le H(t,x,y)\le \frac{\upbeta}{\mu(B_{\sqrt{t}}(x))} e^{-\frac{\dist^2(x,y)}{\upbeta t}}
\end{equation}
for all $x,y\in X$ and $t\in (0,R^2]$,
\item[(c)] there exists $\upalpha \in (0,1)$ and $A>0$ such that for any $x,y,z \in X$ and $s,t\in (0,R)$ such that $|t-s|\leq t/4$ and $\dist(y,z) \leq \sqrt{t}$,
\begin{equation}
\label{eq:Rem(c)}
|H(s,x,z)-H(t,x,y)|\le A \left( \frac{\sqrt{|t-s|}+\dist(y,z)}{\sqrt{t}} \right)^\upalpha\,H(t,x,y).
\end{equation}
\end{itemize}
\end{prop}

The next theorem is our key statement to establish Theorem \ref{th:improvedKS}. 

\begin{theorem}
\label{propo:hk2DS} 
Let $(X,\dist, \mu,\cE)$ be a metric Dirichlet space such that $(X,\dist)$ is geodesic and for which a heat kernel $H$ exists and satisfies $(a)$, $(b)$ and $(c)$ in the previous proposition \ref{prop:HKPI}. Then $(X,\dist_\cE \mu,\cE)$ is a $\PI_{\upkappa,\upgamma}(R)$ Dirichlet space for some $\upkappa \ge 1$ and $\upgamma>0$ depending only on the constants from $(b)$ and $(c)$, and the distance $\dist$ is bi-Lipschitz equivalent to the intrinsic distance $\dist_{\cE}$. 
\end{theorem}
\subsubsection{Domain characterization}
In order to prove Theorem \ref{propo:hk2DS}, we start by showing the following crucial proposition. It is a generalization of a similar result of A. Grigor'yan, J. Hu and K-S. Lau \cite[Theorem 4.2]{GriHuLau} (see also \cite[Corollary 4.2]{Grigoryan10} and references therein) where the mesure is additionally assumed uniformly Ahlfors regular.

\begin{prop}
Under the assumptions of Theorem \ref{propo:hk2DS} , the domain of $\cE$ coincides with the Besov space $B_{2,\infty}(X)$, consisting of the functions $u\in L^{2}(X,\mu)$ such that 
$$N(u)^2:=\limsup_{r \rightarrow 0^+} \frac{1}{r^2} \int_X \fint_{B_{r}(x)} (u(x)-u(y))^2\di\mu(y)\di\mu(x) < \infty.$$
Moreover, there is a constant $C$ depending only on $\upbeta$ such that for any $u\in \cD(\cE)$,
$$\frac{1}{C} N(u)^2\le \cE(u)\le C N(u)^2.$$
\end{prop}

\begin{proof}
For any function $u \in L^{2}(X,\mu)$, define the decreasing function $t \mapsto \cE_t(u)$ where for any $t>0$,
$$\cE_t(u):=\frac{1}{t}\langle u-e^{-tL}u,u \rangle = \int_{X \times X}H(t,x,y)(u(x)-u(y))^2 \frac{\di\mu(x) \di\mu(y)}{2t}\, \cdot$$
We first observe that, because of assumption (a), a function $u$ belongs to $\cD(\cE)$ if and only if $\sup_t\cE_t(u) < \infty$. Moreover, if $u \in \cD(\cE)$, then $\cE(u)=\lim_{t \rightarrow 0^+}\cE_t(u).$ This is explained in \cite[Sect.~2.2]{Grigoryan10}, for instance.

\medskip

\noindent \textbf{Step 1.} We begin with showing the easiest inclusion, namely $\cD(\cE) \subset B_{2, \infty}(X)$. Take $u \in \cD(\cE)$. For $t>0$, set
$$I(t):=\int_{\{(x,y)\in X \times X : \dist(x,y) \leq \sqrt{t}\}} H(t,x,y)(u(x)-u(y))^2 \frac{\di\mu(x)\di\mu(y)}{2t} \, ,$$
and observe that $I(t) \leq \cE_t(u)\leq \cE(u)$. The lower bound for the heat kernel given by assumption $(b)$ implies 
\begin{align*}
I(t) & \geq \int_{\{(x,y)\in X \times X: \dist(x,y) \leq \sqrt{t}\}}\frac{\upbeta^{-1}}{\mu(B_{\sqrt{t}}(x))}e^{-\frac{\upbeta\dist^2(x,y)}{t}}(u(x)-u(y))^2 \frac{d\mu(x)d\mu(y)}{2t} \\
& \geq \frac{\upbeta^{-1}e^{-\upbeta}}{2t}\int_X \fint_{B_{\sqrt{t}}(x)} (u(x)-u(y))^2 d\mu(x)d\mu(y),
\end{align*}
hence letting $t$ tend to $0$ shows that $N(u)^2 \leq 2\upbeta e^{\upbeta}\cE(u)$.

\medskip

\noindent  \textbf{Step 2.} In order to prove the converse inclusion, we need some volume estimates. Our assumptions imply that the measure $\mu$ is doubling at scale $R$. Indeed for all $x \in X$ and $r \leq R$, thanks to assumptions (a) and (b) we have
$$\upbeta e^{-4\upbeta}\frac{\mu(B_{2r}(x))}{\mu(B_{r}(x))}\le \int_{B_{2r}(x)}H(r^2,x,y)d\mu(y)\le \int_{X}H(t,x,y)d\mu(y)= 1,$$ 
and therefore
$$\mu(B_{2r}(x))\le \beta e^{4\upbeta} \mu(B_{r}(x)).$$
Because of the doubling condition at scale $R$, we obtain for $s \leq r \leq 2R$, $x \in X$
\begin{equation*}
\label{eq:Ia}
\mu(B_{r}(x)) \leq C\left(\frac{r}{s} \right)^{\nu}\mu(B_{s}(x)),
\end{equation*}
where $\nu$ and $C$ depends only on $\upbeta$ (see Proposition \ref{prop:doubling}). 
The Gaussian estimate of the heat kernel \eqref{eq:gaussian} implies that if $0<t\le\tau\le R^2$ then 
\begin{equation}\label{Harnackl}
H(t,x,y)\le H(\tau,x,y)\, C\left(\frac{\tau}{t}\right)^{\nu/2} e^{ -\dist^2(x,y)\left(\frac{1}{\upbeta t}-\frac{\upbeta}{\tau}\right)}
\end{equation}

We introduce $\Omega_r=\{(x,y)\in X \times X \, :\, \dist(x,y) \geq r\}$ and 
$$I_\lambda(t):=\int_{ (X \times X)\setminus \Omega_{\lambda \sqrt{t} } } H(t,x,y)(u(x)-u(y))^2 \frac{\di\mu(x)\di\mu(y)}{2t}.$$
The same reasoning as in Step 1 implies that for $\lambda\ge 1$ and $t>0$ such that $\lambda\sqrt{ t}<R$ tehn
\begin{equation}
I_\lambda(t)\le \frac{C\lambda^\nu}{2t} \int_X \fint_{B_{\lambda\sqrt{ t}}(x)} (u(x)-u(y))^2\di\mu(y)\di\mu(x).
\end{equation}
Using the estimate \eqref{Harnackl} with $\tau=\lambda t$ and assuming $\lambda\ge 1$ and $\lambda^2 t\le R^2$, we estimate:
\begin{align*}\cE_t(u)-I_\lambda(t)
&\le C\lambda^{\nu/2} e^{-\lambda^2 t\left(\frac{1}{\upbeta t}-\frac{\upbeta}{\lambda t}\right)}  \int_{\Omega_{ \lambda \sqrt{t}}} H(\lambda t,x,y)(u(x)-u(y))^2 \frac{\di\mu(x)\di\mu(y)}{2t}\\
&\le C\lambda^{\nu/2+1} e^{-\lambda^2 \left(\frac{1}{\upbeta }-\frac{\upbeta}{\lambda }\right)}\cE_{\lambda t}(u)\\
&\le C\lambda^{\nu/2+1} e^{-\lambda^2 \left(\frac{1}{\upbeta }-\frac{\upbeta}{\lambda }\right)}\cE_{ t}(u),
\end{align*}
where we have used that $t\mapsto \cE_t(u)$ is non increasing. If we choose $\lambda=\lambda(\upbeta)$ sufficiently large so that 
$C\lambda^{\nu+1} e^{-\lambda^2 \left(\frac{1}{\upbeta }-\frac{\upbeta}{\lambda }\right)}\le \frac 12$, then we get
$$\cE_t(u)\le \frac{2C\lambda^\nu}{2t} \int_X \fint_{B_{\lambda\sqrt{ t}}(x)} (u(x)-u(y))^2\di\mu(y)\di\mu(x).$$
Hence the result.
\end{proof}
\begin{rem}\label{rem:gene}
The above reasoning implies that if $U\colon X\times X\rightarrow \R_+$ is a non negative integrable function such that the limit 
$\lim_{t\to 0+} \int_{X \times X} H(t,x,y)C(x,y)\frac{\di\mu(x)\di\mu(y)}{2t}$ exist and is finite  then
$$\lim_{t\to 0+} \int_{X \times X} H(t,x,y)C(x,y)\frac{\di\mu(x)\di\mu(y)}{2t}\le C(\upbeta) \limsup_{r \rightarrow 0^+} \int_X \fint_{B_{r}(x)} C(x,y) \frac{\di\mu(y)\di\mu(x)}{r^2}.$$
\end{rem}
We are now in position to prove Theorem \ref{propo:hk2DS}.

\begin{proof}[Proof of Theorem \ref{propo:hk2DS}]

\hfill

\textbf{Regularity.} We start by showing that $(X,\dist, \mu, \cE)$ is regular, that is we prove that the space $\cC_c(X)\cap \cD(\cE)$ is dense in $(\cC_c(X), ||\cdot||_{\infty})$ and in $(\cD(\cE), |\cdot|_{\cD(\cE)})$. 

Observe that $\mbox{Lip}_c(X)$ is contained in $\cD(\cE)$. Indeed, for any $u \in \mbox{Lip}_c(X)$ there exists $\Lambda$ such that for all $x,y \in X$
$$|u(x)-u(y)|^2 \leq \Lambda \dist(x,y),$$
and there exists $\rho >0$ such that the support of $u$ is included in the ball $B_{\rho}(o)$. Therefore for any $r>0$ and $x \in X$ we have
\[
e_r(x):= \fint_{B_r(x)} \left(u(x)-u(y)\right)^2\,\di \mu(y) \leq \Lambda^2r^2
\]
and moreover $e_r(x)=0$ if $x \notin B_{\rho+1}(o)$. As a consequence, for any $u \in \mbox{Lip}_c(X)$, there exists $\rho$ such that 
$$N(u)^2 \leq \mu(B_{\rho+1}(o))\Lambda^2,$$
thus $\mbox{Lip}_c(X) \subset B_{2,\infty}(X)$ and by the previous theorem $\mbox{Lip}_c(X) \subset \cD(\cE)$. Since $\mbox{Lip}_c(X)$ is dense in $(\cC_c(X),||\cdot||_{\infty})$, this implies that $\cC_c(X) \cap \cD(\cE)$ is also dense in $(\cC_c(X),||\cdot||_{\infty})$. 

In order to prove that $\cC_c(X) \cap \cD(\cE)$ is dense in $\cD(\eps)$, we follow the same argument as in the proof of \cite[Prop.~3.8]{CT19} i.e.~we show that if $t \in (0,R^2)$ and $u \in L^2(X,\mu)$, then $f=e^{-tL}u$ belongs to $\cC_0(X)$. To see that $f$ tends to zero at infinity, notice that the upper bound for the heat kernel implies
$$|f(x)|\le \upbeta \frac{1}{\mu(B_{\sqrt{t}}(x))} e^{-\frac{\dist^2(x,\mathrm{supp} u)}{\upbeta t} }\int_X |u|\di\mu$$
for any $x \in X$; from Proposition \ref{prop:doubling}-i), if $o \in \mbox{supp} \, u$ we obtain
$$|f(x)|\le  \frac{C}{\mu(B_{\sqrt{t}}(o))} e^{-\frac{\dist^2(x,\mathrm{supp} u)}{\upbeta t}+\lambda\frac{\dist(o,x)}{\sqrt{t}}}\int_X |u|\di\mu,$$
therefore $f$ is bounded and tends to zero as $\dist(o,\cdot)$ goes to infinity. 

As for the continuity of $f$, assumption (c) ensures that for any $x,x' \in X$ such that $\dist(x,x') \leq \sqrt{t}$ we have
$$|f(x)-f(x')|\le A \left(\frac{\dist(x,x')}{\sqrt{t}}\right)^\upalpha (e^{-tL}|u|)(x).$$
Since $e^{-tL}|u|$ is also bounded, this shows that $f$ is continuous. 

\medskip

\textbf{Strong locality.} We aim to prove that if $u,v \in \cD(\cE)$ have compact supports and if $u$ is constant in a neighbourhood of $\mbox{supp}(v)$, then $\cE(u,v)=0$. Assume that both $u$ and $v$ are supported in $B_{\rho}(o)$ and denote by $K$ the support of $v$. There exist $\eta >0$ and $c \in \setR$ such that if $\dist(x, K) \leq \eta$, then $u(x)=c$. Let us introduce $K^r=\bigcup_{x \in K}B_{r}(x)$. Then for any $r \leq \eta$, $u$ is constantly equal to $c$ on $K^r$. 

As in the previous theorem, we can define
$$\cE_t(u,v)= \int_{X \times X} H(t,x,y) (u(x)-u(y))(v(x)-v(y))\frac{\di\mu(x)\di\mu(y)}{2t},$$
and we have $\cE(u,v)=\lim_{t \rightarrow 0^+}\cE_t(u,v)$. From Remark \ref{rem:gene}, there exists a constant $C>0$ such that
$$|\cE(u,v)| \leq C \limsup_{r\rightarrow 0^+} \frac{1}{r^2}\int_X\left(\fint_{B_{r}(x)}\left|u(x)-u(y)\right|\left|v(x)-v(y)\right|\, \di\mu(y)\right)\di\mu(x).$$

Now observe that for any $r>0$, if $x \notin K^r$ and $y \in B_{r}(x)$, the triangle inequality ensures that $\dist(y,K)>0$, thus both $v(x)$ and $v(y)$ are equal to zero. We are then left with considering
$$\limsup_{r\rightarrow 0^+} \frac{1}{r^2}\int_{K^r}\left(\fint_{B_{r}(x)}\left|u(x)-u(y)\right|\left|v(x)-v(y)\right|\, \di\mu(y)\right)\di\mu(x).$$
But the same arguing implies that when $r \leq \eta/2$, $x \in K^r$ and $y \in B_{r}(x)$, then $y \in K^{2r}$;  as a consequence both $x$ and $y$ belong to $K^{\eta}$, so we have $u(x)=u(y)=c$. Finally, for $r \leq \eta/2$ we get

$$\int_X\left(\fint_{B_{r}(x)}\left|u(x)-u(y)\right|\left|v(x)-v(y)\right|\, \di\mu(y)\right)d\mu(x) = 0.$$

This ensures that $\cE(u,v)=0$ and thus $(X,\dist, \mu,\cE)$ is strongly local.
\medskip

\textbf{Equivalence between the distance and the intrinsic distance.} Let us begin with proving the existence of $C>0$ such that
\[
\dist_\cE \ge C \dist.
\]
Again from Remark \ref{rem:gene}, there exists a constant $C$ such that for any $u\in \cD(\cE)$ and  $\phi \in \cC_c(X) \cap\cD(\cE)$ with $\phi\ge 0$ then
$$\int_X \phi\di\Gamma(u)\le C \limsup_{r\rightarrow 0^+} \frac{1}{r^2}\int_X \phi(x)\left(\fint_{B_{r}(x)}\left(u(x)-u(y)\right)^2\, \di\mu(y)\right)\di\mu(x).$$
If $u\in \Lip_c(X)$ then
$$\int_X \phi\di\Gamma(u)\le C \Lip(u)^2 \int_X \phi\di\mu.$$
Hence $$\di\Gamma(u)\le C \Lip(u)^2\di\mu.$$
Take $x,y\in X$ and set $r:=\dist(x,y)$ and 
$$u(z):=\chi\left(\frac{\dist(x,z)}{2r}\right)\, \dist(x,z)$$ for any $z\in Z$, where $\chi$ is defined as in \eqref{coupure}. Then $u\in \Lip_c(X)$ and $u(y)-u(x)=\dist(x,y)$. Moreover, $\Lip(u)\le 3$. Thus, testing $u/(9C)$ in the definition of $\dist_\cE$, we get
$$\dist_\cE(x,y) \ge (3\sqrt{C})^{-1} \dist(x,y).$$

Now let us prove
\[
\dist_\cE \le \sqrt{\upbeta/2} \dist.
\]
We act as in the proof of \cite[Prop.~3.9]{CT19}. We consider a bounded function $v \in \cD_{loc}(\cE)\cap C(X)$ such that $\Gamma(v) \le \mu$. For any $a \ge 0$, $t \in (0,R^2)$ and $x \in X$ we set $\xi_a(x,t):=av(x)-a^2t/2$. Take $x,y \in X$ and assume with no loss of generality that $v(y)-v(x)>0$.  From \cite[Claim 3.10]{CT19} applied to $f=\mathrm{1}_{B_{\sqrt{t}}(y)}$, one gets
\[
\int_{B_{\sqrt{t}}(x)} \left( \int_{B_{\sqrt{t}}(y)} H(t,z_1,z_2) \di \mu(z_2)\right)^2 e^{\xi_a(z_1,t)} \di \mu(z_1) \le \int_{B_{\sqrt{t}}(y)} e^{av} \di \mu
\]
which leads to
{\small\begin{equation}\label{eq:eq}
 \mu(B_{\sqrt{t}}(x))\mu(B_{\sqrt{t}}(y)) \exp\left(\, a \delta_t(x,y) - a^2\frac{t}{2}\right)\, \inf_{B_{\sqrt{t}}(x))\times B_{\sqrt{t}}(y)} H(t,\cdot,\cdot)^2\le 1,
\end{equation}}
where we define $$\delta_t(x,y):=\inf_{B_{\sqrt{t}}(y)}v-\sup_{B_{\sqrt{t}}(x)}v.$$
Observe that
\[
\sup_{B_{\sqrt{t}}(x)\times B_{\sqrt{t}}(y)} \dist(\cdot,\cdot) \le \dist(x,y)+2\sqrt{t},
\]
so that the Gaussian lower bound in \eqref{eq:gaussian}
yields, for any $(z_1,z_2) \in B_{\sqrt{t}}(x)\times B_{\sqrt{t}}(y)$,
\[
H(t,z_1,z_2) \ge \frac{\upbeta^{-1}}{\mu(B_{\sqrt{t}}(z_1))} \exp\left(-\upbeta\frac{ (\dist(x,y)+2\sqrt{t})^2}{4t}\right);
\]
the doubling condition implies $\mu(B_{\sqrt{t}}(z_1)) \le \mu(B_{2\sqrt{t}}(x) \le \upkappa \mu(B_{\sqrt{t}}(x))$,  then we get
\[
\left( \inf_{B_{\sqrt{t}}(x))\times B_{\sqrt{t}}(y)} H(t,\cdot,\cdot)^2\right) \ge \frac{(\upbeta \upkappa)^{-2}}{\mu(B_{\sqrt{t}}(x)^2} \,\exp\left(-\upbeta\frac{ (\dist(x,y)+2\sqrt{t})^2}{2t}\right).
\]
The continuity of $v$ yields $\lim_{t\to 0+} \delta_t(x,y)=v(y)-v(x)$,  hence we can take $t$ small enough to ensure $\delta_t(x,y)>0$ and choose $a=\delta_t(x,y)/t$. Then \eqref{eq:eq} implies
$$\frac{\mu(B_{\sqrt{t}}(y))}{\mu(B_{\sqrt{t}}(x))} \exp\left(-\upbeta \frac{\left(\dist(x,y)+2\sqrt{t}\right)^2}{2t}+\frac{\delta^2_t(x,y)}{2t}\right) \le (\upbeta \upkappa)^2.$$
Thanks to Proposition \ref{prop:doubling}-i), this leads to
$$ \exp\left(-\lambda \frac{\dist(x,y)}{\sqrt{t}}-\upbeta \frac{\left(\dist(x,y)+2\sqrt{t}\right)^2}{2t}+\frac{\delta^2_t(x,y)}{2t}\right) \le C(\upbeta \upkappa)^2.$$
Letting $t \to 0$, we get $$(v(y)-v(x))^2\le \frac{\upbeta}{2} \dist(x,y)^2.$$
Since $v$ is arbitrary,  we finally obtain
$$\dist_\cE(x,y)\le \sqrt{\upbeta/2}\ \dist(x,y).$$

\end{proof}

\subsection{Proof of  Theorem \ref{th:improvedKS}.}
\begin{proof}
We assume that the spaces $\{(X_\alpha,\dist_\alpha=\dist_{\cE_\alpha},\mu_\alpha,o_\alpha,\cE_\alpha)\}_{\alpha \in A}$ are $\mathrm{PI}_{\upkappa,\upgamma}(R)$
 Dirichlet spaces and that for any $\alpha$,
 \begin {equation}\label{noncolla}
 \eta^{-1}\le \mu_\alpha\left(B_R(o_\alpha)\right)\le \eta.
  \end{equation}
 The existence of $(X,\dist,\mu,o)$ and a subsequence $B \subset A$ such that 
 $$(X_\beta,\dist_\beta,\mu_\beta,o_\beta)\stackrel{pmGH}{\longrightarrow} (X,\dist,\mu,o)$$  follow from Proposition \ref{precompmGH}. Moreover $(X,\dist)$ is complete and geodesic, and $(X,\dist,\mu)$ is $\upkappa$-doubling at scale $R$.
 
Furthermore, Proposition \ref{prop:HKPI} ensures that any $\cE_\beta$ admits a stochastically complete heat kernel $H_\beta$ satisfying the Gaussian bounds \eqref{eq:gaussian} and the estimate \eqref{eq:Rem(c)} with constants $\upbeta$, $\upalpha$ and $A$ depending only on $\upkappa$ and $\upgamma$.  Let $t\in \left(0,R^2\right)$ and $\rho>0$. By Proposition \ref{prop:doubling}-i), we get that for any $x,y\in B_\rho(o_\beta)$,
$$ H_\beta(t,x,y)\le \frac{C e^{\lambda\frac{\rho}{\sqrt{t}}}}{\mu_\beta\left(B_{\sqrt{t}}(o_\beta)\right)} \, ,$$
from which the doubling condition and the non-collapsing condition \eqref{noncolla} yield
the uniform estimate :
\begin{equation}\label{HeatL}
H_\beta(t,x,y)\le C \left(\frac{R}{\sqrt{t}}\right)^\nu \exp\left(\lambda\frac{\rho}{\sqrt{t}}\right)\, \eta,\end{equation}
where $\nu,C,\lambda$ depend only on $\upkappa,\upgamma$.
Hence for any $t\in \left(0,R^2\right)$ and $\rho>0$, there is a  constant $\Lambda$ depending only on $t,\rho,\upkappa,\upgamma,R,\eta$ such that for $x,x',y,y'\in B_\rho(o_\beta)$, we get the H\"older estimate:
\begin{equation}\label{eq:localHölder} |H_\beta(t,x,y)-H_\beta(t,x',y')| \le \Lambda \min\left\{ t^{\frac{\upalpha}{2}}\,;\,[\dist(x,x')+\dist(y,y')]^\upalpha\right\}.\end{equation}
Thanks to this local Hölder continuity estimate and the uniform estimate \eqref{HeatL}, the Arzelà-Ascoli  theorem with respect to pGH convergence (see e.g.~\cite[Prop.~27.20]{Villani}) implies that, up to extracting another subsequence, the functions $H_\beta(t,\cdot,\cdot)$ converge uniformly on compact sets to some function $H(t,\cdot,\cdot)\in \cC(X\times X)$, where $t>0$ is fixed from now on. A priori this subsequence may depend on $t$, but for the moment $t\in \left(0,R^2\right)$ is fixed.

Let $S\colon L_c^2(X,\mu)\rightarrow L_{loc}^\infty(X,\mu) $ be the integral operator on defined by setting
\[
Su(x):=\int_X u(y)H(t,x,y)\di \mu(y)
\]
for any $u \in L_c^2(X,\mu)$ and $x \in X$.

We claim that $S$ has a bounded linear extension $S\colon L^2(X,\mu)\rightarrow L^2(X,\mu)$.
Firstly, thanks to the uniform convergence on compact sets $H_\beta(t,\cdot,\cdot)\to H(t,\cdot,\cdot)$, the symmetry with respect to the two space variables of $H_\beta$ transfers to $H$. Moreover, $\mathrm{PI}$ Dirichlet spaces are stochastically complete hence for any $x\in X_\beta$:
$$\int_{X_\beta} H_\beta(t,x,y)\di \mu_\beta(y)=1.$$
 Using the uniform Gaussian estimate \eqref{eq:gaussian} and Proposition \ref{prop:convdom}, we have similarly $$\int_{X} H(t,x,y)\di \mu(y)=1,$$ for any $x\in X$. The Schur test implies that for any $p\in [0,+\infty]$, $S$ extends to a bounded operator $S\colon L^p(X,\mu)\rightarrow L^p(X,\mu)$  with operator norm satisfying:
\begin{equation}\label{contract} \left\|S\right\|_{L^p\to L^p}\le 1.\end{equation}
 The symmetry with respect to the two space variables of $H$ implies that $$S\colon L^2(X,\mu)\rightarrow L^2(X,\mu)$$ is self-adjoint.  Hence there exists a non-negative self-adjoint operator $L$ with dense domain $\cD(L) \subset L^2(X,\mu)$ such that $$S=e^{-tL}.$$
 Moreover we have: $f\ge 0\Rightarrow Sf\ge 0$.  
 
Let us show now the strong convergence of bounded operators 
\begin{equation}\label{eq:foronet}
e^{-tL_\beta} \to e^{-tL}.
\end{equation}
The operator are all self-adjoint hence it is enough to show the weak convergence of bounded operators and
this amounts to showing that if $u_\beta\stackrel{L^2}{\weakto}u$, then  $e^{-tL_\beta}u_\beta \stackrel{L^2}{\weakto} e^{-tL}u$. Note that $\sup_\beta \|u_\beta\|_{L^2}<+\infty$. Since the operators $e^{-tL_\beta}$ have all an operator norm less than  $1$, then $$\sup_\beta\|e^{-tL_\beta}u_\beta\|_{L^2} <+\infty.$$ 
In particular, the functions $e^{-tL_\beta}u_\beta$ are bounded in $L^2$. Now take $x_\beta \to x$. The uniform Gaussian estimate \eqref{eq:gaussian} and Proposition \ref{prop:convdom} ensures that the functions $f_\beta=H_\beta(t,x_\beta,\cdot)$ converge strongly in $L^2$ to the functions $f=H(t,x,\cdot)$. Then
$$\langle f_\beta,u_\beta\rangle_{L^2(X_\beta,\di\mu_\beta)} \to \langle f,u\rangle_{L^2(X,\di\mu)},$$ that is to say
\[
e^{-tL_\beta}u_\beta(x_\beta) = \int_{X_\beta} H_\beta(t,x_\beta,y) u _\beta(y) \di \mu_\beta(y)\to\int_{X} H(t,x,y) u (y) \di \mu(y) = e^{-tL}u(x).
\]
The same argument can be used with $f_{\beta,r}(y)=\int_{B_r(x_\beta)}H_\beta(t,z,y)\di\mu_\beta(z)$ and $f_{r}(y)=\int_{B_r(x)}H(t,z,y)\di\mu(z)$ for any $r>0$, hence Lemma \ref{lem:criterion} eventually implies
$$e^{-tL_\beta}u_\beta\stackrel{L^2}{\weakto} e^{-tL}u.$$

Using now \cite[Theorem 2.1]{KuwaeShioya}, we get that for any $\tau>0$,  
the strong convergence of bounded operators 
$$
e^{-\tau L_\beta} \to e^{-\tau L}.$$
But the above argumentation shows that for any $\tau\in (0,R^2)$, the function $H_\beta(\tau,\cdot,\cdot)$ has a unique limit (for the uniform convergence on compact set of $X\times X$) and this limit is the Schwartz kernel of the operator $e^{-\tau L}$. Moreover
$$0\le f\le 1\Longrightarrow 0\le e^{-\tau L}f\le 1,$$ hence the quadratic form \[
\cE(u):= \int_X (Lu)u \di \mu
\]
 define a Dirichlet form $\cE$ on $(X,\dist,\mu)$.

From Proposition \ref{prop:equivmosco}, the strong convergence \eqref{eq:foronet} implies the Mosco convergence $\cE_\beta \to \cE$. As a consequence, the functions $H_\beta : (0,R^2] \times X_\beta \times X_\beta$ uniformly converge on compact sets to the heat kernel $\tilde{H}$ of $\cE$ restricted to $(0,R^2] \times X_\beta \times X_\beta$. Then the dominated convergence theorem ensures that $\tilde{H}$ satisfy the assumptions $(a)$, $(b)$ and $(c)$ in Theorem \ref{propo:hk2DS}, thus $(X,\dist,\mu,\cE)$ is regular, strongly local, and with intrinsic distance $\dist_\cE$ bi-Lipschitz equivalent to $\dist$ so that $(X,\dist,\mu,\cE)$ is a $\mathrm{PI}_{\upkappa,\upgamma'}(R)$ Dirichlet space.
 
 It remains to show that  $\dist\le \dist_\cE$.  According to \eqref{eq:optimalgaussian}, the heat kernel of $\cE_\beta$ satisfies the uniform upper Gaussian estimate
$$H_\beta(t,x,y)\le \frac{C}{\mu_\beta (B_R(x))} \frac{R^\nu}{t^{\frac \nu 2}} \left(1+\frac{\dist_{\cE_\beta}^2(x,y)}{t}\right)^{\nu+1} \ \exp\left(-\frac{\dist_{\cE_\beta}^2(x,y)}{4t}\right)$$ valid for any $x,y\in X_\beta$ and $t\in \left(0,R^2\right)$. By uniform convergence on compact set $H_\beta\to H$ and $\dist_{\cE_\beta}\to \dist$, we get the same estimate for the heat kernel of $\cE$:
$$H(t,x,y)\le \frac{C}{\mu (B_R(x))} \frac{R^\nu}{t^{\frac \nu 2}} \left(1+\frac{\dist^2(x,y)}{t}\right)^{\nu+1} \ \exp\left(-\frac{\dist^2(x,y)}{4t}\right).$$
From there it is easy to conclude that $\dist\le \dist_\cE$ thanks to Varadhan's formula \eqref{eq:varadhan}.
\end{proof}

\subsection{Further convergence results} 

In this last subsection we assume that $(X,\dist,\mu,\cE,o)$ is a $\mathrm{PI}_{\upkappa,\upgamma}(R)$ Dirichlet space that is a pointed Mosco-Gromov-Hausdorff limit of a sequence $\left\{(X_\alpha,\dist_\alpha,\mu_\alpha,\cE_\alpha,o_\alpha)\right\}_\alpha$ of  $\mathrm{PI}_{\upkappa,\upgamma}(R)$ Dirichlet spaces, and we use the notations $\{\eps_\alpha\}$, $\{R_\alpha\}$, $\{\Phi_\alpha\}$ of Characterization \ref{chara}.

We begin with a technical result.

\begin{prop}\label{prop:compact}
Let $\{u_\alpha\}$ be such that $u_\alpha \in L^2(B_r(o_\alpha),\mu_\alpha)$ for any $\alpha$, for some $r>0$. Assume that:
\begin{enumerate}
\item there exists $u \in L^2(B_r(o),\mu)$ such that $u_\alpha \stackrel{L^2}{\weakto} u$
\item $\sup_\alpha \int_{B_r(o_\alpha)} \di \Gamma_{\alpha}(u_\alpha) < +\infty$.
\end{enumerate}
Then
\begin{equation}\label{eq:compact}
\lim_\alpha \int_{B_r(o_\alpha)} u_\alpha^2\di\mu_\alpha= \int_{B_r(o) }u^2\di\mu.
\end{equation}
\end{prop}
\begin{proof}We first prove that for any $s<r$,
\begin{equation}\label{eq:cvL2s}\lim_\alpha \int_{B_s(o_\alpha)} u_\alpha^2\di\mu_\alpha= \int_{B_s(o) }u^2\di\mu.\end{equation}
For $\eps<(r-s)/4$, we introduce $$u_{\alpha,\eps}(x)=\int_{B_\eps(x)} u_\alpha\di\mu_\alpha.$$
The Poincaré inequality implies the Pseudo Poincaré inequality \cite[Lemme in page 301]{CoulhonSC1993}:
$$\left\| u_\alpha-u_{\alpha,\eps}\right\|_{L^2(B_s(o_\alpha))}\le C\eps$$
and
$$\left\| u-u_{\eps}\right\|_{L^2(B_s(o))}\le C\eps,$$
where $C$ depends only on the doubling constant, the Poincaré constant and of $\sup_\alpha \int_{B_r(o_\alpha)} \di \Gamma_{\alpha}(u_\alpha)$.
The H\"older inequality and the doubling property [\ref{prop:doubling}-v] imply that for fixed $\eps>0$, the sequence $\{ u_{\alpha,\eps}\}_\alpha$ is equicontinuous on $B_{(s+r)/2}(o_\alpha)$, hence $ u_{\alpha,\eps}\to u_\eps$ uniformly in $B_{s}(o_\alpha)$. Hence we get the strong convergence in $L^2(B_s(o))$ and the convergence \eqref{eq:cvL2s}.

Since the spaces are uniformly PI, they satisfy a same local Sobolev inequality \cite[Th.~2.6]{sturm1996analysis}  meaning that there exists $C>0$, $\nu>2$ and $\delta \in (0,1)$ independant on $\alpha$ such that
\[
 \left(\int_{B_{r}(o_\alpha)} |\psi|^{\frac{2\nu}{\nu-2}}\di\mu_\alpha\right)^{1-2/\nu}\le C\left(\int_{B_{r}(o_\alpha)}d\Gamma_\alpha(\psi)+\int_{B_{r}(o_\alpha)} |\psi|^2\di\mu_\alpha\right)
\]
for any $\psi\in \cD(\cE_\alpha)$.
In particular, we get the a priori bound
$$\sup_{\alpha}\|u_\alpha\|_{L^{\frac{2\nu}{\nu-2}}B_r(o_\alpha)}\le C.$$
With Hölder's inequality and the doubling property [\ref{prop:doubling}-v], this yields
\begin{align*}
\left| \int_{B_s(o_\alpha)} u_\alpha^2\di\mu_\alpha- \int_{B_r(o_\alpha)} u_\alpha^2\di\mu_\alpha\right| & = \left| \int_{B_r(o_\alpha)} u_\alpha^2(\mathbf{1}_{B_s(o_\alpha)}-1)\di\mu_\alpha\right|\\
& \le  \left(\int_{B_{r}(o_\alpha)} |\psi|^{\frac{2\nu}{\nu-2}}\di\mu_\alpha\right)^{1-\frac{2}{\nu}} \mu(B_r(o_\alpha)\backslash B_s(o_\alpha))^{\frac{2}{\nu}}\\
& \le C (r-s)^{2\delta/\nu}
\end{align*}
which justifies the intervention
$$\lim_{s\to r} \lim_\alpha \int_{B_s(o_\alpha)} u_\alpha^2\di\mu_\alpha= \lim_\alpha \lim_{s\to r}\int_{B_s(o_\alpha)} u_\alpha^2\di\mu_\alpha.$$
\end{proof}
\subsubsection{Convergence of the core $\cC_c\cap \cD(\cE)$}
The next result indicates that in a certain sense the space $\cC_c(X)\cap \cD(\cE)$ is the limit of the spaces $\cC_c(X_\alpha)\cap \cD(\cE_\alpha)$.
\begin{prop}\label{prop:cvCore} Let $\varphi\in \cC_c(X)\cap \cD(\cE)$. Then there exists $\{\varphi_\alpha\}$, with $\varphi_\alpha \in \cC_c(X_\alpha)\cap \cD(\cE_\alpha)$ for any $\alpha $, such that $$\varphi_\alpha\stackrel{\cC_c}{\to} \varphi\text{ and } \varphi_\alpha\stackrel{\mE}{\to} \varphi.$$
Moreover if $\varphi$ is non negative then each $\varphi_\alpha$ can be chosen to be also non negative.
\end{prop}

\begin{proof} \textbf{Step 1.} We construct $\psi_\alpha\in\cC_0(X_\alpha)\cap \cD(\cE_\alpha)$ such that $\psi_\alpha\to \varphi$ uniformly on compact sets and such that $\psi_\alpha\stackrel{\mE}{\to} \varphi.$

Proposition \ref{prop:approx} allows us to build $f_\alpha\in \cC_c(X_\alpha)$ such that  $f_\alpha\stackrel{\cC_c}{\to} \varphi$. Moreover we know that the sequence $\{f_\alpha\}$ is uniformly equicontinuous: there is $\omega\colon \R_+\rightarrow \R_+$ non decreasing, bounded and satisfying $ \omega(\delta) \to0$ when $\delta \to 0$ such that for any $\alpha$ 
$$\left| f_\alpha(x)- f_\alpha(y)\right|\le \omega\left(\dist_\alpha(x,y)\right)$$
for any $\alpha$  and any $x, y \in X_\alpha$. In addition, if it turns out that $\varphi$ is non negative, so is $f_\alpha$.
As $f_\alpha\stackrel{\cC_c}{\to} \varphi$, we also have $f_\alpha\stackrel{L^2}{\to} \varphi$ and for any $\eps>0$:
$P^\alpha_\eps f_\alpha\stackrel{\mE}{\to} P_\eps\varphi$. And $\varphi$ being in $\cD(\cE)$ we get that 
$$\cD(\cE)-\lim_{\eps\to 0} P_\eps\varphi=\varphi.$$

Let us show now that if $\eps_\alpha\downarrow 0$ then we have $P^\alpha_{\eps_\alpha} f_\alpha\to \varphi$ uniformly.
To do so, it is sufficient to demonstrate that
$$\lim_{\eps\to 0} \sup_\alpha \left\| P^\alpha_{\eps} f_\alpha-f_\alpha\right\|_{L^\infty}=0.$$
Using the stochastic completeness, we know that for any $x\in X_\alpha$:
$$P^\alpha_{\eps} f_\alpha(x)-f_\alpha(x)=\int_{X_\alpha} H_\alpha(\eps,x,y)\left(f_\alpha(y)-f_\alpha(x)\right)\,\di\mu_\alpha(y).$$

Hence for any $\kappa>0$, we have

\begin{align*}
\left| P^\alpha_{\eps} f_\alpha(x)-f_\alpha(x)\right|&\le \int_{X_\alpha} H_\alpha(\eps,x,y)\ \omega\left(\dist_\alpha(x,y)\right)\di\mu_\alpha(y)\\
&\le  \omega\left(\kappa\sqrt{\eps}\right)+\|\omega\|_{L^\infty} \int_{X_\alpha\setminus B_{\kappa\sqrt{\eps}}(x)} H_\alpha(\eps,x,y)\,\di\mu_\alpha(y)\\
&\le \omega\left(\kappa\sqrt{\eps}\right)+\frac{\|\omega\|_{L^\infty} \upgamma}{\mu\left( B_{\sqrt{\eps}}(x)\right)}\int_{X_\alpha\setminus B_{\kappa\sqrt{\eps}}(x)} e^{-\frac{\dist_\alpha^2(x,y)}{\upgamma \eps}}\,\di\mu_\alpha(y)\\
&\omega\left(\kappa\sqrt{\eps}\right)+\frac{\|\omega\|_{L^\infty} \upgamma}{\mu\left( B_{\sqrt{\eps}}(x)\right)}\int_{\kappa \sqrt{\eps}}^\infty e^{-\frac{r^2}{\upgamma \eps}}\, \frac{2r}{\upgamma \eps}\, \mu\left( B_{r}(x)\right)dr
\end{align*}
We use the doubling condition [\ref{prop:doubling}-iii)] to deduce that if $r\ge \kappa\sqrt{\eps}>R>\sqrt{\eps}$ then
$$\mu\left( B_{r}(x)\right)\le e^{\lambda \frac{r}{\sqrt{\eps}}} \mu\left( B_{\sqrt{\eps}}(x)\right),$$
Hence if $\kappa\sqrt{\eps}>R>\sqrt{\eps}$ then
\begin{align*}\left| P^\alpha_{\eps} f_\alpha(x)-f_\alpha(x)\right|&\le \omega\left(\kappa\sqrt{\eps}\right)+\|\omega\|_{L^\infty} \int_{\kappa \sqrt{\eps}}^\infty e^{-\frac{r^2}{\upgamma \eps}+\lambda \frac{r}{\sqrt{\eps}} }\, \frac{2r}{ \eps}\, dr\\
&\le \omega\left(\kappa\sqrt{\eps}\right)+\|\omega\|_{L^\infty} \int_{\kappa }^\infty e^{-\frac{r^2}{\upgamma }+\lambda r }\, 2r\, dr\\
 &\le \omega\left(\kappa\sqrt{\eps}\right)+C(\lambda,\upgamma) e^{-\frac{\kappa^2}{2\upgamma}} \|\omega\|_{L^\infty} .\end{align*}
 We then choose $\kappa=\eps^{-\frac 14}$ and we get for $\eps$ small enough:
 $$\left\| P^\alpha_{\eps} f_\alpha-f_\alpha\right\|_{L^\infty}\le  \omega\left(\eps^{\frac 14}\right)+C(\lambda,\upgamma) e^{-\frac{1}{2\upgamma\sqrt{\eps}}} \|\omega\|_{L^\infty}.$$
 \begin{rem}\label{rem:decroissance} The same estimate leads to the following decay estimate for $P^\alpha_{\eps} f_\alpha$. Assume that $R>0$ and $L$ are such that $\supp f_\alpha\subset B_{R}(o_\alpha)$ and that $\| f_\alpha\|_{L^\infty}\le L$.
 Then for $x\in X_\alpha\setminus B_{2R}(o_\alpha)$:
 $$\left|P^\alpha_{\eps} f_\alpha(x)\right|\le C L\,  e^{-\frac{\dist_\alpha^2(o_\alpha,x)}{4\upgamma\eps}+\lambda\frac{R}{\sqrt{\eps}} }.$$
 \end{rem}
 
 To build $\psi_\alpha$ we use U. Mosco's argument for the proof of the implication $2)\Rightarrow 1)$ in Proposition \ref{prop:equivmosco}.
 We find a decreasing sequence $\eta_\ell\downarrow 0$ and a increasing sequence $\alpha_\ell\uparrow+\infty$ such that 
 $$0<\eps\le \eta_\ell\Longrightarrow  \left| \| P_\eps \varphi\|_{L^2}^2- \|  \varphi\|_{L^2}^2 \right|+  \left| \cE\left( P_\eps \varphi\right) -\cE\left(  \varphi\right) \right|\le 2^{-\ell}$$
 $$\alpha\ge \alpha_\ell  \Longrightarrow  \left| \| P^\alpha_{\eta_\ell} f_\alpha\|_{L^2}^2- \|  P_{\eta_\ell}\varphi\|_{L^2}^2 \right|+  \left| \cE_\alpha\left( P^\alpha_{\eta_\ell} f_\alpha \right) -\cE\left(  P_{\eta_\ell} \varphi\right) \right|\le 2^{-\ell}$$

Then if $\alpha \in [\alpha_\ell,\alpha_{\ell+1})$, we define $\eps_\alpha=\eta_\ell$  and $\delta_\alpha=2^{1-\ell}$ and we let
$$\psi_\alpha=P^\alpha_{\eps_\alpha} f_\alpha.$$
Then we have $\lim_\alpha \delta_\alpha=0$ and  $\psi_\alpha\to \varphi$ uniformly and
$$ \left| \| \psi_\alpha\|_{L^2}^2- \|  \varphi\|_{L^2}^2 \right|+  \left| \cE_\alpha\left(\psi_\alpha \right) -\cE\left(  \varphi\right) \right|\le \delta_\alpha.$$
We necessarily have $\psi_\alpha\stackrel{L^2}{\weakto}\varphi$ and the above estimate implies the strong convergence  
$\psi_\alpha\stackrel{\mE}{\to}\varphi$.

\textbf{Step 2.} We modify each $\psi_\alpha$ with appropriate cut-off functions. 
Let $R>0$ such that that for any $\alpha$:  $\supp f_\alpha\subset B_R(o_\alpha)$ and 
 $\supp \varphi\subset B_R(o)$ and let $L$ such that for any $\alpha$
$$\|f_\alpha\|_{L^\infty}\le L.$$
We let $\chi_\alpha\colon X_\alpha\rightarrow \R$ be defined by
$$\chi_\alpha(x)=\xi(\dist_\alpha(o_\alpha,x)/2R).$$
where $\xi$ is defined by \eqref{coupure}.
And we let 
$$\varphi_\alpha=\chi_\alpha\,\psi_\alpha.$$

It is easy to check that $\varphi_\alpha\stackrel{\cC_c}{\to} \varphi$. In order to verify that $\varphi_\alpha\stackrel{\mE}{\to}\varphi$, we need to check that $$\lim_{\alpha\to\infty} \cE_\alpha\left( \,(1-\chi_\alpha)\psi_\alpha\right)=0.$$
The chain rule implies that 
\begin{equation}\label{{decomp1}} \cE_\alpha\left( \,(1-\chi_\alpha)\psi_\alpha\right)=\cE_\alpha\left(\psi_\alpha,(1-\chi_\alpha)^2\psi_\alpha\right)+\int_{X_\alpha} \psi_\alpha^2\di\Gamma_\alpha(\chi_\alpha).\end{equation}
We have $\psi_\alpha\stackrel{\mE}{\to}\varphi$ and  $(1-\chi_\alpha)^2\psi_\alpha\stackrel{\mE}{\weakto} 0$ hence the first term in the right hand side of \eqref{{decomp1}} tends to $0$ when $\alpha\to\infty$.
The function $\chi_\alpha$ are uniformly $(1/2r)$-Lipschitz hence by \eqref{EnergyLip} there is a constant $C$ independant of $\alpha$ such that 
$$\int_{X_\alpha} \psi_\alpha^2\di\Gamma_\alpha(\chi_\alpha)\le C \mu_\alpha\left(B_{4R}(o_\alpha)\right) \sup_{B_{4R}(o_\alpha) \setminus B_{2R}(o_\alpha)}| \psi_\alpha|^2 .$$
Using the Remark \eqref{rem:decroissance},we can conclude that 
$$\lim_{\alpha\to\infty}\int_{X_\alpha} \psi_\alpha^2\di\Gamma_\alpha(\chi_\alpha)=0.$$
\end{proof}
\subsubsection{Energy convergence and convergence of the \textit{carré du champ}}We can now easily deduced the following convergence result for the \textit{carré du champ} under convergence in energy.
\begin{prop}\label{prop:cvenergyI} Assume that $\varphi \in \cC(X) \cap \cD(\cE)$, $u \in \cD_{loc}(\cE)$, and $\varphi_\alpha \in \cC(X_\alpha) \cap \cD(\cE_\alpha)$, $u_\alpha \in \cD_{loc}(\cE_\alpha)$ for any $\alpha$, are such that $\varphi_\alpha\stackrel{\cC_c}{\to} \varphi$, $\varphi_\alpha\stackrel{\mE}{\to} \varphi$, $u_\alpha\stackrel{\mE_{loc}}{\to} u$ and $\sup_\alpha \|u_\alpha\|_{L^\infty}<\infty$. Then
\begin{equation}\label{eq:cvenergyI}
\lim_{\alpha\to\infty} \int_{X_\alpha} \varphi_\alpha\di\Gamma(u_\alpha)=\int_{X} \varphi\di\Gamma(u).
\end{equation}
Moreover if for each $\rho>0$ there is some $p>1$ such that 
$$\sup_{\alpha}\int_{B_\rho(o_\alpha)}\left|\frac{\di\Gamma(u_\alpha)}{\di\mu_\alpha}\right|^p\di\mu_\alpha<\infty$$ then for each $\rho>0:$
\begin{equation}\label{eq:cvenergyII}
\lim_{\alpha\to\infty} \int_{B_\rho(o_\alpha)}\di\Gamma(u_\alpha)=\int_{B_\rho(o)} \di\Gamma(u).
\end{equation}
\end{prop}
\begin{proof} We give a proof only in case $u_\alpha\stackrel{\mE}{\to} u$, the demonstration in the stated case being identical up to a few immediate but cumbersome justifications.

To prove \eqref{eq:cvenergyI} we use the definition of the carré du champ,
\begin{equation}\label{defcarré} \int_{X_\alpha} \varphi_\alpha\di\Gamma(u_\alpha)=\cE_\alpha(\varphi_\alpha \,u_\alpha,u_\alpha)-\frac 12 \cE_\alpha\left(\varphi_\alpha,u_\alpha^2\right),\end{equation}
together with the following observation: the chain rule implies that the sequences $\{\varphi_\alpha u_\alpha\}_\alpha$ and $\{u_\alpha^2\}_\alpha$ are bounded in energy and thus have weak limit in $\cE$. However when $\psi_\alpha\stackrel{\mE}{\to} \psi$ we get $\psi_\alpha\varphi_\alpha\stackrel{\mE}{\to} \psi\varphi$ and then
$\int_{X_\alpha}\psi_\alpha\varphi_\alpha u_\alpha\di\mu_\alpha=\int_{X}\psi\varphi u\di\mu$, so that 
$$\varphi_\alpha u_\alpha\stackrel{\mE}{\weakto} \varphi u.$$
Moreover, when $\psi_\alpha\stackrel{\mE}{\to} \psi$ we get $\psi_\alpha u_\alpha\stackrel{\L^2}{\to} \psi u$
 and then $\int_{X_\alpha}\psi_\alpha u_\alpha^2 \di\mu_\alpha=\int_{X}\psi u^2\di\mu$, so that 
$$ u^2_\alpha\stackrel{\mE}{\weakto} u^2.$$
Thus \eqref{defcarré} converges to $\int_X \varphi \di \Gamma(u)$.

To prove \eqref{eq:cvenergyII}, take $\eps>0$.  Acting as in the proof of Proposition \ref{prop:compact}, with H\"older's inequality and the doubling condition  [\ref{prop:doubling}-v] we can find $\tau\in (0,\rho)$ such that for any $\alpha$,
$$\left|\int_{B_{\rho-\tau}(o_\alpha)}\di\Gamma(u_\alpha) -\int_{B_\rho(o_\alpha)}\di\Gamma(u_\alpha)\right|\le \frac{\eps}{3}.$$
Moreover, by regularity of the Radon measure $\Gamma(u)$, we an assum that
$$\left|\int_{B_{\rho-\tau}(o)}\di\Gamma(u) -\int_{B_\rho(o)}\di\Gamma(u)\right|\le \frac{\eps}{3}.$$
We set
$$\varphi(x):=\begin{cases} 1&\text{ if } x\in B_{\rho-\tau}(o),\\
(2\rho-\tau-2\dist(o,x))/\tau&\text{ if } x\in B_{\rho-\tau/2}(o)\setminus B_{\rho-\tau}(o),\\
0&\text{ if } x\not\in B_{\rho-\tau/2}(o).\end{cases}$$
We have
\begin{equation}\label{eq:est1}\left|\int_{B_{\rho}(o)}\di\Gamma(u) -\int_{X}\varphi\di\Gamma(u)\right|\le \frac \eps 3\, \cdot\end{equation}
Thanks to Proposition \ref{prop:cvCore}, we can choose $\varphi_\alpha\in \cC_c(X_\alpha)\cap \cD(\cE_\alpha)$ non-negative for any $\alpha$ such that $\varphi_\alpha\stackrel{\cC_c}{\to} \varphi$ and $\varphi_\alpha\stackrel{\mE}{\to} \varphi$. Then there is some sequence $\delta_\alpha\downarrow 0$ such that 
\begin{itemize}
\item $|\varphi_\alpha-1|\le \delta_\alpha$ on $B_{\rho-\tau}(o_\alpha)$,
\item $\varphi_\alpha\le 1+\delta_\alpha$,
\item  $\varphi_\alpha\le \delta_\alpha$ outside $B_{\rho}(o_\alpha)$.
\end{itemize}
We easily get 
\begin{equation}\label{eq:est2}\left|\int_{B_{\rho}(o_\alpha)}\di\Gamma(u_\alpha) -\int_{X_\alpha}\varphi_\alpha\di\Gamma(u_\alpha)\right|\le \delta_\alpha\cE_\alpha(u_\alpha)+(1+\delta_\alpha)\frac \eps 3.\end{equation}
Using \eqref{eq:est1} and  \eqref{eq:est2}, we find $\underline{\alpha}$ such that 
$$\alpha\ge \underline{\alpha}\Longrightarrow \left|\int_{B_{\rho}(o_\alpha)}\di\Gamma(u_\alpha) -\int_{B_{\rho}(o)}\di\Gamma(u) \right| <\eps.$$
\end{proof}
The above argument also implies that the lower semi-continuity of the Carré du champ under weak convergence in energy.

\begin{prop}\label{prop:convfaibleenergie}
In the setting of the previous proposition, we also have $$\liminf_{\alpha\to\infty} \int_{X_\alpha} \varphi_\alpha\di\Gamma(u_\alpha)\ge\int_{X} \varphi\di\Gamma(u),$$
and for any $\rho>0$,
$$\liminf_{\alpha\to\infty}\int_{B_\rho(o_\alpha)}\di\Gamma(u_\alpha)\ge \int_{B_\rho(o)} \di\Gamma(u).$$
\end{prop}
\begin{proof} We also only prove the result under the stronger hypothesis $u_\alpha\stackrel{\mE}{\weakto}  u$.

Once we have shown that when  $v_\alpha\stackrel{\mE}{\to} v$  and $\sup_\alpha \|v_\alpha\|_{L^\infty}<\infty$, then
$$\lim_{\alpha\to\infty} \int_{X_\alpha} \varphi_\alpha\di\Gamma(u_\alpha,v_\alpha)=\int_{X} \varphi\di\Gamma(u,v),$$
it is classical to conclude.
We use the formula
$$2\int_{X_\alpha} \varphi_\alpha\di\Gamma(u_\alpha,v_\alpha)=\cE_\alpha(\varphi_\alpha \,u_\alpha,v_\alpha)+\cE_\alpha(\varphi_\alpha \,v_\alpha,u_\alpha)-\cE_\alpha\left(\varphi_\alpha,u_\alpha v_\alpha\right).$$ 

The above arguments show that the first and the last term in the right hand side are converging to the right data. For the second term, we need to show that $\varphi_\alpha \,v_\alpha\stackrel{\mE}{\to}  \varphi \,v$. We have strong convergence in $L^2$ and 
$$\cE_\alpha(\varphi_\alpha \,v_\alpha)=\cE_\alpha\left(\varphi_\alpha ,\,v^2_\alpha\varphi_\alpha\right)+\int_{X_\alpha} \varphi_\alpha^2\di\Gamma(v_\alpha).$$
Using Proposition \ref{prop:cvenergyI}, we deduce that 
$$\lim_{\alpha} \int_{X_\alpha} \varphi_\alpha^2\di\Gamma(v_\alpha)=\int_{X} \varphi^2\di\Gamma(v),$$ and as $\varphi_\alpha \stackrel{\mE}{\to} \varphi$ and $v^2_\alpha\varphi_\alpha\stackrel{\mE}{\weakto}  
v^2_\alpha\varphi$, so that $\lim_{\alpha}\cE_\alpha\left(\varphi_\alpha ,\,v^2_\alpha\varphi_\alpha\right)=\cE\left(\varphi ,\,v^2\varphi\right)$ so that 
$$\lim_{\alpha}\cE_\alpha(\varphi_\alpha \,v_\alpha)=\cE(\varphi \,v)$$ and $\varphi_\alpha \,v_\alpha\stackrel{\mE}{\to}  \varphi \,v$ and $\lim_{\alpha}\cE_\alpha(\varphi_\alpha \,v_\alpha,u_\alpha)=\cE_\alpha(\varphi \,v,u).$
\end{proof}

\begin{rem}\label{rem:localisationcv} The above result can be localized, meaning that if we assume that functions $u_\alpha\in \cD(B_\rho(o_\alpha),\cE_\alpha)$ satisfy
\begin{itemize}
\item $u_\alpha\stackrel{\mE}{\to} u$,
\item $\sup_{\alpha} \|u_\alpha\|_{L^\infty}<\infty$ ,
\item $\sup_{\alpha}\int_{B_\rho(o_\alpha)}\left|\frac{\di\Gamma(u_\alpha)}{\di\mu_\alpha}\right|^p\di\mu_\alpha<\infty$ for some $p>1$, 
\end{itemize} then
$$\lim_{\alpha\to\infty} \int_{B_\rho(o_\alpha)}\di\Gamma(u_\alpha)=\int_{B_\rho(o)} \di\Gamma(u).$$
\end{rem}
Using Proposition \ref{prop:compact}, we also get the following strong convergence result for the energy measure density.
\begin{prop}\label{prop:strongW12} Assume that functions $u_\alpha\in \cD(B_\rho(o_\alpha),\cE_\alpha)$ satisfy
\begin{itemize}
\item $u_\alpha\stackrel{\mE}{\to} u$,
\item $\sup_{\alpha} \|u_\alpha\|_{L^\infty}<\infty$ ,
\item $\sup_\alpha \int_{B_\rho(o_\alpha)}\di\Gamma(\rho_\alpha)<\infty$,  where $\rho_\alpha=\left|\frac{\di\Gamma(u_\alpha)}{\di\mu_\alpha}\right|^{\frac 12}$ for any $\alpha$,
\end{itemize}
then $$\left|\frac{\di\Gamma(u_\alpha)}{\di\mu_\alpha}\right|^{\frac 12}\stackrel{L^2}{\to} \left|\frac{\di\Gamma(u)}{\di\mu}\right|^{\frac 12}.$$
\end{prop} 
\begin{proof}
Indeed Proposition \ref{prop:compact} implies that up to extracting a  subsequence we can assume that $\rho_\alpha\stackrel{L^2}{\to} f$ and we want to show that $f=\left|\frac{\di\Gamma(u)}{\di\mu}\right|^{\frac 12}$b$\mu$a.e.~on $B_\rho(o)$. Following the proof of  Proposition \ref{prop:compact} (using the Sobolev inequality), we have some $p>1$ such that $\sup_\alpha \|\rho_\alpha\|_{L^{2p}}<\infty$, hence Remark \ref{rem:localisationcv} implies that if $x_\beta\in B_\rho(o_\beta)\to x\in B_\rho(o)$, for any $r>0$ such that $r+\dist(o,x)<\rho$ we have
$$\lim_{\alpha\to\infty} \int_{B_r(x_\alpha)}\di\Gamma(u_\alpha)=\int_{B_r(x)} \di\Gamma(u).$$
But the strong $L^2$-convergence also yields that 
$$\lim_{\alpha\to\infty} \int_{B_r(x_\alpha)}\di\Gamma(u_\alpha)=\lim_{\alpha\to\infty} \int_{B_r(x_\alpha)}\rho^2_\alpha\di\mu_\alpha= \int_{B_r(x)}f^2\di\mu.$$
Hence for any $x\in B_\rho(o)$ and  $r>0$ such that $r+\dist(o,x)<\rho$ :
$$\int_{B_r(x)}f^2\di\mu=\int_{B_r(x)} \di\Gamma(u).$$
By Lebesgue differentiation theorem (that hold true on any doubling space), this implies that $f^2=\frac{\di\Gamma(u)}{\di\mu}$ $\mu$-a.e.
\end{proof}

\subsubsection{Convergence of harmonic functions}
\begin{prop}\label{prop:cvW22} Let $u_\alpha\colon B_{\rho}(o_\alpha)\rightarrow \R$ such that 
$$\sup_{\alpha} \left( \|u_\alpha\|_{L^\infty}+  \|L_\alpha u_\alpha\|_{L^2} \right)<\infty.$$
Then there is a sub-sequence $B\subset A$ and $u\colon B_{\rho}(o)\rightarrow \R$ in $\cD_{loc}(B_\rho(o),\cE)$ such that $Lu\in L^2$ and
$$u_\beta\stackrel{L^2}{\to} u\text{ and }L_\beta u_\beta\stackrel{L^2_{loc}}{\weakto} Lu.$$
Moreover, if $\varphi_\beta\in \cC_c( B_{\rho}(o_\beta))\cap \cD(\cE_\beta)$ and $\varphi \in \cC_c( B_{\rho}(o))\cap \cD(\cE)$ are such that $\varphi_\beta\stackrel{\cC_c}{\to} \varphi$ and $\varphi_\beta\stackrel{\mE}{\to} \varphi,$ then
$$\lim_{\beta\to\infty} \int_{X_\beta} \varphi_\beta\di\Gamma(u_\beta)=\int_{X} \varphi\di\Gamma(u).$$\end{prop}
\begin{rem} Looking at the proof of Proposition \ref{prop:cvCore}, we remark that for any $\varphi \in \cC_c( B_{\rho}(o))\cap \cD(\cE)$, we can find a sequence $\varphi_\alpha\in \cC_c( B_{\rho}(o_\alpha))\cap \cD(\cE_\alpha)$ such that $\varphi_\alpha\stackrel{\cC_c}{\to} \varphi$ and $\varphi_\alpha\stackrel{\mE}{\to} \varphi.$
\end{rem}
\begin{proof} We can find a sub-sequence $B\subset A$, $u\in L^2(B_{\rho}(o))\cap L^\infty(B_{\rho}(o))$  and $f\in \in L^2(B_{\rho}(o))$ such that 
$$u_\beta\stackrel{L^2\cap L^4}{\weakto} u\text{ and } f_\beta:=L_\beta u_\beta\weakto f.$$

For any $r<\rho$, we consider the function $\chi_\beta(x)=\xi\left(2\frac{\dist_\beta(o_\beta,x)-\frac{3r-\rho}{2}}{\rho-r}\right)$ where $\xi$ is defined by \eqref{coupure}, this function has the following properties:
\begin{itemize}
\item $\chi_\beta=1$ on $B_r(o_\beta)$,
\item $\chi_\beta=0$ outside $B_{(\rho+r)/2}(o_\beta)$,
\item $\chi_\beta$ is $2/(\rho-r)$-Lipschitz.
\end{itemize} 
We have the estimates
$$
\int_{B_r(o_\beta)} \di\Gamma(u_\beta)\le \int_{B_\rho(o_\beta)} \di\Gamma(\chi_\beta u_\beta)$$ but
\begin{align*}\int_{B_\rho(o_\beta)} \di\Gamma(\chi_\beta u_\beta)&=\int_{B_\rho(o_\beta)}u_\beta^2 \di\Gamma(\chi_\beta)+\cE_\beta(\chi_\beta^2 u_\beta,u_\beta)\\
&=\int_{B_\rho(o_\beta)}u_\beta^2 \di\Gamma(\chi_\beta)+\int_{B_\rho(o_\beta)}u_\beta\chi_\beta^2 f_\beta\di\mu_\beta).\end{align*}
Using  the comparison \eqref{EnergyLip}, we get 
$$
\int_{B_r(o_\beta)} \di\Gamma(u_\beta)\le \frac{4}{(\rho-r)^2} \|u_\beta\|^2_{L^\infty} \mu_\beta\left(B_\rho(o_\beta)\right)+\|u_\beta\|_{L^2}\|f_\beta\|_{L^2}.$$
Hence $u_\beta$ is locally bounded in $\cD_{loc}(\cE_{\alpha})$ and $u\in\cD_{loc}(\cE)$ and
$u_\beta \stackrel{\mE_{loc}}{\weakto} u$ and also $u^2_\beta \stackrel{\mE _{loc}}{\weakto} u^2$.
The formula
$$ \int_{B_\rho(o_\beta)} \varphi_\beta\di\Gamma(u_\beta)=\int_{B_\rho(o_\beta)} \varphi_\beta u_\beta f_\beta\di\mu_\beta-\frac12 \cE_\beta\left(\varphi_\beta,u^2_\beta\right)$$ and the fact that from Proposition \ref{prop:compact}, we have $u_\beta \stackrel{L^2_{loc}}{\to} u$ implies the claimed result of convergence.
\end{proof}

In case where we have a sequence of harmonic functions, this result can be slightly improved.
\begin{prop}\label{prop:cvharmonic} Let $\{h_\alpha\}$ be such that $h_\alpha\colon B_\rho(o_\alpha)\rightarrow \R$ is $L_\alpha$-harmonic for any $\alpha$ and
$$\sup_\alpha \|h_\alpha\|_{L^\infty}<\infty,$$
Then there is a subsequence $B\subset A$ and a harmonic function $h\colon B_\rho(o)\rightarrow \R$ such that 
\begin{enumerate}[i)]
\item $h_\beta\stackrel{L^2}{\to} h$
\item For each $r<\rho$: $\left.h_\beta\right|_{B_r(o_\beta)}\to \left.h\right|_{B_r(o)}$ uniformly.
\item For each $r<\rho$: $\lim_{\beta\to\infty} \int_{B_r(o_\beta)} \di\Gamma(h_\beta)=\int_{B_r(o)} \di\Gamma(h)$.
\end{enumerate}
\end{prop}

\begin{proof}
Proposition \ref{prop:cvW22} implies the existence of a subsequence $B\subset A$  and of a harmonic function $h\colon B_\rho(o)\rightarrow \R$ such that we have the strong convergence in $L^2$.
The uniform convergence follows from the fact that each $(X_\beta,\dist_\beta,\mu_\beta,\cE_\beta)$ satisfies uniform Parabolic/Elliptic Harnack inequality and hence uniform local H\"older estimate for harmonic function (\cite[Lemma 2.3.2]{SaloffCoste} or \cite[Lemma 5.2]{GriTelcs}. In our cases, there is a constant $\theta\in (0,1)$ and $C$ such that for any $\alpha$ and $x,y\in B_r(o_\alpha)$
$$\left|h_\alpha(x)-h_\alpha(y)\right|\le C \left(\frac{\dist_\alpha(x,y)}{\rho-r}\right)^\theta  \|h_\alpha\|_{L^\infty}.$$
The last point is a consequence of a uniform reverse H\"older inequality for the energy density of harmonic function. There is some $p>1$ and some constant $C$ such that if $B\subset X_\alpha$ is a ball of radius $r(B)\le R$ and $f\colon B\rightarrow \R$ is harmonic then
$$\left(\fint_{\frac 12 B} \left|\frac{\di\Gamma(f)}{\di\mu_\alpha}\right|^p\di\mu_\alpha\right)^{\frac 1p}\le C \fint_{ B} \di\Gamma(f).$$
This is explained in \cite[subsection 2.1]{AuscherCoulhon}, it relies on a self-improvement of the $L^2$-Poincaré inequality to a  $L^{2-\eps}$-Poincaré inequality \cite{KeithZhong} and of the Gehring lemma 
\cite[Chapter V]{Giaquinta}.
\end{proof} 

\subsubsection{Approximation of harmonic functions}

Let us conclude with an approximation result for harmonic functions.

\begin{prop}\label{prop:approxharmonic} Let $h\colon B_\rho(o)\rightarrow \R$ be a harmonic function and let $r<\rho$.  Then there exists $\{h_\alpha\}$, with $h_\alpha\colon B_r(o_\alpha)\rightarrow \R$ harmonic for any $\alpha$, such that 
\begin{enumerate}[i)]
\item $h_\alpha\to \left.h\right|_{B_r(o)}$ uniformly on compact sets,
\item $\int_{B_s(o_\alpha)} \di\Gamma(h_\alpha) \to \int_{B_s(o)} \di\Gamma(h)$ for any $s\le r$.
\end{enumerate}
\end{prop}
\begin{proof}
Set $\delta:=(\rho-r)/4$ and $$\xi(x):=\chi\left(1+\frac{\dist(o,x)-(r+2\delta)}{\delta}\right)$$ for any $x \in X$, where $\chi$ is as in \eqref{coupure}.  Set $$\varphi:=\xi h\in \cC_c(X)\cap \cD(\cE).$$ 
Let $\{\varphi_\alpha\}$ be given by Proposition \ref{prop:cvCore}, i.e.~$ \varphi_\alpha \in  \cC_c(X_\alpha)\cap \cD(\cE_\alpha)$ for any $\alpha$ and $\varphi_\alpha\stackrel{\cC_c}{\to} \varphi$,  $\varphi_\alpha\stackrel{\mE}{\to} \varphi.$ For any $\alpha$, let $h_\alpha$ be the harmonic replacement of $\varphi_\alpha$ on $B_{r+\delta}(o_\alpha)$ that is to say $h_\alpha \in \cD(\cE_\alpha)$ is the unique solution of
$$\begin{cases} L_\alpha h_\alpha=0&\text{ on } B_{r+\delta}(o_\alpha),\\
h_\alpha=\varphi_\alpha&\text{ outside } B_{r+\delta}(o_\alpha),\end{cases}$$
which is characterized by 
$$\cE_\alpha(h_\alpha)=\inf\left\{ \cE_\alpha(f)\, :\, f\in \cD(\cE_\alpha)\text{ and } f=\varphi_\alpha\text{ outside } B_{r+\delta}(o_\alpha)\right\}.$$
In particular $\cE_\alpha(h_\alpha)\le \cE_\alpha(\varphi_\alpha)$ for any $\alpha$, hence we can find a subsequence $B\subset A$ and $f\in \cD(\cE)$ such that $h_\beta\stackrel{\mE}{\weakto} f$. The lower semi-continuity of the energy implies
$$\cE(f)\le \cE(\varphi),$$
and we have $f=\varphi$ on $X\backslash B_{r+\delta}(o)$.
However, $\varphi$ is its own harmonic replacement on $B_{r+\delta}(o)$,  hence the variational characterization of the harmonic replacement implies $f=\varphi$. Thus $h_\alpha\stackrel{\mE}{\weakto} \varphi$ and the result is then a consequence of Proposition \ref{prop:cvharmonic}.
\end{proof}

\bibliographystyle{alpha} 
\bibliography{KatoLimits.bib}

\end{document}